\documentclass[11pt]{amsart}

\usepackage{amsmath, amsthm, amssymb}

\usepackage{mathtools}
\usepackage{leftindex}
\usepackage{subcaption}
\usepackage{bm}

\usepackage[T1]{fontenc}                    
\usepackage{lmodern}
\usepackage{palatino}                       
\usepackage{inconsolata}                    
\usepackage{comment}

\usepackage[abs]{overpic}		  
\usepackage{import}
\usepackage{transparent}

\usepackage{xcolor}
\definecolor{lred}{RGB}{226, 106, 106}
\definecolor{nred}{RGB}{237, 28, 36}
\definecolor{ddred}{RGB}{255, 0, 0}
\definecolor{lblue}{RGB}{52, 152, 219}
\definecolor{nblue}{RGB}{0, 174, 239}
\definecolor{lyellow}{RGB}{232, 197, 91}
\definecolor{dgreen}{RGB}{0, 148, 68}
\definecolor{l1yellow}{RGB}{217, 224, 33}
\definecolor{l2yellow}{RGB}{216, 177, 64}
\definecolor{lgrey}{RGB}{179, 179, 179}
\definecolor{indigo}{rgb}{0.29, 0.0, 0.51}  
\usepackage[colorlinks, urlcolor=indigo, linkcolor=indigo, citecolor=indigo]{hyperref}

\usepackage[hcentering, vcentering, total={5.9in, 8.2in}]{geometry}  

\usepackage{tikz}
\usetikzlibrary{trees}
\usetikzlibrary{arrows}
\usepackage[all,cmtip]{xy}
 
\theoremstyle{plain}
\newtheorem{theorem}{Theorem}
\newtheorem{corollary}[theorem]{Corollary}
\newtheorem{proposition}[theorem]{Proposition}
\newtheorem{lemma}[theorem]{Lemma}
\newtheorem{question}[theorem]{Question}

\theoremstyle{definition}
\newtheorem{definition}[theorem]{Definition}

\theoremstyle{remark}
\newtheorem{remark}[theorem]{Remark}

\numberwithin{theorem}{section}

\usepackage{thmtools}
\usepackage{thm-restate}

\newcommand{\dfn}[1]{{\em #1}}        
\newcommand{\R}{\mathbb{R}}           
\newcommand{\ZZ}{\mathbb{Z}}           
\newcommand{\FF}{\mathbb{F}}           

\newcommand{\A}{\mathcal{A}}           
\newcommand{\C}{\mathcal{C}}           
\newcommand{\F}{\mathcal{F}}           
\newcommand{\HH}{\mathcal{H}}          
\newcommand{\OO}{\mathcal{O}}          
\newcommand{\SSS}{\mathcal{S}}          
\newcommand{\Y}{\mathcal{Y}}           
\newcommand{\Z}{\mathcal{Z}}           

\DeclareMathOperator{\interior}{int}  

\makeatletter
\newcommand*\bigcdot{\mathpalette\bigcdot@{0.6}}
\newcommand*\bigcdot@[2]{\mathbin{\vcenter{\hbox{\scalebox{#2}{$\m@th#1\bullet$}}}}}
\makeatother

\DeclareMathOperator\tb{tb}                               
\DeclareMathOperator\rot{rot}                             
\DeclareMathOperator\chat{\widehat {c}}                 
\DeclareMathOperator\Lhat{\widehat{\mathfrak{L}}}                 
\DeclareMathOperator\LOSS{\mathfrak{L}}                

\DeclareMathOperator{\HFhat}{\widehat{\mathit{HF}}}       
\DeclareMathOperator{\CFhat}{\widehat{\mathit{CF}}}       
\DeclareMathOperator{\HFK}{{\mathit{HFK}}}               
\DeclareMathOperator{\HFKhat}{\widehat{\mathit{HFK}}}     
\DeclareMathOperator{\CFK}{\mathit{CFK}}                 
\DeclareMathOperator{\CFKhat}{\widehat{\mathit{CFK}}}     
\DeclareMathOperator{\CFA}{\widehat{\mathit{CFA}}}        
\DeclareMathOperator{\CFD}{\widehat{\mathit{CFD}}}        

\DeclareMathOperator{\SFC}{\mathit{SFC}}       
\DeclareMathOperator{\SFH}{\mathit{SFH}}       
\newcommand{\EH}{\mathit{EH}}       

\newcommand{\DD}{\mathit{DD}}       
\newcommand{\DA}{\mathit{DA}}       
\newcommand{\AAA}{\mathit{AA}}       
\newcommand{\AD}{\mathit{AD}}       
    
\DeclareMathOperator{\BSA}{\widehat{\mathit{BSA}}}        
\DeclareMathOperator{\BSD}{\widehat{\mathit{BSD}}}        
\DeclareMathOperator{\BSAA}{\widehat{\mathit{BSAA}}}        
\DeclareMathOperator{\BSDD}{\widehat{\mathit{BSDD}}}        
\DeclareMathOperator{\BSDA}{\widehat{\mathit{BSDA}}}        
\DeclareMathOperator{\BSAD}{\widehat{\mathit{BSAD}}}        

\DeclareMathOperator\rk{rk}                     
\DeclareMathOperator{\AZ}{\mathsf{AZ}}          
\newcommand{\Aoo}{\mathcal{A}_{\infty}}         
\DeclareMathOperator\TW{\mathcal{TW}}           

\begin{document}

\title{On contact invariants in bordered Floer homology}

\author{Hyunki Min}
\author{Konstantinos Varvarezos}

\address{Department of Mathematics \\ University of California \\ Los Angeles, CA}
\email{hkmin27@math.ucla.edu}
\email{kkv@math.ucla.edu}


\begin{abstract}
In this paper, we define contact invariants in bordered sutured Floer homology. Given a contact 3-manifold with convex boundary, we apply a result of Zarev \cite{Zarev:joining} to derive contact invariants in the bordered sutured modules $\BSA$ and $\BSD$, as well as in bimodules $\BSAA,$ $\BSDD,$ $\BSDA$ in the case of two boundary components. In the connected boundary case, our invariants appear to agree with bordered contact invariants defined by Alishahi--F\"oldv\'ari--Hendricks--Licata--Petkova--V\'ertesi \cite{AFHLPV:borderedInvariants} whenever the latter are defined, although ours can be defined in broader contexts. We prove that these invariants satisfy a pairing theorem, which is a bordered extension of the Honda--Kazez--Mati\'c gluing map \cite{HKM:suturedInvariant} for sutured Floer homology.  We also show that there is a correspondence between certain $\Aoo$ operations in bordered type-$A$ modules and bypass attachment maps in sutured Floer homology. As an application, we characterize the Stipsicz--V\'ertesi map  from $\SFH$ to $\HFKhat$ as an $\Aoo$ action on $\CFA$. We also apply the immersed curve interpretation of Hanselman--Rasmussen--Watson \cite{HRW:immersed} to prove results involving contact surgery. 
\end{abstract}

\maketitle

\section{Introduction}\label{sec:intro}

The study of $3$-dimensional contact topology has been greatly advanced by Heegaard Floer-theoretic contact invariants. The first such invariant was introduced by Ozsav\'ath and Szab\'o \cite{OS:contactInvariant} for closed contact $3$-manifolds, inspired by Kronheimer and Mrowka's contact invariant in Monopole Floer homology \cite{KM:invariants}. This invariant provided various applications, such as classification of tight contact structures \cite{GLS:classification, Tosun:classification} and fillability of contact structures \cite{Ghiggini:fillability, LS:surgery}. Later, Honda, Kazez and Mati\'c \cite{HKM:suturedInvariant} introduced a version of the invariant for sutured contact $3$-manifolds. More recently, Alishahi, F\"oldv\'ari, Hendricks, Licata, Petkova and V\'ertesi \cite{AFHLPV:borderedInvariants, ALPV:friendly} defined contact invariants for multipointed bordered manifolds.  

In this paper, we define invariants of a contact structure on a bordered sutured manifold. These invariants are cycles in bordered sutured (bi)modules, defined by Zarev \cite{Zarev:borderedSutured}. We prove that these are well-defined invariants and satisfy a pairing theorem. Then we show that the $\Aoo$ operation $m_2$ in bordered Floer homology is in fact a bypass attachment map, which provides a new connection between contact topology and Floer theory. Lastly, we provide some examples and applications involving contact surgery.

\subsection{Bordered contact invariants}\label{subsec:intro1}
Let $\Y = (Y,\Gamma,\F_1 \sqcup \F_2)$ be a bordered sutured manifold where $\F_1$ and $\F_2$ are parametrized by arc diagrams $\Z_1$ and $\Z_2$ of rank $n_1$ and $n_2$, respectively. To each $I \subset \{1,\ldots,n_1\}$ and $J \subset \{1,\ldots,n_2\}$, we can associate \emph{elementary dividing sets} $\Gamma_I$ and $\Gamma_J$, respectively. See Section~\ref{subsec:elementary} for more details. By adding these dividing sets onto $\F_1$ and $\F_2$, we can obtain a sutured manifold $C(\Y) = (Y, \Gamma_I \cup \Gamma \cup \Gamma_J)$. If $\xi$ is a contact structure on $C(\Y)$, then we say $\xi$ is \emph{compatible with $\Y$}, and call $(\Y,\xi)$ a \emph{bordered sutured contact manifold}.

To define various flavors of contact invariants, we need to tweak the manifold by attaching a \emph{twisting slice} of a sutured surface $\F=(F,\Lambda)$: 
\[
  \TW_{\mathcal{F}}^{\pm} = (F\times [0,1], \Gamma_{\pm}, -\mathcal{F}\cup -\overline{\mathcal{F}}).
\]
Here, we identify $-\mathcal{F}$ with $F \times \{0\}$ and $-\overline{\mathcal{F}}$ with $F\times \{1\}$. The dividing set $\Gamma_{\pm}$ is obtained from $\Lambda \times [0,1]$ by applying $1/n$-th of a positive (\dfn{resp.} negative) Dehn twist along each component of $\partial F \times \{1/2\}$. See Figure~\ref{fig:twisting} for an example. Now we are ready to state our first main theorem. 

\begin{theorem}\label{thm:main}
  Let $\Y = (Y,\Gamma,\F_1 \sqcup \F_2)$ be a bordered sutured manifold and $\xi$ a contact structure compatible with $\Y$. Also, suppose $\TW_1^+$ and $\TW_2^+$ are positive twisting slices for $\F_1$ and $\F_2$, respectively. 
  Then there exist contact invariants of $\xi$ in bordered bimodules:   
  \begin{alignat*}{3}
    &c_{\AAA}(\xi) \in \BSAA(-\Y),\quad
    &&c_{\AD}(\xi) \in \BSAD(-\Y \cup \TW_2^+),\\
    &c_{\DA}(\xi) \in \BSDA(\TW_1^+ \cup -\Y),\quad
    &&c_{\DD}(\xi) \in \BSDD(\TW_1^+ \cup -\Y \cup \TW_2^+)
  \end{alignat*}
  If $\Y = (Y,\Gamma,\F)$, there exist contact invariants of $\xi$ in bordered modules: 
  \begin{align*}
    c_A(\xi) \in \BSA(-\Y) \quad \text{and}\quad
    c_D(\xi) \in \BSD(\TW^+_{\F} \cup -\Y)  
  \end{align*} 
\end{theorem}

Let $\Y_1 = (Y_1,\Gamma_1,\mathcal{F})$ and $\Y_2  = (Y_2,\Gamma_2,-\mathcal{F})$ be bordered sutured manifolds. Suppose $\xi_1$ and $\xi_2$ are contact structures compatible with $\Y_1$ and $\Y_2$, respectively. We can \emph{glue} two sutured manifolds $C(\Y_1)$ and $C(\Y_2)$ together along $\F$ by inserting a negative twisting slice between $\Y_1$ and $\Y_2$:  
\[
  C(\Y_1) \Cup_{\F} C(\Y_2) := \Y_1 \cup \TW^-_{\F} \cup \Y_2.
\]
Here we use the symbol $\Cup$ to distinguish the gluing operation from an ordinary union. There is also a well-defined glued contact structure $\xi_1\Cup_{\F} \xi_2$ on $C(\Y_1) \Cup_{\F} C(\Y_2)$. See Section~\ref{subsec:gluing} for more details. Our second main theorem is about pairing two contact invariants; in particular, we recover the HKM contact invariant of the gluing.

\begin{theorem}\label{thm:pairing}
  Let $\Y_1 = (Y_1,\Gamma_1,\mathcal{F}_1,\mathcal{F})$ and $\Y_2  = (Y_2,\Gamma_2,-\mathcal{F},\mathcal{F}_2)$ be bordered sutured manifolds. Also, suppose $\xi_1$ and $\xi_2$ are contact structures compatible with $\Y_1$ and $\Y_2$, respectively. Then we can pair two contact invariants of $\xi_1$ and $\xi_2$, and obtain a contact invariant of $\xi_1 \Cup_{\F} \xi_2$.
  \begin{align*}
    c_{\AAA}(\xi_1) \boxtimes c_{\DA}(\xi_2) = c_{\AAA}(\xi_1 \Cup_{\mathcal{F}} \xi_2) \in \BSAA(-\Y_1 \cup \TW^+_{\F} \cup -\Y_2)
  \end{align*}
  Any combination of bimodules for $\Y_1$ and $\Y_2$ can be used, where one is type–$A$ for $\A(\Z)$, and the other is type–$D$ for $A(\Z)$. If $\F_1 = \F_2 = \varnothing$, then we have
  \begin{align*}
    c_{A}(\xi_1) \boxtimes c_{D}(\xi_2) = c(\xi_1 \Cup_{\mathcal{F}} \xi_2) \in \SFC(-\Y_1 \cup \TW^+_{\F} \cup -\Y_2).
  \end{align*} 
  where $c(\xi)$ is a contact class in sutured Floer complex such that $[c(\xi)] = \EH(\xi)$.
\end{theorem}

Let $\Y = (Y,\Gamma_D, \F)$ be a bordered sutured manifold where $\Gamma_D$ is a dividing set with a single suture on a disk $D$, called a \emph{disk suture}. In this case, we can consider $D$ as a basepoint and $\Y$ as a bordered manifold with a single basepoint. Zarev \cite{Zarev:borderedSutured} showed that
\[
  \BSA(\Y) \cong \CFA(\Y), \quad \BSD(\Y) \cong \CFD(\Y).
\]
Under this identification, we can define contact invariants in bordered Floer homology, and recover the Ozsv\'ath--Szab\'o contact invariant via pairing.

\begin{corollary}\label{cor:main} Let $\mathcal{Y} = (Y, \mathcal{F})$ be a bordered $3$-manifold. Also, suppose $\xi$ is a contact structure compatible with $\Y$. Then there exist bordered contact invariants 
\[
  c_A(\xi) \in \CFA(-\Y)\quad\text{and}\quad c_D(\xi) \in \CFD(\TW^+_{\F} \cup -\Y).
\] 
Also, let $\Y_1 = (Y_1,\mathcal{F})$ and $\Y_2  = (Y_2,-\mathcal{F})$ be two bordered manifolds. Suppose $\xi_1$ and $\xi_2$ are contact structures compatible with $\Y_1$ and $\Y_2$, respectively. Then we have
\[
  c_A(\xi_1) \boxtimes c_D(\xi_2) = c(\xi_1 \Cup_{\F} \xi_2) \in \CFhat(-\Y_1 \cup \TW_{\F}^+ \cup -\Y_2)
\]
where $c(\xi)$ is a contact class in the hat version of Heegaard Floer complex such that $[c(\xi)] = \widehat{c}(\xi)$.
\end{corollary}

\begin{remark}
  Alishahi, F\"oldv\'ari, Hendricks, Licata, Petkova, and V\'ertesi \cite{AFHLPV:borderedInvariants} also defined contact invariants in bordered sutured Floer homology, although their construction necessarily works only with multiple disk sutures and hence can be thought of as defining invariants in multipointed bordered Floer homology $\widetilde{CFA}(-\Y)$ and $\widetilde{CFD}(-\Y)$. Moreover, the invariants they defined appear to be a special case of the ones described here; see Remark \ref{rmk:friendly} below.
\end{remark}

\subsection{\texorpdfstring{$\Aoo$}{Aoo} operations and bypass attachments}\label{subsec:intro2}
Let $\Sigma$ be a compact oriented connected surface with boundary which is parametrized by an arc diagram $\Z$ of rank $n$. Mathews \cite{Mathews:contactCategory} and Zarev \cite{Zarev:joining} (see also \cite{Cooper:contactCategory}) showed that there exist the following natrual isomorphisms:
\begin{align} \label{eq:iso}
  CA(\Sigma,\Z) \cong H_*(\A(\Z)) \cong H_*(\BSAA(\TW_{\F}^-)) \cong \bigoplus_{I,J \subset\{1,\ldots,n\}} \SFH(-\Sigma\times[0,1],-\Gamma_{I\to J})   
\end{align}
where $CA(\Sigma, \Z)$ is a \emph{contact category algebra}, an algebra over $\mathbb{Z}_2$ generated by isotopy classes of tight contact structures on $(\Sigma \times [0,1],\Gamma_{I\to J})$ for $I ,J \subset \{1,\ldots,n\}$. Here, $\Gamma_{I\to J}$ is a dividing set consisting of $-\Gamma_I$ on $\Sigma\times\{0\}$, $\Gamma_J$ on $\Sigma\times\{1\}$ and vertical dividing curves on $\partial\Sigma \times [0,1]$. This algebra is equipped with a formal addition and the multiplication is stacking two contact structures. To establish the first isomorphism in (\ref{eq:iso}), Mathews constructed a contact structure $\xi_a$ on $(\Sigma \times [0,1],\Gamma_{I\to J})$ for each generator $a = a_i(S,T,\phi) \in \A(\Z)$ and showed that if $a$ is not a cycle, $\xi_a$ is overtwisted. 

Now we establish the equivalence between bypass attachment maps and $\Aoo$ operations. 

\begin{restatable}{theorem}{bypass}\label{thm:bypass}
  Let $\Y = (Y,\Gamma,\F)$ be a bordered sutured manifold where $\F$ is parametrized by an arc diagram $\Z$ and $\xi$ a contact structure compatible with $\Y$. Consider a cycle $a = a_i(S,T,\phi) \in \A(\Z)$, and the corresponding contact structure $\xi_a$ on $(\Sigma\times [0,1],\Gamma_{I\to J})$. Then we have  
  \[
    m_{2}(c_{A}(\xi), a) = c_{A}(\xi \Cup \xi_a).
  \] 
  A similar statement also holds for $(\BSAA, m_{1|1|0}, m_{0|1|1})$, $(\BSDA,m_{0|1|1})$ and $(\BSAD,m_{1|1|0})$.
\end{restatable}

According to Mathews \cite{Mathews:contactCategory}, strands in $a_i = a_i(S,T,\phi)$ correspond to bypasses in $\xi_a$ and $\xi_a$ can be obtained by attaching theses bypasses to $\Sigma$ (more generally, any contact structure on $\Sigma \times [0,1]$ can be obtained by a sequence of bypass attachments, see \cite{Giroux:convex, HH:convex}). Therefore, $m_{2}(\cdot, a)$ can be considered as a bordered (possibly a sequence of) bypass attachment map.

In \cite{Mathews:Ainifinity}, Mathews applied the construction of Kadeishvili \cite{Kadeishvili:Aoo} to equip an $\Aoo$ structure on $H_*(\A(\Z))$, and provided a contact geometric interpretation of the higher $\Aoo$ operations $\mu_n$. We can ask the same question for general bordered modules.

\begin{question}
  Is there a contact geometric interpretation for higher $\Aoo$ operations
  \[
    m_n\colon \BSA(-\Y) \otimes \A(\Z)^{\otimes (n-1)} \to \BSA(-\Y)
  \]
  for $n > 2$?  
\end{question}

\subsection{Torus boundary}\label{subsec:intro3}
For 3-manifolds with torus boundary, we can provide an explicit description of bypass attachment maps. Let $\Y = (Y, \Gamma_D, \F)$ be a bordered sutured $3$-manifold where $\Gamma_D$ is a disk suture and $\F = (T^2_{\bullet},\Lambda)$ is a punctured torus parametrized by an arc diagram $\Z$ shown in Figure~\ref{fig:parametrizedTorus}.

For $I \subset \{0,1\}$, there exists an elementary dividing set $\Gamma_I$ and we obtain a sutured manifold $C(Y) = (Y,\Gamma_D \cup \Gamma_I)$. When $I = \varnothing$ and $\{0,1\}$,  however, the dividing set $\Gamma_D \cup \Gamma_I$ contains a contractible dividing curve and any contact structure on $C(Y)$ is overtwisted. Thus we only need consider $I = \{0\}$ and $\{1\}$. We will denote 
\[
  \Gamma_0 := \Gamma_D \cup \Gamma_{\{0\}}, \quad \Gamma_1 := \Gamma_D \cup \Gamma_{\{1\}}.
\]
For $i=0,1$, the dividing set $\Gamma_i$ consists of two closed curves parallel to the parametrizing arc $\alpha_i$. 

According to Theorem~\ref{thm:bypass}, $m_2(\cdot,a)$ is a bypass attachment map corresponding to the contact structure $\xi_a$ on $(T_{\bullet} \times [0,1], \Gamma_{I\to J})$. Since tight contact structures on $T^2 \times I$ with convex boundary having two dividing curves on each component were thoroughly studied by Honda \cite{Honda:classification1,Honda:classification2}, we can explicitly describe the contact structures $\xi_a$.

\begin{theorem}\label{thm:contactTorusAlg} 
  Let $\Z$ be an arc diagram shown in Figure~\ref{fig:parametrizedTorus}. Then each generator of $\A(\Z,0)$ corresponds to a tight contact structure on $T^2 \times [0,1]$ as follows:

  \begin{itemize}
    \item $\iota_0$: an $I$-invariant neighborhood of $(T^2, \Gamma_{0})$. 
    \item $\iota_1$: an $I$-invariant neighborhood of $(T^2, \Gamma_{1})$. 
    \item $\rho_1$: a positive basic slice from $\Gamma_{0}$ to $\Gamma_{1}$.
    \item $\rho_2$: a negative basic slice from $\Gamma_{1}$ to $\Gamma_{0}$.
    \item $\rho_3$: a negative basic slice from $\Gamma_{0}$ to $\Gamma_{1}$.
    \item $\rho_{12}$: a union of two basic slices corresponding to $\rho_1$ and $\rho_2$.
    \item $\rho_{23}$: a union of two basic slices corresponding to $\rho_2$ and $\rho_3$.
    \item $\rho_{123}$: a union of three basic slices corresponding to $\rho_1$, $\rho_2$ and $\rho_3$.
  \end{itemize}
\end{theorem}

\begin{remark}
  Notice that the positive basic slice from $\Gamma_1$ to $\Gamma_0$ does not correspond to any algebra element. This is due to the fact that the contact structure actually lives on $T^2_{\bullet}\times [0,1]$ instead of the thickened torus.
  In \cite[Chapter 11]{LOT:bordered}, a generalized torus algebra, which includes $\rho_0$ as a generator, is described. We suspect $\rho_0$ is a candidate for this ``missing'' basic slice, under a suitable extension of the bordered modules.
\end{remark}

Let $K$ be a null-homologous Legendrian knot with $\tb(K) = n$ in a closed contact 3-manifold $(Y,\xi)$. We denote by $(Y(K), \Gamma_n, \xi_n)$ the complement of a standard neighborhood of $K$. 
Consider a negative basic slice $(T^2 \times [0,1], \Gamma_{n\to\mu}, \xi_{-})$ where $\Gamma_{\mu}$ consists of two parallel meridional sutures. Stipsicz and V\'ertesi \cite{SV:LOSS} considered the following HKM gluing map   
\begin{align*}
  \SFH(-Y(K), -\Gamma_n) &\to \SFH(-Y(K), -\Gamma_{\mu})\\
  x &\mapsto \Phi(x, \EH(\xi_{-})) 
\end{align*}
According to Juh\'asz \cite{Juhasz:sutured}, there is a canonical identification of $\SFH(-Y(K), -\Gamma_{\mu})$ and $\HFKhat(-Y,K)$. Under this identification, we obtain the \emph{Stipsicz--V\'ertesi map} 
\begin{align*}
  \Phi_{\mathit{SV}}\colon \SFH(-Y(K), -\Gamma_n) \to \HFKhat(-Y,K)
\end{align*}
which maps $\EH(\xi_n) \mapsto \Lhat(K).$
Using Theorem~\ref{thm:bypass}, \ref{thm:contactTorusAlg} and \ref{thm:CFA=SFC}, we can characterize the Stipsicz--V\'ertesi map on bordered Floer homology as follows.

\begin{corollary}\label{cor:StipsiczVertesi}
  Let $\Y$ be a knot complement such that the boundary is parametrized by an $\alpha$-type arc diagram as shown in Figure~\ref{fig:parametrizedTorus.c}. Suppose that $\alpha_0$ is isotopic to a meridian and $\alpha_1$ is isotopic to a preferred longitude. Then the Stipsicz--V\'ertesi map is equivalent to a $\rho_2$ action on $\CFA$, i.e.,  
  \[
    \Phi_{\mathit{SV}}([x]) =  [m_2(x, \rho_2)].
  \]
  If $\alpha_0$ is isotopic to a longitude and $\alpha_1$ is isotopic to a meridian, then the Stipsicz--V\'ertesi map is equivalent to a $\rho_3$ action on $\CFA$.
\end{corollary}

\subsection{Applications on contact surgery}\label{subsec:intro4}
Various tools in Heegaard Floer homology, such as the surgery exact triangle and mapping cone formula have previously been used to prove many interesting properties about contact surgery. For instance, Oszv\'ath and Szab\'o \cite{OS:contactInvariant} showed that a negative contact surgery preserves a non-vanishing contact invariant. Following that, there have been several studies on the behavior of contact invariants under positive contact surgery. Lisca and Stipsicz \cite{LS:surgery} showed that any positive contact surgery on a Legendrian representative of an algebraic knot with maximal Thurston--Bennequin invariant preserves a non-vanishing contact invariant. Hedden and Plamenevskaya \cite{HP:surgery} proved a similar result for fibered knots that support $(S^3,\xi_{std})$ (see also \cite{Conway:surgery}). Eventually, Golla \cite{Golla:surgery} and Mark--Tosun \cite{MT:surgery} completely characterized the behavior of a contact invariant under positive contact surgery on any Legendrian knot in $(S^3,\xi_{std})$.  

The bordered contact invariants we have defined are, in the case of a manifold with torus boundary, amenable to the immersed curve techniques of Hanselman, Rasmussen, and Watson \cite{HRW:immersed}, which have proved to be convenient for performing bordered calculations. Using them, we can extend the results of Mark--Tosun and 
Golla using a more concise argument.

\begin{restatable}{theorem}{surgery} \label{thm:surgery}
  Let $(Y,\xi)$ be a contact L-space with $\chat(\xi) \neq 0$ and $K$ a null-homologous Legendrian knot in $Y$. Denote by $\xi_{(r)}$ the contact structure obtained by a positive contact $(r)$-surgery on $K$, in which all stabilizations are chosen to be negative. Let $s = r + \tb(K)$ be the corresponding smooth surgery coefficient.
  \begin{itemize}
    \item If $\tb(K) - \rot(K) < 2\tau_{\xi}(K) - 1$, then $\chat(\xi_{(r)}) = 0$,
    \item if $\tb(K) - \rot(K) = 2\tau_{\xi}(K) - 1$ and
    \begin{itemize}
      \item if $\epsilon_{\xi}(K) = 0,1$, then $\chat(\xi_{(r)})\neq 0$ if  $s \geq 2\tau_{\xi}(K)$ and $\chat(\xi_{(r)}) = 0$ if $s \leq 2\tau_{\xi}(K) -1$,
      \item If $\epsilon_{\xi}(K) = -1$, then $\chat(\xi_{(r)})=0$ 
    \end{itemize}
  \end{itemize}
  where $\tau_{\xi}$ is the \emph{contact tau-invariant} and $\epsilon_{\xi}$ is the \emph{contact epsilon-invariant}.
\end{restatable}

See Section~\ref{subsec:knots} for the definitions of $\tau_{\xi}$ and $\epsilon_{\xi}$.
 
\begin{remark}
  The statement of Theorem~\ref{thm:surgery} for the case of knots in $S^3$ is equivalent to the result of Golla \cite{Golla:surgery}, but slightly weaker than the result of Mark--Tosun \cite{MT:surgery}. Theorem~\ref{thm:surgery} does not consider the case when $\epsilon_{\xi}(K) = 0,1$ and $2\tau_{\xi}(K)-1 < s < 2\tau_{\xi}(K)$. 
\end{remark}

Wand \cite{Wand:surgery} showed that Legendrian surgery preserves tightness. However, the effects of Legendrian surgery on knots in overtwisted contact manifolds have not been well studied. We investigate Legendrian surgery on certain \emph{strongly non-loose Legendrian knots}, of which the complement does not contain a boundary-parallel half Giroux torison (Legendrian surgery on knots that are not strongly non-loose immediately produces an overtwisted contact structure). It is well-known that Legendrian surgery on any strongly non-loose unknot results in a tight contact structure. Also, it follows easily from \cite{EMM:nonloose} that Legendrian surgery on any strongly non-loose positive torus knot $T_{2,2n+1}$ produces a tight contact structure. We extend this result by showing that Legendrian surgery on any strongly non-loose $T_{3,4}$ produces a tight contact structure. This exemplifies certain kinds of computations which will play an important role in classifying tight contact structures on surgeries on torus knots \cite{EMTV:torus}.

\begin{restatable}{theorem}{ptorus}\label{thm:T34}
  Let $K$ be a strongly non-loose Legednrian $T_{3,4}$ in an overtwisted contact 3-sphere. Then Legendrian surgery on $K$ yields a tight contact structure with a non-vanishing contact invariant.
\end{restatable}

Matkovi\v{c} \cite{Matkovic:nonloose} showed that Legendrian surgery on any strongly non-loose negative torus knot produces a tight contact structure. We give a concise alternative proof of this result for the left-handed trefoil.

\begin{restatable}{theorem}{lht} \label{thm:lht}
  Let $K$ be a strongly non-loose Legendrian left-handed trefoil in an overtwisted contact 3-sphere. Then Legendrian surgery on $K$ yields a tight contact structure with a non-vanishing contact invariant.
\end{restatable}
   
\noindent
{\bf Acknowledgements.} 
The authors appreciate Ko Honda and Sucharit Sarkar for helpful discussions.

\section{Preliminaries}

We gather the relevant facts about contact topology and Heegaard Floer homology for the later sections. We assume readers have a basic understanding of 3–dimensional contact topology, in particular convex surface theory and open book decompositions. We also assume readers are familiar with Heegaard Floer homology. We recommend that readers refer to \cite{Etnyre:openbook, Geiges:book, LOT:bordered, OS:HF2, OS:HF1} for background.

\subsection{Gluing sutured manifolds} \label{subsec:gluing}
We first recall the definitions of sutured $3$-manifolds and sutured surfaces with a dividing set.

\begin{definition}
  A pair $(Y,\Gamma)$ is a \emph{sutured 3-manifold} if   
  \begin{itemize}
    \item $Y$ is a compact oriented $3$-manifold without a closed component, 
    \item $\Gamma$ is a collection of oriented simple closed curves in $\partial Y$, called \emph{sutures},
    \item $\partial Y$ is divided by $\Gamma$ into two regions $R_+(\Gamma)$ and $R_-(\Gamma)$ such that $\partial R_{\pm}(\Gamma) = \pm\Gamma$,
    \item $R_+(\Gamma)$ and $R_-(\Gamma)$ have no closed components,
    \item $Y$ is \emph{balanced} if $\chi(R_+(\Gamma)) = \chi(R_-(\Gamma))$
  \end{itemize}
\end{definition}

\begin{definition}
  A pair $\F = (F,\Lambda)$ is a \emph{sutured surface} if 
  \begin{itemize}
    \item $F$ is a compact oriented surface without a closed component,
    \item $\Lambda \subset \partial F$, called \emph{sutures}, is a finite collection of points with sign,
    \item $\partial F$ is divided by $\Lambda$ into two regions $S_+(\Lambda)$ and $S_-(\Lambda)$ such that $\partial S_{\pm}(\Lambda) = \pm\Lambda$,
    \item $S_+(\Lambda)$ and $S_-(\Lambda)$ have no closed components.
  \end{itemize} 
\end{definition}

\begin{definition}
  Let $\F = (F, \Lambda)$ be a sutured surface. A \emph{dividing set} $\Gamma$ for $\F$ is a finite collection of properly embedded arcs and simple closed curves in $F$ such that
  \begin{itemize}
    \item $\partial \Gamma = -\Lambda$ as an oriented boundary,
    \item $F$ is divided by $\Gamma$ into two regions $R_+(\Gamma)$ and $R_-(\Gamma)$ with $\partial R_{\pm}(\Gamma) = \pm\Gamma \cup S_{\pm}(\Lambda)$.
  \end{itemize}
\end{definition}

Given a sutured surface $\F = (F,\Lambda)$, we can construct two sutured surfaces $-\F = (-F, -\Lambda)$ and $\overline{\F} = (-F, \Lambda)$. Both $-\F$ and $\overline{\F}$ reverse the orientation of the surface, but the difference is that the roles of $S_+$ and $S_-$ are reversed in $\overline{\F}$. 

To glue two sutured $3$-manifolds, we first need to \dfn{sharpen} the boundaries of the sutured $3$-manifolds by turning them into manifolds with corners. Let $(Y,\Gamma)$ be a sutured $3$-manifold and $F \subset \partial Y$ a subsurface such that all components of $\partial F$ intersect $\Gamma$ transversely and nontrivially. We may \emph{sharpen} $\partial Y$ along $\partial F$ into a corner so that the sutures on $F$ and $\partial Y \setminus F$ interleave along the corner and the sutures on $F$ are to the right of the sutures on $\partial Y \setminus F$. See Figure~\ref{fig:sharpening} for an example. We can also \dfn{round the edges} by reversing the sharpening operation.

\begin{figure}[htbp]{\scriptsize
  \begin{overpic}[tics=20]{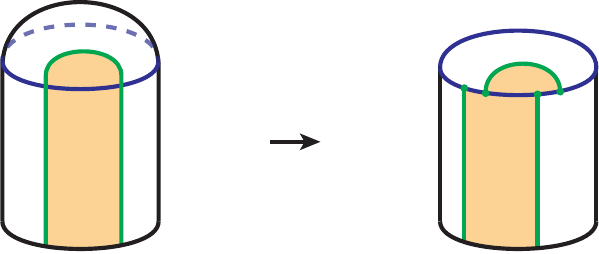}
  \end{overpic}}
  \caption{Sharpening the boundary of a sutured manifold}
  \label{fig:sharpening}
\end{figure}

Now let $(Y_i,\Gamma_i)$ for $i=1,2$ be sutured $3$-manifolds. Suppose there are subsurfaces $(F, \Gamma) \subset (\partial Y_1, \Gamma_1)$ and $(-F, -\Gamma) \subset (\partial Y_2, \Gamma_2)$. Then the \dfn{gluing} of $(Y_1,\Gamma_1)$ and $(Y_2,\Gamma_2)$ along $F$ is defined as follows: first sharpen $\partial Y_i$ along $\partial F$ for $i=1,2$. Then glue two sharpened manifolds by identifying $(F,\Gamma)$ and $(-F, -\Gamma)$; we denote the result using $\Cup$ by 
\[
  (Y_1,\Gamma_1) \Cup (Y_2,\Gamma_2) = (Y_1 \cup_F Y_2, \Gamma_1 \cup_{\Gamma} \Gamma_2).
\]
Note that the gluing operation introduces some twisting on the sutures since we sharpen the boundaries before gluing. See Figure~\ref{fig:gluing} for an example.

\begin{figure}[htbp]{\scriptsize
  \vspace{0.2cm}
  \begin{overpic}[tics=20]{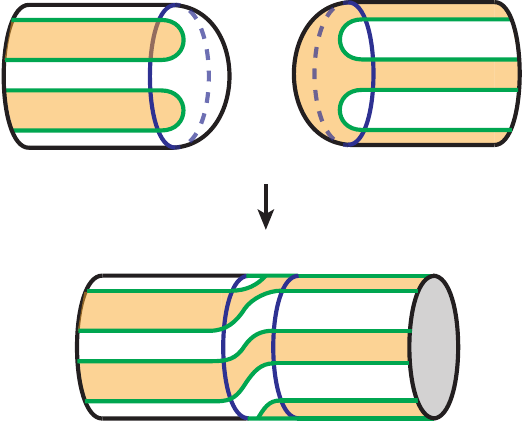}
  \end{overpic}}
  \caption{gluing two sutured manifolds}
  \label{fig:gluing}
\end{figure}

In \cite{Zarev:joining}, Zarev introduced the \dfn{joining} operation on two sutured manifolds, which is a generalization of the gluing operation. Here we define the gluing operation as a special case of the joining operation.

\begin{definition}
  A \emph{partially sutured manifold} is a triple $\Y = (Y,\Gamma,\F)$ such that
  \begin{itemize}
    \item $Y$ is a compact oriented $3$-manifold with corners,
    \item $\F = (F,\Lambda)$ is a sutured surface such that $F \subset \partial Y$ and the corner of $Y$ is $\partial F$.
    \item $\Gamma$ is a dividing set for $(\partial Y \setminus F, -\Lambda)$.
  \end{itemize}
\end{definition}

Let $\Y=(Y,\Gamma,\F_1\sqcup\F_2)$ and $\Y'=(Y',\Gamma',-\F_2\sqcup\F_3)$ be two partially sutured manifolds. We can \emph{pair} $\Y$ and $\Y'$ along $\F_2$ and $-\F_2$ and obtain
\[
  \Y \cup \Y' = (Y \cup Y', \Gamma \cup \Gamma',\F_1 \sqcup \F_3).
\]
We use the term \emph{pairing} to distinguish from the \emph{gluing} operation of two sutured manifolds defined previously. We now recall some important partially sutured manifolds.

\begin{definition} \label{def:cap}
  Let $\mathcal{F} = (F,\Lambda)$ be a sutured surface and $\Gamma$ a dividing set for $\F$. The \dfn{cap for $\mathcal{F}$ associated with $\Gamma$} is a partially sutured manifold defined to be 
  \[
    \mathcal{C} = (C, \Gamma\times\{1\}, (-F\times\{0\}, -\Lambda \times \{0\})),
  \]
  where $C = F \times [0,1]/\sim$, where $(p,t)\sim (p,t')$ for $p\in \partial F$ and $t,t' \in [0,1]$. 
\end{definition}

As mentioned earlier, gluing two sutured manifolds results in a twisting of sutures, while pairing does not. Therefore, we need a partially sutured manifold that adds a twisting. 

\begin{definition}
  A \dfn{positive (resp. negative) twisting slice of a sutured surface $\mathcal{F} = (F,\Lambda)$} is a partially sutured manifold 
  \[
    \TW_{\mathcal{F}}^{\pm} = (F\times [0,1], \Gamma_{\pm}, -\mathcal{F}\cup -\overline{\mathcal{F}})
  \]
  where we identify $-\mathcal{F}$ with $F \times \{0\}$ and $-\overline{\mathcal{F}}$ with $F\times \{1\}$. The dividing set $\Gamma$ is obtained from $\Lambda \times [0,1]$ by applying $1/n$-th of a positive (\dfn{resp.} negative) Dehn twist along each component of $\partial F \times \{1/2\}$, containing $n$ points of $\Lambda$.
\end{definition}

See Figure~\ref{fig:twisting} for examples of twisting slices.

\begin{figure}[htbp]{\scriptsize
  \begin{overpic}[tics=20]{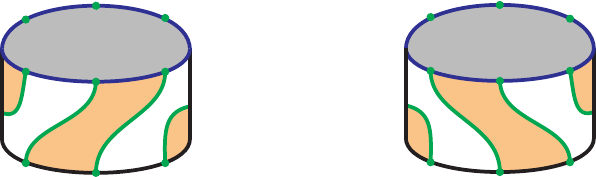}
  \end{overpic}}
  \caption{Positive and negative twisting slices}
  \label{fig:twisting}
\end{figure}

Now we are ready to define the gluing operation via pairing. Let $(Y_i,\Gamma_i)$ for $i=1,2$ be sutured $3$-manifolds. Suppose there are subsurfaces $(F, \Gamma) \subset (\partial Y_1, \Gamma_1)$ and $(-F, -\Gamma) \subset (\partial Y_2, \Gamma_2)$. Let $\F = (F,\Lambda)$ and $\Lambda = -(\partial F \cap \Gamma)$. Then the caps $\C_1$ for $\F$ associated with $\Gamma$ embeds in $(Y_1,\Gamma_1)$. Similarly, the cap $\C_2$ for $-\F$ associated with $-\Gamma$ embeds in $(Y_2,\Gamma_2)$. Then we can glue $(Y_1,\Gamma_1)$ and $(Y_2,\Gamma_2)$ together along $F$ as follows:
\[
  (Y_1,\Gamma_1) \Cup_{F} (Y_2,\Gamma_2) = (Y_1 \setminus \mathcal{C}_1) \cup \TW_{\F}^- \cup (Y_2 \setminus \mathcal{C}_2).
\]

\subsection{The contact invariant in sutured Floer homology}\label{subsec:invariant}
In this subsection, we review the invariant of a contact structure on a sutured $3$-manifold defined by Honda, Kazez, and Mati\'c \cite{HKM:suturedInvariant}. 

Let $Y$ be a compact oriented $3$-manifold with boundary and $\xi$ a contact structure on $Y$ that induces a convex boundary with a dividing set $\Gamma$. If $Y$ has corners, we further assume that the corners are Legendrian. We say $(Y,\Gamma,\xi)$ is a \emph{sutured contact manifold} in short. 

The construction of the contact invariant of a sutured contact $3$-manifold $(Y,\Gamma,\xi)$ starts with a \emph{partial open book decomposition}. 

\begin{definition}
  A \emph{partial open book decomposition} is a triple $(S,P,h)$ where 
  \begin{enumerate}
    \item $S$ is a compact oriented surface with boundary,
    \item $\displaystyle P = \bigcup P_i$ is the union of $1$-handles for $i=1,\ldots,n$ such that $S$ is obtained from $\overline{S\setminus P}$ by attaching $1$-handles $P_i$, 
    \item $h\colon P \to S$ is an embedding that is the identity along $\partial P \cap \partial S$.
  \end{enumerate}
\end{definition}

We can construct a balanced sutured $3$-manifold $(Y,\Gamma)$ from a partial open book $(S,P,h)$ as follows. Let 
\begin{align*}
  H &= S \times [-1,0] / \sim, \\
  N &= P \times [0,1] / \sim
\end{align*}
be two handlebodies where $(x,t) \sim (x,t')$ for $x \in \partial S$ and $t,t' \in [-1,0]$, and $(y,t) \sim (y,t')$ for $y \in \partial S \cap \partial P$ and $t,t' \in [0,1]$, respectively. Then we define the $3$-manifold $Y$ to be
\begin{align*} 
  Y = H \cup N / \sim 
\end{align*}
where we identify $P\times \{0\} \subset \partial H$ with $P\times \{0\} \subset \partial N$, and $h(P) \times \{-1\} \subset \partial H$ with $P \times \{1\} \subset \partial N$. We also define the suture $\Gamma$ to be
\[
  \Gamma = \overline{\partial S \setminus \partial P} \times \{0\} \cup -\overline{\partial P \setminus \partial S} \times \{1/2\}. 
\]

Now we define a contact structure on $(Y,\Gamma)$ that is compatible with the partial open book $(S,P,h)$ as follows.

\begin{definition}
  A partial open book $(S,P,h)$ \emph{supports} a contact structure $\xi$ on $(Y,\Gamma)$ if   

  \begin{enumerate}
    \item $(H,\xi|_H)$ and $(N,\xi|_N)$ are tight.
    \item $\partial H$ is a convex surface with dividing set $\partial S \times \{0\}$
    \item $\partial N$ is a convex surface with dividing set $\partial P \times \{1/2\}$.
  \end{enumerate}
\end{definition}

Honda, Kazez and Mati\'c \cite{HKM:suturedInvariant} showed that the relative Giroux correspondence; see also Etg\"{u} and \"{O}zba\u{g}c\i \cite{EO:GriouxCorrespondence, EO:pobd}.

\begin{theorem}[Honda--Kazez--Mati\'c \cite{HKM:suturedInvariant}]
  For a sutured contact $3$-manifold $(Y,\Gamma,\xi)$, there exists a partial open book decomposition $(S,P,h)$ of $(Y,\Gamma)$ that supports $\xi$. Also, there exists a unique contact structure on $(Y,\Gamma)$ supported by $(S,P,h)$ up to isotopy.  
\end{theorem}

From a partial open book decomposition $\OO = (S,P,h)$ of $(Y,\Gamma)$, we can construct a Heegaard diagram $\HH_{\OO}$ of $(-Y,-\Gamma)$ as follows. First, we define the Heegaard surface $\Sigma \subset Y$ to be 
\[
  \Sigma := -S \times \{-1\} \cup P \times \{0\} / \sim
\]
where $(\partial S \cap \partial P) \times \{-1\}$ is identified with $(\partial S \cap \partial P) \times \{0\}$. Let $(a_1, \ldots, a_n)$ be the set of co-cores of $P_i \times \{0\}$ for $i = 1,\ldots,n$, and $(b_1,\ldots,b_n)$ the set of properly embedded arcs on $P \times \{0\}$ where $b_i$ is obtained from $a_i$ by a small perturbation such that   
\begin{enumerate}
  \item the endpoints of $a_i$ are isotoped along $\partial S$ in the direction given by the boundary orientation of $S$.
  \item $a_i$ intersects $b_i$ transversely and positively in one point $x_i$. 
\end{enumerate}
Now we define $\bm{\alpha} = (\alpha_1, \ldots, \alpha_n)$ and $\bm{\beta} = (\beta_1, \ldots, \beta_n)$ to be the sets of simple closed curves on $\Sigma$ such that
\begin{align*}
  \alpha_i &= a_i \times \{0\} \cup a_i \times \{-1\},\\ 
  \beta_i &= b_i \times \{0\} \cup h(b_i) \times \{-1\}.
\end{align*}
Then $(\Sigma, \bm{\alpha}, \bm{\beta})$ is a Heegaard diagram of $(Y,-\Gamma)$. By switching the role of $\bm{\alpha}$ and $\bm{\beta}$, we obtain a Heegaard diagram $\HH_{\OO} = (\Sigma, \bm{\beta},\bm{\alpha})$ of $(-Y,-\Gamma)$. Let $\xi$ be the contact structure on $(Y,\Gamma)$ supported by $\OO = (S,P,h)$. Now we are ready to define the contact invariant of $\xi$. 

\begin{definition}[HKM contact invariant]
  Suppose $(Y,\Gamma,\xi)$, $\OO=(S,P,h)$, and $\HH_{\OO}$ are as above. Then the contact invariant of $(Y,\Gamma,\xi)$ associated to $\OO$ is defined to be a homology class in the sutured Floer homology:
  \[
    \EH(\xi,\HH_{\OO}) = [(x_1,\ldots,x_n)] \in \SFH(\HH_{\OO}).   
  \]
\end{definition}

Honda, Kazez, and Mati\'c \cite{HKM:suturedInvariant} proved that the contact invariant $\EH$ is well-defined. In particular, they showed that $\EH$ is invariant under the positive stabilization of the open book decomposition. We restate these using the naturality of Heegaard Floer homology \cite{JTZ:naturality}.

Juh\'asz, Thurston and Zemke \cite{JTZ:naturality} proved that the sutured Floer homology is a strong Heegaard invariant. That is, for a balanced sutured $3$-manifold $(Y,\Gamma)$ and its Heegaard diagrams $\HH$ and $\HH'$ there exists a sequence of Heegaard moves consisting of a sequence of handleslides and (de)stabilizations that sends $\HH$ to $\HH'$. Then this path induces a graded isomorphism of sutured Floer homology and it is independent of the choice of paths. We denote the isomorphism by 
\[
  F_{\HH\to\HH'}\colon \SFH(\HH) \to \SFH(\HH').
\]
Let $(Y,\Gamma,\xi)$ be a sutured contact $3$-manifold and $\OO=(S,P,h)$, $\OO' = (S',P',h')$ two partial open book decompositions of $(Y,\Gamma)$ that support $\xi$. Let $\HH_{\OO}$ and $\HH_{\OO'}$ be the Heegaard diagrams induced from the open book $\OO$ and $\OO'$. Then according to \cite{Juhasz:sutured}, there are Heegaard moves from $\HH_{\OO}$ to $\HH_{\OO'}$. Then by the naturality of Heegaard Floer homology \cite{JTZ:naturality}, we can restate the invariance of $\EH$ \cite{HKM:suturedInvariant} as   
\begin{align*}
  \EH(\xi,\HH_{\OO'}) = F_{\HH_{\OO}\to \HH_{\OO'}}(\EH(\xi,\HH_{\OO}))
\end{align*}

Now let $\HH$ be another Heegaard diagram for $(Y,\Gamma)$. Then we can define $\EH$ for this Heegaard diagram by 
\begin{align*}
  \EH(\xi,\HH) := F_{\HH_{\OO}\to \HH}(\EH(\xi,\HH_{\OO}))
\end{align*}
Now we have well-defined $\EH(\xi,\HH)$ for any Heegaard diagram $\HH$ of $(Y,\Gamma)$ and indeed obtain well-defined $\EH(\xi) \in \SFH(-Y,-\Gamma)$.

\subsection{The HKM gluing map}
Honda, Kazez and Mati\'c \cite{HKM:suturedTQFT}  constructed a map $\Phi_{\xi}$, which we call the \emph{HKM map}. Let $(Y,\Gamma)$ and $(Y',\Gamma')$ be balanced sutured $3$-manifolds such that $Y' \subset \interior(Y)$. Let $Z = Y \setminus \interior(Y')$ and assume that there is no \emph{isolated component}, a connected component of $Z$ that does not contain any component of $\partial Y$. For a contact structure $\xi$ on $(Z, \Gamma \sqcup \Gamma')$, there exists an HKM map associated to the contact structure $\xi$
\[
  \Phi_{\xi} \colon \SFH(-Y',-\Gamma') \to \SFH(-Y, -\Gamma)
\]
with the following property: If $\eta$ is a contact structure on $(Y',\Gamma')$, then 
\[
  \Phi_{\xi}(\EH(\eta)) = \EH(\eta \cup \xi).
\]

Honda, Kazez and Mati\'c also defined a gluing map \cite[Theorem~1.3]{HKM:suturedTQFT}. Let $(X,\Gamma_X)$, $(Y,\Gamma_Y)$ and $(Z,\Gamma_Z)$ be compact oriented sutured $3$-manifolds. Suppose there is a surface with dividing curves $(F,\Gamma_F)$ such that there are embeddings $(F,\Gamma_F) \hookrightarrow (X,\Gamma_X)$ and $(-F,-\Gamma_F) \hookrightarrow (Z,\Gamma_Z)$. Assume that $(Y,\Gamma_Y)$ can be obtained by gluing $(X,\Gamma_X)$ and $(Z,\Gamma_Z)$ together along $F$: 
\[
  (Y,\Gamma_Y) = (X,\Gamma_X) \Cup_F (Z,\Gamma_Z).
\]
then there exists an \emph{HKM gluing map} 
\[
  \Phi_F \colon \SFH(-X,-\Gamma_X) \otimes \SFH(-Z,-\Gamma_Z) \to \SFH(-Y, -\Gamma_Y).
\]
with the following property: Let $\xi_X$ and $\xi_Z$ be contact structure on $(X,\Gamma_X)$ and $(Z,\Gamma_Z)$, respectively. Then we have 
\[
  \Phi_F(\EH(\xi_X) \otimes \EH(\xi_Z)) = \EH(\xi_X \Cup \xi_Z).
\]

We recall a combinatorial description of HKM gluing map in terms of contact handle attachments by Juh\'asz and Zemke \cite{JZ:contact_handle}. We first review the definition of $3$-dimensional contact $k$-handles due to Giroux \cite{Giroux:convex}; see also \"{O}zba\u{g}c\i~\cite{Ozbagci:handle}. 

\begin{definition}
  For $0 \leq k \leq 3$, a $3$-dimensional \emph{contact $k$-handle} is a sutured contact $3$-manifold  $(H,\Gamma,\xi)$ where $H = D^k \times D^{3-k}$ and $\xi$ is tight. For $k=0,3$, the dividing set $\Gamma$ consists of a single closed curve on $\partial H$. For $k = 1,2$, the dividing set $\Gamma$ consists of one properly embedded arc on each $\partial D^1 \times D^2$ and two properly embedded arcs on $D^1 \times \partial D^2$ that connects two boundary components of $D^1 \times \partial D^2$. 
\end{definition}

Next, we recall the contact handle maps due to Juh\'asz and Zemke \cite{JZ:contact_handle}; see also Leigon and Salmoiraghi \cite{LS:HKMequivZarev}. Since we assume  there are no isolated components, we will only focus on $1$- and $2$-handles.

Let $h^1$ be a contact $1$-handle and $(Y,\Gamma)$ a balanced sutured $3$-manifold. Denote by $(Y',\Gamma')$ the result of contact $1$-handle attachment. Let $\HH = (\Sigma,\bm{\alpha},\bm{\beta})$ be a Heegaard diagram of $(-Y,-\Gamma)$. Then the new Heegaard surface $\Sigma_1$ is obtained by attaching 2-dimensional 1-handle along the points on $\partial\Sigma$ where the attaching sphere of the $h^1$ is located. There are no new $\alpha$ and $\beta$ curves. Thus we obtain a new Heegaard diagram $\HH_1 = (\Sigma_1, \bm{\alpha}, \bm{\beta})$. Then the contact handle map for $h^1$ is defined to be   
\begin{align*}
  C_{h^1} \colon \SFH(\HH) &\to \SFH(\HH_1)\\
  C_{h^1}([\mathbf{x}]) &= [\mathbf{x}]
\end{align*}

Let $h^2$ be a contact $2$-handle and $(Y,\Gamma)$ a balanced sutured $3$-manifold. Denote by $(Y',\Gamma')$ the result of contact $2$-handle attachment. Let $\HH = (\Sigma,\bm{\alpha},\bm{\beta})$ be a Heegaard diagram of $(-Y,-\Gamma)$. The attaching sphere of $h^2$ intersects the dividing set twice in the points $p, q$. Let $\lambda$ be a curve on $Y$ that the attaching sphere is identified and $\lambda_{\pm} := \lambda \cap R_{\pm}(\Gamma)$. We can assume $\lambda_{\pm}$ lie on $\Sigma$ and do not intersect $\alpha$ and $\beta$. Now the surface $\Sigma_2$ is obtained by attaching a $2$-dimensional $1$-handle along $p$ and $q$. Let $c$ be the core of the 2-dimensional 1-handle and define $\alpha_0 := \lambda_+ \cup c$ and $\beta_0 = \lambda_- \cup c$. By perturbing $\alpha_0$ and $\beta_0$, they intersect transversely at one point $x_0$. We obtain a new Heegaard diagram $\HH_2 = (\Sigma_2, \bm{\alpha}\cup\{\alpha_0\}, \bm{\beta}\cup\{\beta_0\})$ for $(-Y_2,-\Gamma_2)$. Then the contact handle map for $h^2$ is defined to be
\begin{align*}
  C_{h^2} \colon \SFH(\HH) &\to \SFH(\HH_2)\\
  C_{h^2}([\mathbf{x}]) &= [\mathbf{x},x_0].
\end{align*}

Using these contact handle maps, Juh\'asz and Zemke \cite{JZ:contact_handle} decomposed the HKM gluing map; see also Leigon and Salmoiraghi \cite{LS:HKMequivZarev}.

\begin{theorem}[Juh\'asz--Zemke \cite{JZ:contact_handle}, Leigon--Salmoiraghi \cite{LS:HKMequivZarev}]\label{thm:digramGluing}
  Let $(Y,\Gamma)$ and $(Y',\Gamma')$ be balanced sutured $3$-manifolds such that $Y' \subset \interior(Y)$. Assume further that there are no isolated components in $Z = Y \setminus \interior(Y')$. For a contact structure $\xi$ on $(Z, \Gamma \sqcup \Gamma')$, suppose that $Z$ admits a contact handle decomposition $N \cup h_1 \cup h_2 \cdots \cup h_n$ where $N$ is a neighborhood of $\partial Y'$ such that $\xi|_N$ is an $I$-invariant neighborhood. Then the HKM map associated with the contact structure $\xi$ can be decomposed into 
  \[
    \Phi_{\xi} = C_{h_n} \circ \cdots \circ C_{h_1} \circ \Phi_{\xi|_N}.
  \] 
  Also, assume that $(Y,\Gamma_Y)$ can be obtained by gluing $(X,\Gamma_X)$ and $(Z,\Gamma_Z)$ together along a subsurface $F$. Suppose that $\xi_Z$ is a contact structure on $(Z,\Gamma_Z)$ and admits a contact handle decomposition $F\times[0,1] \cup h_1 \cup h_2 \cdots \cup h_n$. Then the HKM gluing map associated to $F$ can be decomposed into
  \[
    \Phi_F(\mathbf{x}, \EH(\xi_Z)) = C_{h_n} \circ \cdots \circ C_{h_1} \circ \Phi_F(\mathbf{x}, \EH(\xi|_{F\times I})).
  \]
\end{theorem}

\subsection{Bordered Sutured Floer homology}
We briefly recall the definitions and properties of bordered sutured Floer homology. We recommend that readers refer to \cite{LOT:bordered,Zarev:borderedSutured} for more details.

\begin{definition}
  An arc diagram of rank $k$ is a triple $\Z = (\mathbf{Z},\mathbf{a},M)$ such that 
  \begin{itemize}
    \item $\mathbf{Z}$ is a finite collection of oriented arcs,
    \item $\mathbf{a} = \{a_1,\ldots,a_{2k}\}\subset \mathbf{Z}$ is a collection of points,
    \item $M\colon \mathbf{a} \to \{1,\ldots,k\}$ is a $2$ to $1$ matching,  
    \item $\Z$ is either $\alpha$ or $\beta$ type, and
    \item The result of surgeries on all $M^{-1}(i)$ in $\mathbf{Z}$ has no closed components.
  \end{itemize}
\end{definition}

For a rank $k$ arc diagram $\Z$, we can construct a graph $G(\Z)$ by attaching arcs $e_i$ at the pair $M^{-1}(i) \in \Z$ for $i = 1,\ldots,k$.

Let $\Z$ be an $\alpha$ type arc diagram and $\F(\Z) = (F(\Z),\Lambda(\Z))$ a sutured surface constructed from $\mathbf{Z} \times [0,1]$ by attaching $1$-handles along $M^{-1}(i) \times \{0\}$ for $i=1,\ldots,k$. The sutures are $\Lambda({\Z}) = \partial Z \times \{1/2\}$ with the positive region $S_+(\Lambda)$ being the region containing $\mathbf{Z} \times \{1\}$. We say \emph{$\F(\Z)$ is parametrized by $\Z$}. When $\Z$ is of $\beta$ type, everything is the same but we attach the $1$-handles to $\mathbf{Z}\times\{1\}$. To each arc diagram Z, we can associate a dg algebra $\A(\Z)$, defined combinatorially in terms of chords in the diagram. 

\begin{definition}
    A \emph{bordered sutured manifold} $\Y = (Y, \Gamma, \F)$ is a partially sutured $3$-manifold where $\F$ is parameterized by an arc diagram $\Z$. 
\end{definition}

$\Y$ can be considered as a cobordism from $\F(-\Z_1)$ to $\F(\Z_2)$

\begin{definition}
  A \emph{bordered sutured Heegaard diagram}. $\HH = (\Sigma, \bm{\alpha}, \bm{\beta}, \Z)$ consists of 
  \begin{itemize}
    \item $\Sigma$ is a compact oriented surface without a closed component
    \item $\bm{\alpha} = \bm{\alpha}^a \cup \bm{\alpha}^c$ is a collection of pairwise disjoint curves in $\Sigma$ where $\bm{\alpha}^a$ consists of properly embedded arcs and $\bm{\alpha}^c$ consists of simple closed curves.  
    \item $\bm{\beta}$ is a collection of pairwise disjoint simple closed curves in $\Sigma$.
    \item There exists an embedding $G(\Z) \hookrightarrow \Sigma$ where $\mathbf{Z}$ is sent into $\partial \Sigma$ and $e_i \mapsto \alpha_i^c$.
  \end{itemize}
\end{definition}

Given a bordered sutured Heegaard diagram $\HH = (\Sigma, \bm{\alpha}, \bm{\beta}, \Z)$, we can construct a bordered sutured manifold $(Y,\Gamma,\Z)$ as follows. We first obtain $Y$ from $\Sigma \times [0,1]$ by attaching $2$- handles to $\bm{\beta} \times \{1\}$ and $\bm{\alpha}^c \times \{0\}$. Then we construct $\Gamma = (\partial \Sigma \setminus \mathbf{Z}) \times \{1/2\}$. $F(\Z)$ is a neighborhood of $\mathbf{Z} \times [0,1] \cup \bm{\alpha}^c \times \{0\}$.            
To a bordered sutured Heegaard diagram $\HH$ for $\Y=(Y,\Gamma,\Z)$, Zarev associates a right type-$A$ $\Aoo$-module $\BSA(\HH)_{\A(\Z)}$ and a left type-$D$ module $\leftindex[I]^{\A(\Z)}{\BSD(\HH)}$. We write $\BSA(\Y)$ and $\BSD(\Y)$ as a homotopy equivalence class of ${\BSA(\HH)}$ and ${\BSD(\HH)}$, repsectively. 

For $\Y = (Y, \Gamma, \Z_1 \sqcup \Z_2)$, we can also define bordered bimodules 
\begin{align*}
  &\leftindex[I]_{\A(\Z_1)}\BSAA(\Y)_{\A(\Z_2)}\quad \leftindex[I]^{\A(\Z_1)}\BSDD(\Y)^{\A(\Z_2)}\\ &\leftindex[I]^{\A(\Z_1)}\BSDA(\Y)_{\A(\Z_2)}\quad \leftindex[I]_{\A(\Z_1)}\BSAD(\Y)^{\A(\Z_2)} 
\end{align*}

\begin{theorem}[Zarev \cite{Zarev:borderedSutured}]
  \begin{align*}
    \BSA(\Y_1) \boxtimes \BSD(\Y_2) &= \SFC(\Y_1\cup\Y_2)\\
    \BSAA(\Y_1) \boxtimes \BSDD(\Y_2) &= \BSAD(\Y_1 \cup \Y_2).
  \end{align*} 
   Any combination of bimodules for $\Y_1$ and $\Y_2$ can be used, where one is type–$A$ for $\A(\Z)$, and the other is type–$D$ for $A(\Z)$.
\end{theorem}

Let $\Y = (Y,\Gamma,\F)$ be a bordered sutured manifold. If $\Gamma_D$ is a dividing set with a single suture on a disk $D$ then $\BSA(\Y)$ and $\BSD(\Y)$ can be identified with $\CFA(\Y)$ and $\CFD(\Y)$, respectively.

\subsection{The Zarev gluing map}\label{subsec:gluingZarev}
Zarev \cite{Zarev:joining} defined a \dfn{gluing map} for two sutured manifolds using bordered sutured Floer homology. Let $(Y_1,\Gamma_1)$ and $(Y_2,\Gamma_2)$ be sutured $3$-manifolds, and $F$ and $-F$ are subsurfaces of $Y_1$ and $Y_2$, respectively. Let $\C_1$ and $\C_2$ be the caps associated to $F_1$ and $F_2$, respectively. Recall that the caps embed in $Y_1$ and $Y_2$, and we obtain the following partially sutured manifolds. $\Y_1 = Y_1 \setminus \C_1$ and $\Y_2 = Y_2 \setminus \C_2$. As discussed in Section~\ref{subsec:gluing}, we can glue the two sutured manifolds by inserting a twisting slice.
\[
  (Y_1,\Gamma_1) \Cup_F (Y_2,\Gamma_2) = \Y_1 \cup \TW^{\pm}_{\F} \cup \Y_2
\]
Let $U = \BSA(\Y_1)$ $V=\BSA(\Y_2)$ $M=\BSD(\C_1)$ $A= \BSDD(\TW^+)$. Then we have $\SFC(Y_1,\Gamma_1) = U \boxtimes M$ $\SFC(\Gamma_2) = M^{\vee} \boxtimes V$ and $\SFC(Y_1 \Cup Y_2) = U \boxtimes A  \boxtimes V$ Under this identification, Zarev defined the gluing map $\Psi_F$ as follows:
\begin{align}
  \nonumber \Psi_F\colon \SFC(\Y_1) \otimes \SFC(\Y_2) &\to \SFC(\Y_1 \Cup_F \Y_2)\\
  \Psi (u\boxtimes m, m^{\vee} \boxtimes v) &= u \boxtimes \iota^{\vee} \boxtimes v \label{eq:ZarevJoin}
\end{align}

Zarev also showed that his gluing map is equivalent to $\Aoo$-operation $m_2$. 

\begin{theorem}\label{thm:m2}
  Let $\Y = (Y,\Gamma,\F)$ be a bordered sutured manifold. Then under the identification in Theorem~\ref{thm:CFA=SFC}, we have
  \[
    \Psi_F(x\boxtimes m,a) = m_2(x,a).
  \]  
  In case of $\Y$ has two boundary components, a similar statement holds for $m_{1|1|0}$ and $m_{0|1|1}$. 
\end{theorem}

\subsection{Elementary modules, caps, and twisting slices} \label{subsec:elementary}
Here, we review the bordered modules of elementary caps and twisting slices. For more details, see \cite{Zarev:joining}. We first recall the definition of \emph{elementary modules}. 

\begin{definition}
  A type-$A$ module $M_A$ is called \emph{elementary} if
  \begin{itemize}
    \item $M$ is generated by a single element $m$,
    \item all $\Aoo$-operations on $M$ vanish except for multiplication by an idempotent. 
  \end{itemize}

  A type-$D$ module $\leftindex[I]^A{M}$ is called \emph{elementary} if
  \begin{itemize}
    \item $M$ is generated by a single element $m$,
    \item $\delta(m) = 0$.
  \end{itemize}
\end{definition}

Let $\F$ be a surface parametrized by an arc diagram $\Z$ of rank $n$. For a subset $I \subset \{1,\ldots,n\}$, we define the \emph{elementary cap} $\C_I$ corresponding to $\F$ and $I$ to be the bordered sutured manifold $\C_I(\F) = (\F\times [0,1],\F\times\{0\},\Gamma_I),$ where the dividing curve $\Gamma_I$ is defined to be:
\[
  \Gamma_I=\partial\left(R_0 \cup \bigcup_{i\in I} \nu(e_i)\right)\setminus S_+
\]

Roughly, the dividing curves run ``parallel'' to the parametrizing arcs that are included in $I;$ see Figures \ref{fig:cap_construction} and \ref{fig:cap-torus} for examples.

\begin{figure}
  \centering
  \footnotesize
  \begin{subfigure}{0.4\textwidth}
    \centering
    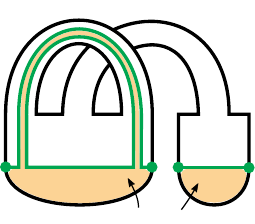
    \caption{$I=\{1\}$}
  \end{subfigure}
  \hspace{.5cm}
  \begin{subfigure}{0.4\textwidth}
    \centering
    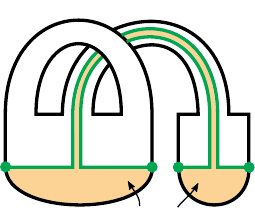
    \caption{$I=\{0\}$}
  \end{subfigure}\\
  \vspace{0.3cm}
  \begin{subfigure}{0.4\textwidth}
    \centering
    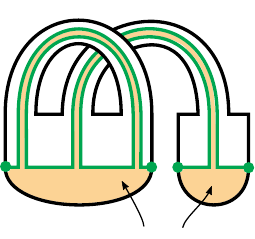
    \caption{$I=\{0,1\}$}
  \end{subfigure}
  \hspace{.5cm}
  \begin{subfigure}{0.4\textwidth}
    \centering
    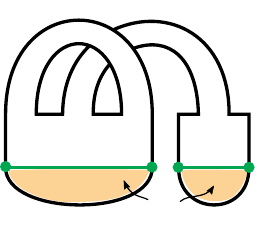
    \caption{$I=\varnothing$}
  \end{subfigure}
  \caption{Constructing the dividing curves for an annulus cap}
  \label{fig:cap_construction}
\end{figure}

\begin{figure}[htbp]{\scriptsize
  \begin{overpic}[tics=20]{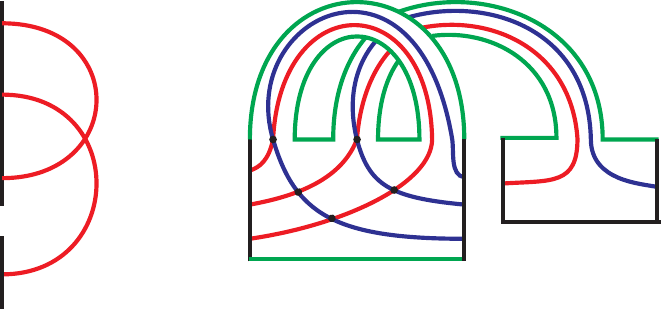}
  \end{overpic}}
  \caption{An arc diagram for an annulus and the corresponding $\overline{\AZ}$ diagram.}
  \label{fig:strandAnnulus}
\end{figure}

\begin{figure}
  \centering
  \includegraphics[scale=.8]{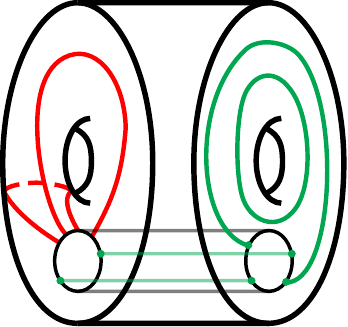}
  \caption{An elementary cap for the torus}
  \label{fig:cap-torus}
\end{figure}

An elementary cap $\C_I$ has a particulary simple bordered sutured Heegaard diagram $\HH_I =  (\Sigma,\boldsymbol{\alpha},\boldsymbol{\beta},\phi:\Z \hookrightarrow \partial \Sigma)$ characterized as follows:
\begin{itemize}
  \item $\boldsymbol{\alpha}$ consists entirely of $\alpha$ arcs such that $\nu(\phi(\Z) \cup \boldsymbol{\alpha})$ is an embedding of $\F$ into $\Sigma$
  \item for each $i\in I,$ there is exactly one $\beta$ circle intersecting $\alpha_i$ exactly once
  \item $\Sigma \setminus (\boldsymbol{\alpha}\cup \boldsymbol{\beta})$ is a disjoint union of disks, each of which contains exactly one component of $\partial \Sigma\setminus \phi(\Z)$
\end{itemize}

\begin{figure}
  \centering
  \begin{subfigure}{.4\textwidth}
    \centering
    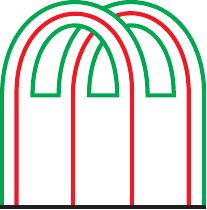
    \caption{$I=\varnothing$}
  \end{subfigure}\hspace{.5cm}
  \begin{subfigure}{.4\textwidth}
    \centering
\begingroup%
  \makeatletter%
  \providecommand\color[2][]{%
    \errmessage{(Inkscape) Color is used for the text in Inkscape, but the package 'color.sty' is not loaded}%
    \renewcommand\color[2][]{}%
  }%
  \providecommand\transparent[1]{%
    \errmessage{(Inkscape) Transparency is used (non-zero) for the text in Inkscape, but the package 'transparent.sty' is not loaded}%
    \renewcommand\transparent[1]{}%
  }%
  \providecommand\rotatebox[2]{#2}%
  \newcommand*\fsize{\dimexpr\f@size pt\relax}%
  \newcommand*\lineheight[1]{\fontsize{\fsize}{#1\fsize}\selectfont}%
  \ifx\svgwidth\undefined%
    \setlength{\unitlength}{157.04502977bp}%
    \ifx\svgscale\undefined%
      \relax%
    \else%
      \setlength{\unitlength}{\unitlength * \real{\svgscale}}%
    \fi%
  \else%
    \setlength{\unitlength}{\svgwidth}%
  \fi%
  \global\let\svgwidth\undefined%
  \global\let\svgscale\undefined%
  \makeatother%
  \begin{picture}(1,0.39239705)%
    \lineheight{1}%
    \setlength\tabcolsep{0pt}%
    \put(0,0){\includegraphics[width=\unitlength,page=1]{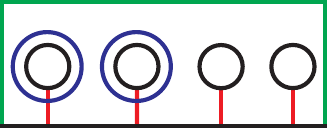}}%
    \put(0.1177901,0.16551058){\color[rgb]{0,0,0}\makebox(0,0)[lt]{\lineheight{1.25}\smash{\begin{tabular}[t]{l}$A$\end{tabular}}}}%
    \put(0.70384349,0.16551058){\color[rgb]{0,0,0}\reflectbox{\makebox(0,0)[lt]{\lineheight{1.25}\smash{\begin{tabular}[t]{l}$A$\end{tabular}}}}}%
    \put(0.39041601,0.16550475){\color[rgb]{0,0,0}\makebox(0,0)[lt]{\lineheight{1.25}\smash{\begin{tabular}[t]{l}$B$\end{tabular}}}}%
    \put(0.92353095,0.16550475){\color[rgb]{0,0,0}\reflectbox{\makebox(0,0)[lt]{\lineheight{1.25}\smash{\begin{tabular}[t]{l}$B$\end{tabular}}}}}%
  \end{picture}%
\endgroup%

    \caption{$I=\{0,1\}$}
  \end{subfigure}\\
  \begin{subfigure}{.4\textwidth}
    \centering
    \vspace{0.3cm}
\begingroup%
  \makeatletter%
  \providecommand\color[2][]{%
    \errmessage{(Inkscape) Color is used for the text in Inkscape, but the package 'color.sty' is not loaded}%
    \renewcommand\color[2][]{}%
  }%
  \providecommand\transparent[1]{%
    \errmessage{(Inkscape) Transparency is used (non-zero) for the text in Inkscape, but the package 'transparent.sty' is not loaded}%
    \renewcommand\transparent[1]{}%
  }%
  \providecommand\rotatebox[2]{#2}%
  \newcommand*\fsize{\dimexpr\f@size pt\relax}%
  \newcommand*\lineheight[1]{\fontsize{\fsize}{#1\fsize}\selectfont}%
  \ifx\svgwidth\undefined%
    \setlength{\unitlength}{99.00050606bp}%
    \ifx\svgscale\undefined%
      \relax%
    \else%
      \setlength{\unitlength}{\unitlength * \real{\svgscale}}%
    \fi%
  \else%
    \setlength{\unitlength}{\svgwidth}%
  \fi%
  \global\let\svgwidth\undefined%
  \global\let\svgscale\undefined%
  \makeatother%
  \begin{picture}(1,1.01387068)%
    \lineheight{1}%
    \setlength\tabcolsep{0pt}%
    \put(0,0){\includegraphics[width=\unitlength,page=1]{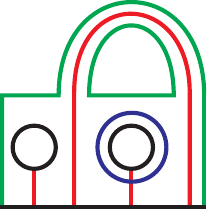}}%
    \put(0.12139559,0.26253516){\color[rgb]{0,0,0}\makebox(0,0)[lt]{\lineheight{1.25}\smash{\begin{tabular}[t]{l}$A$\end{tabular}}}}%
    \put(0.68164736,0.26254117){\color[rgb]{0,0,0}\reflectbox{\makebox(0,0)[lt]{\lineheight{1.25}\smash{\begin{tabular}[t]{l}$A$\end{tabular}}}}}%
  \end{picture}%
\endgroup%

    \caption{$I=\{0\}$}
  \end{subfigure}
  \hspace{.5cm}
  \begin{subfigure}{.42\textwidth}
    \centering
\begingroup%
  \makeatletter%
  \providecommand\color[2][]{%
    \errmessage{(Inkscape) Color is used for the text in Inkscape, but the package 'color.sty' is not loaded}%
    \renewcommand\color[2][]{}%
  }%
  \providecommand\transparent[1]{%
    \errmessage{(Inkscape) Transparency is used (non-zero) for the text in Inkscape, but the package 'transparent.sty' is not loaded}%
    \renewcommand\transparent[1]{}%
  }%
  \providecommand\rotatebox[2]{#2}%
  \newcommand*\fsize{\dimexpr\f@size pt\relax}%
  \newcommand*\lineheight[1]{\fontsize{\fsize}{#1\fsize}\selectfont}%
  \ifx\svgwidth\undefined%
    \setlength{\unitlength}{99.00100348bp}%
    \ifx\svgscale\undefined%
      \relax%
    \else%
      \setlength{\unitlength}{\unitlength * \real{\svgscale}}%
    \fi%
  \else%
    \setlength{\unitlength}{\svgwidth}%
  \fi%
  \global\let\svgwidth\undefined%
  \global\let\svgscale\undefined%
  \makeatother%
  \begin{picture}(1,1.01386559)%
    \lineheight{1}%
    \setlength\tabcolsep{0pt}%
    \put(0,0){\includegraphics[width=\unitlength,page=1]{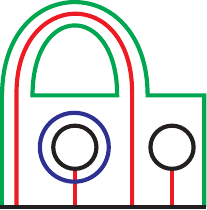}}%
    \put(0.31833986,0.26253964){\color[rgb]{0,0,0}\makebox(0,0)[lt]{\lineheight{1.25}\smash{\begin{tabular}[t]{l}$B$\end{tabular}}}}%
    \put(0.87870349,0.26253964){\color[rgb]{0,0,0}\reflectbox{\makebox(0,0)[lt]{\lineheight{1.25}\smash{\begin{tabular}[t]{l}$B$\end{tabular}}}}}%
  \end{picture}%
\endgroup%

    \caption{$I=\{1\}$}
    \label{fig:caps.d}
  \end{subfigure}
  \caption{Heegaard diagram for elementary caps corresponding to the torus}
  \label{fig:caps}
\end{figure}

\begin{proposition}
  Let $\C_I$ be an elementary cap associated to the elementary dividing set $\Gamma_I$. Then there exists a Heegaard diagram $\HH_I$ such that 
  \begin{itemize}
    \item $\BSA(\HH_I)$ is an elementary typa-$A$ module for $\iota_{I^c}$
    \item $\BSD(\HH_I)$ is an elementary typa-$D$ module for $\iota_{I}$
  \end{itemize}
\end{proposition}

Auroux \cite{Auroux:Fukaya} and Zarev \cite{Zarev:joining} constructed a special Heegaard diagram $\AZ$ for a twisting slice $\TW_{\F}^{+}$. See Figure~\ref{fig:strandAnnulus} and \ref{fig:DD_id} for examples.
We can also obtain the Heegaard diagram $\overline{\AZ}$ for $\TW_{\F}^-$ by interchanging the role of $\alpha$ and $\beta$ arcs. 

\begin{proposition}\label{prop:AZ}
  Let $\AZ$ be the Heegaard diagram for $\TW_{\F}^{+}$ and $\overline{\AZ}$ be the Heegaard diagram for $\TW_{\F}^{-}$. Then there exist canonical identifications as follows.
  \begin{align*}
    \BSAA(\AZ) = \A(\Z)^{\vee}, \quad \BSAA(\overline{\AZ}) =\A(\Z).
  \end{align*}
\end{proposition} 

Zarev \cite{Zarev:joining} characterized bordered sutured Floer homology in terms of sutured Floer homology.

\begin{theorem}[Zarev \cite{Zarev:joining}]\label{thm:CFA=SFC}
  Let $\Y = (Y,\Gamma,\F)$ be a bordered sutured manifold where $\F$ is parametrized by $\Z$ of rank $n$. Then there is a decomposition 
  \[
    \BSA(\Y)_A \cong \bigoplus_{I\subset \{1,\ldots,n\}} \BSA(\Y) \cdot \iota_I  
  \]
  where 
  \[
    \BSA(\Y) \cdot \iota_I \cong \BSA(\Y)\boxtimes \BSD(\HH_I) \cong \SFC(C_{I}(\Y))
  \]
  A similar statement holds for $\BSAA$. In particular,
  \begin{align*}
    \BSAA(\AZ) \cong &\bigoplus_{I,J\subset \{1,\ldots,n\}} \BSD(-\HH_I) \boxtimes \BSAA(\AZ) \boxtimes \BSD(\HH_J)\\ \cong &\bigoplus_{I,J\subset \{1,\ldots,n\}} SFC(\Sigma \times [0,1], \Gamma_{I\to J}) 
  \end{align*}
\end{theorem}

\subsection{Solid tori and basic slices}\label{subsec:basicSlices}
Here, we recall how to construct a tight contact structure on a solid torus. To do so, we first review the Farey graph. We define the \emph{Farey sum} and \emph{Farey multiplication} of two rational numbers to be 
\begin{align*}
  \frac{a}{b} &\oplus \frac{c}{d} := \frac{a+c}{b+d},\\  
  \frac{a}{b} &\bigcdot \frac{c}{d} := ad - bc.
\end{align*}
Now consider the Poincar\'e disk in $\R^2$ equipped with the hyperbolic metric. Label the points $(0,1)$ by $0=0/1$ and $(0,-1)$ by $\infty=1/0$ and add a hyperbolic geodesic between the two points. Take the half circle with non-negative $x$ coordinate. Label any point half-way between two labeled points with the Farey sum of the two points and connect it to both points by a geodesic. Iterate this process until all the positive rational numbers are a label on some point on the half circle. Now do the same for the half circle with non-positive $x$ coordinate (for $\infty$, use the fraction $-1/0$). See Figure~\ref{fig:Farey}. We note that for two points on the Farey graph labeled $r$ and $s$, respectively, we have $|r \bigcdot s| = 1$ if and only if there is an edge between them. 

\begin{figure}[htbp]{\scriptsize
  \begin{overpic}[tics=20]{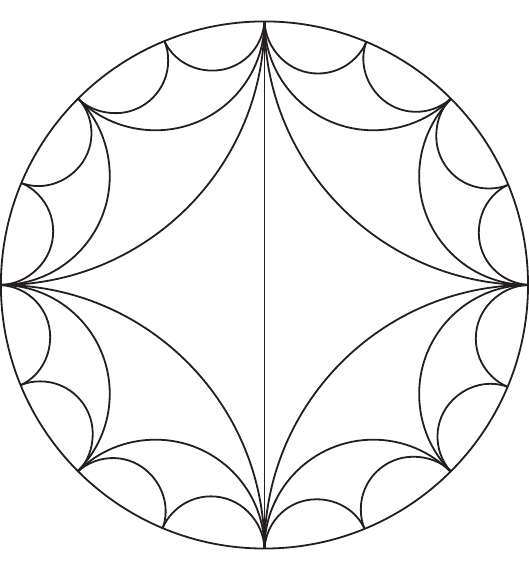}
    \put(123, 0){$\infty$}
    \put(125, 266){$0$}
    \put(-15, 132){$-1$}
    \put(257, 132){$1$}
    \put(20, 38){$-2$}
    \put(222, 38){$2$}
    \put(19, 232){$-1/2$}
    \put(218, 232){$1/2$}
    \put(60, 260){$-1/3$}
    \put(175, 260){$1/3$}
    \put(-13, 185){$-2/3$}
    \put(248, 185){$2/3$}
    \put(-17, 80){$-3/2$}
    \put(248, 80){$3/2$}
    \put(63, 8){$-3$}
    \put(175, 8){$3$}
  \end{overpic}}
  \caption{The Farey graph.}
  \label{fig:Farey}
\end{figure}
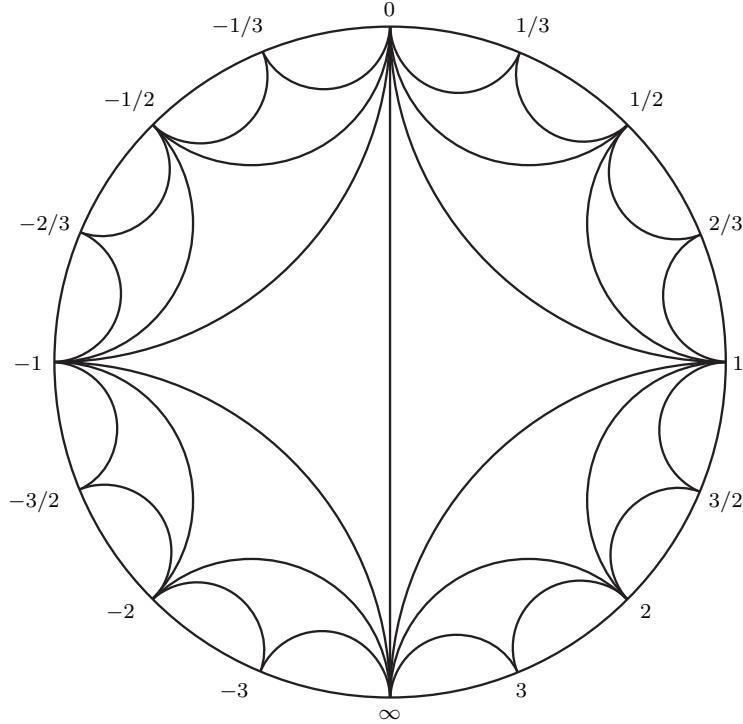

Given a rational number $q/p$ we denote by $(q/p)^c$ is the largest rational number $q'/p'$ satisfying $pq' - p'q = 1$ and by $(q/p)^a$ is the smallest rational number $q''/p''$ satisfying $pq'' - p''q = -1$. Notice that this means there is an edge from $q/p$ to $(q/p)^c$ and to $(q/p)^a$. One may also show that there is an edge between $(q/p)^c$ and $(q/p)^a$. 

A \dfn{path from $r$ to $s$} in the Farey graph is a series of rational numbers $\{r=s_0,s_1,\ldots,s_n=s\}$ moving clockwise from $r$ to $s$ so that $s_i$ and $s_{i+1}$ are connected by an edge in the Farey graph. We say a path $P = \{s_0,s_1,\ldots,s_n\}$ is \dfn{minimal} if $|s_i \bigcdot s_j| = 1$ if and only if $j = i\pm1$. Suppose $P$ is not minimal, that is, $|s_i \bigcdot s_{i+k}| = 1$ for some $k>1$. Then we can \dfn{shorten} the path by removing $s_{i+1}, \ldots, s_{i+k-1}$ and obtain a new path $P' = \{s_1,\ldots,s_i,s_{i+k},\ldots,s_n\}$. A \dfn{decorated path} is a path in the Farey graph such that each edge between two consecutive points is assigned the $+$, $-$, or $\circ$ sign. 

We will use decorated paths in the Farey graph to describe contact structures on $T^2 \times [0,1]$ and $S^1 \times D^2$. A \dfn{basic slice} $B_\pm(s,s')$ is defined to be a minimally twisting contact structure on $T^2 \times I$ with convex boundary with two dividing curves of slopes $s$ and $s'$, where $s'$ is clockwise of $s$ and $|s \bigcdot s'| = 1$. Here minimally twisting means that any convex torus in $B_\pm(s,s')$ that is smoothly isotopic to the boundary has dividing curves that are clockwise of $s$ and anti-clockwise of $s'$ in the Farey graph. There are two non-isotopic contact structures satisfying such conditions and they differ by their coorientation. We denote one by $B_+(s,s')$ and the other by $B_-(s,s')$ and call them a positive and negative basic slice, respectively. Since $|s \bigcdot s'| = 1$, there is an edge between $s$ and $s'$ in the Farey graph. Thus we can describe a basic slice $B_\pm(s,s')$ as a decorated path $(s_0=s,s_1=s')$, consisting of a single edge between $s$ and $s'$, and the sign of the edge is the sign of the basic slice.  

A path in the Farey graph is called a \emph{continued fraction block} if there is a change of basis so that it is $(0,1,\ldots,n)$, a minimal path from $0$ to $n \in \mathbb{N}$. A contact structure $\xi$ on $T^2\times [0,1]$ described by a continued fraction block of edges in the Farey graph will also be called a \dfn{continued fraction block}. Giroux \cite{Giroux:classification} and Honda \cite{Honda:classification1} showed that the signs in a continued fraction block can be \emph{shuffled}, i.e., the following two contact structures are isotopic
\[
    B_+(0,1) \cup B_-(1,2) = B_-(0,1) \cup B_+(1,2).
\]

We will now consider tight contact structures on a solid torus $S$ with boundary being convex with two dividing curves of slope $s$. We will denote such a contact structure by $S(s)$. Kanda~\cite{Kanda:torus} showed that there exists a unique tight contact structure on $S(n)$ up to isotopy fixing boundary if $n \in \mathbb{Z}$, i.e., the longitudinal dividing curves. Since there is an edge between $n$ and $\infty$ in the Farey graph, we can describe the tight contact structure on $S(n)$ as a decorated path $(s_0 = \infty, s_1 =n)$, consisting of a single edge between $\infty$ and $n$, and the sign of the edge is $\circ$.              

Next, we described any tight contact structure on $S(s)$ for any $s \in \mathbb{Q}$. Let $(\infty = s_0, \ldots, s_n = s)$ be the minimal path from $\infty$ to $s$ in the Farey graph. Then we can decompose a tight contact structure on $S(s)$ into 
\[
  S(s) = S(s_1) \cup B(s_1,s_2) \cup \cdots \cup B(s_{n-1},s_n).
\] 
Thus we can describe a tight contact structure on $S(s)$ using a decorated path by assigning $\circ$ to the edge $(s_0,s_1)$ and assigning $+$ or $-$ to all other edges $(s_i,s_{i+1})$ according to the sign of the basic slice $B(s_i,s_{i+1})$ for $1 \leq i \leq n-1$.  Giroux \cite{Giroux:classification} and Honda \cite{Honda:classification1} proved that any tight contact structure on $S(s)$ in this way. 

\begin{figure}[htbp]{
\vspace{0.3cm}
\begin{overpic}[tics=20]{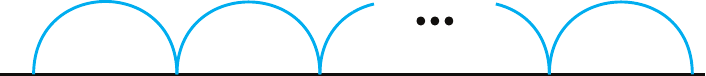}
  \put(47, 10){\LARGE$\circ$}
  \put(116, 10){\large$\pm$}
  \put(296, 10){\large$\pm$}
  \put(10, -10){$\infty$}
  \put(75, -12){$-1$}
  \put(140, -12){$-1/2$}
  \put(328, -12){$1/n$}
\end{overpic}}
\vspace{0.4cm}
\caption{A decorated path in the Farey graph for $S(1/n)$ when $n < 0$.}
\label{fig:decoratedPath1}
\end{figure}

\begin{figure}[htbp]{
  \vspace{0.3cm}
  \begin{overpic}[tics=20]{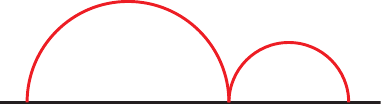}
    \put(57, 12){\LARGE$\circ$}
    \put(136, 10){\large$\pm$}
    \put(8, -10){$\infty$}
    \put(107, -12){$0$}
    \put(160, -12){$1/n$}
  \end{overpic}}
  \vspace{0.4cm}
  \caption{A decorated path in the Farey graph for $S(1/n)$ when $n > 1$.}
  \label{fig:decoratedPath2}
\end{figure}
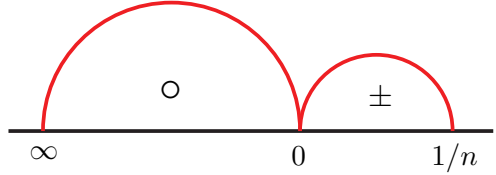

\begin{theorem}[Giroux \cite{Giroux:classification}, Honda \cite{Honda:classification1}]\label{thm:solid-torus}
  Let $P$ be a decorated path from $-\infty$ to $s$ in the Farey graph and $\xi$ be the corresponding contact structure on $S(s)$. Then
  \begin{enumerate}
    \item if $P$ is minimal, then $\xi$ is tight,
    \item if $P$ is minimal and all edges have the same sign except for the first one which is decorated with $\circ$, then $\xi$ is universally tight, 
    \item if $P$ is minimal and contains both $+$ and $-$ signs, then $\xi$ is tight but virtually overtwisted,  
    \item if $P$ is not minimal and there are two adjacent edges $(s_{i-1},s_{i})$ and $(s_i,s_{i+1})$ with different signs that can be shortened ({\it i.e.\ } $|s_{i-1}\bigcdot s_{i+1}|=1$), then $\xi$ is overtwisted,
    \item two minimal decorated paths describe the same contact structure if the decorations differ only by shuffling in continued fraction blocks.
  \end{enumerate}
\end{theorem}

From the above theorem, we derive a useful proposition for later usage.

\begin{proposition}\label{prop:tight-integer-framed}
  Let $n \in \mathbb{Z} \setminus \{0\}$. Then,  
  \begin{itemize}
    \item if $n>1$, there exist two tight contact structures on $S(1/n)$,
    \item if $n=1$, there exists a unique tight contact structure $S(1/n)$,
    \item if $n<0$, there exist $|n|$ distinct tight contact structures on $S(1/n)$.
  \end{itemize}
  Also, any tight contact structure on $S(1/n)$ has non-vanishing contact invariant $\EH$.
\end{proposition}

\begin{proof}
  If $n=1$ the dividing curve is longitudinal, so there exists a unique tight contact structure by Kanda \cite{Kanda:torus} as discussed above.

  Now consider the case $n < 0$. The minimal path from $\infty$ to $1/n$ is $(-\infty, -1, -1/2, -1/3, \ldots, 1/n)$. See Figure~\ref{fig:decoratedPath1}. Thus we can decompose a tight contact structure into 
  \[
    S(1/n) = S(-1) \cup B_{\pm}(-1,-1/2) \cup \cdots \cup B_{\pm}(1/(n+1),1/n).
  \] 
  The subpath $(-1,-1/2, \ldots, 1/n)$ is a continued fraction block, so according to Theorem~\ref{thm:solid-torus}, a contact  structure on $S(1/n)$ is determined by the signs on the edges in the subpath up to shuffling. Since there are $|n|$ different choices of signs on the edges up to shuffling, there are $|n|$ tight contact structures on $S(1/n)$.
   
  If $n > 1$, the minimal path from $\infty$ to $1/n$ is $(\infty, 0, 1/n)$. Therefore, any tight contact structure on $S(1/n)$ can be decomposed into
  \[
  S(0) \cup B_{\pm}(0,1/n)
  \]
  and the sign of the basic slice determines the contact structure. See Figure~\ref{fig:decoratedPath2}.

  The tight contact structures considered so far embed in $(S^3,\xi_{std})$ (c.f. \cite{KM:classification}). Thus they have non-vanishing contact invariant.
\end{proof}

\subsection{Tight contact structures on a twisting slice}
Let $\Sigma$ be a compact oriented connected surface with boundary that is parametrized by an arc diagram $\Z$ of rank $n$. We define $(\Sigma\times [0,1],\Gamma_{I\to J})$ to be $-\C_I \cup \TW^+_{\Sigma} \cup \C_J$ for $I,J \subset \{1,\ldots,n\}$. Mathews \cite{Mathews:contactCategory} showed that there exists an isomorphism between the contact category algebra $CA(\Sigma,\Z)$ and the homology of the strand algebra $H_*(\A(\Z))$, where $CA(\Sigma, \Z)$ is generated by isotopy classes of tight contact structures on $(\Sigma \times [0,1],\Gamma_{I\to J})$ for $I,J \subset \{1,\ldots,n\}$. 

Zarev \cite{Zarev:joining} showed that there exists an isomorphism between the homology of $\iota_I \cdot \A(\Z) \cdot \iota_J$  and the sutured Floer homology of $(-\Sigma\times [0,1],-\Gamma_{I\to J})$ for $I,J \subset \{1,\cdots,n\}$, see Proposition~\ref{prop:AZ}. Combining these, we obtain the following equivalence: 
\begin{align*}
  CA(\Sigma,\Z) \cong H_*(\A(\Z)) \cong \bigoplus_{I,J \subset \{1,\cdots,n\}} \SFH(-\Sigma\times[0,1],-\Gamma_{I\to J}).
\end{align*}

\begin{remark}
  The statements in \cite{Mathews:contactCategory} and \cite{Zarev:joining} are slightly different since $(\Sigma\times [0,1],\Gamma_{I\to J})$ is defined te be $-\C_I \cup \TW^-_{\Sigma} \cup \C_J$. 
\end{remark}

To establish the first isomorphism, Mathews constructed a contact structure $\xi_a$ on $(\Sigma \times [0,1],\Gamma_{I\to J})$ for each generator $a = a_i(S,T,\phi) \in \A(\Z)$. 

\begin{proposition}[Mathews \cite{Mathews:contactCategory}]
  Let $a = a(S,T,\phi)$ and $b = a(S',T',\phi')$ be generators of $\A(\Z)$. Then $\xi_a$ is tight if and only if $[a] \neq 0 \in H(\A(\Z))$. Also, $\xi_a \sim \xi_b$ if and only if $[a]=[b] \in H(\A(\Z))$.
\end{proposition}

The following proposition is a direct consequence of the equivalence above.

\begin{proposition}\label{prop:contactAZ}
  Let $a = a_i(S,T,\phi)$ be a generator of $\A(\Z)$. Then the isomorphism induced from $\AZ$ diagram sends 
  \begin{align*}
    [a] \mapsto \EH(\xi_a)
  \end{align*}
  Moreover, $\EH$ is a complete invariant for contact structures on $(\Sigma\times[0,1], \Gamma_{I\to J})$, i.e., $\EH(\xi_a)$ is nonvanishing if and only if $\xi_a$ is tight, and $\xi_a \sim \xi_b$ if and only if $\EH(\xi_a) = \EH(\xi_b)$.
\end{proposition}

\subsection{Heegaard Floer invariants for knots}\label{subsec:knots}
Let $Y$ be a closed $3$-manifold and $K$ a null-homologous knot in $Y$. Then we have a filtration of $\CFhat(Y)$:
\begin{align*}
  \cdots \subset \F_{m-1} \subset \F_m \subset \F_{m+1} \subset \cdots \subset \CFhat(Y)
\end{align*}
such that $\F_m = 0$ for $m << 0$ and $\F_m = \CFhat(Y)$ for $m >> 0$.

Once we fix a non-zero element $x \in \HFhat(Y)$, then we obtain an integer-valued invariant $\tau_x(Y,K)$:
\begin{align*}
  \tau_x(Y,K) = \min \{ s \mid x \in \mathrm{Im}(H_*(\F_s) \to \HFhat(Y))\}
\end{align*}

\begin{definition}[Hedden \cite{Hedden:tau}]
  Let $(Y,\xi)$ be a closed contact 3-manifold and $K$ a null-homologous knot in $Y$. Suppose the contact class $c(\xi) \in \HFhat(-Y)$ is nonzero. Then the \emph{contact tau-invariant} of $K$ is defined to be
  \begin{align*}
    \tau_{\xi}(Y,K) := -\tau_{c(\xi)}(-Y,K)
  \end{align*}
\end{definition}

\begin{theorem}[Hedden \cite{Hedden:tau}]
  Let $(Y,\xi)$ be a closed contact $3$-manifold with $c(\xi) \neq 0$. Then for any null-homologous Legendrian knot $L$ in $Y$, the following inequality holds.
  \begin{align*}
    \tb(L) + |\rot(L)| \leq 2\tau_{\xi}(Y,L) - 1
  \end{align*}
\end{theorem}

Now let $Y$ be an L-space and $\mathfrak{s}$ a spin$^c$ structure of $Y$. In this case, there is a horizontally simplified basis $\{x_i\}$ of $\CFK(Y,K,\mathfrak{s})$ such that some $x_j$ is the distinguished generator of a vertically simplified basis. Then we define the \emph{epsilon-invariant of $K$ with respect to $\mathfrak{s}$} to be

\begin{itemize}
  \item $\epsilon_{\mathfrak{s}}(Y,K) = 1$ if $x_j$ is in the image of the horizontal differential,
  \item $\epsilon_{\mathfrak{s}}(Y,K) = 0$ if $x_j$ is in the kernel but not the image of the horizontal differential, 
  \item $\epsilon_{\mathfrak{s}}(Y,K) = -1$ if $x_j$ is not in the kernel of the horizontal differential.
\end{itemize}

\begin{definition}
  Let $(Y,\xi)$ be a contact L-space with $c(\xi) \neq 0$. Then the \emph{contact epsilon-invariant} of $K$ is defined to be
  \begin{align*}
    \epsilon_{\xi}(Y,K) := -\epsilon_{\mathfrak{s}_{\xi}}(-Y,K)
  \end{align*}
  where $\mathfrak{s}_{\xi}$ is the spin$^c$ structure induced from the contact structure $\xi$.
\end{definition}

\subsection{Torus Boundary and Immersed curves}\label{sec:immersed}
In the case of a manifold $Y$ with torus boundary, the simplest case is where the sutured part consists of a single disk with a single boundary-parallel dividing curve; see Figure \ref{fig:parametrizedTorus.c}. This corresponds to the ``classic'' torus algebra $\A(T^2)$ and bordered modules $\CFA(Y)$ and $\CFD(Y)$ of \cite{LOT:bordered}. We recall some of the key facts about this case. The summand of the algebra that is relevant for pairing is the level-0: $\A(T^2,0),$ which has 8 generators as an $\F$-vector space:
\begin{itemize}
  \item two idempotents $\iota_0$ and $\iota_1$
  \item three ``basic'' generators $\rho_1,$ $\rho_2,$ and $\rho_3$ satisfying: $\rho_1 = \iota_0 \rho_1 \iota_1,$  $\rho_2 = \iota_1 \rho_2 \iota_0,$ and $\rho_3 = \iota_0 \rho_3 \iota_1$
  \item $\rho_{12} = \rho_1 \rho_2,$ $\rho_{23} = \rho_2 \rho_3,$ and $\rho_{123} = \rho_1 \rho_2 \rho_3$ 
\end{itemize}
Note: all other kinds of products not indicated are zero (e.g., $\rho_2\rho_1=0$).

\begin{figure}
  \centering
  \begin{subfigure}{.4\textwidth}
    \centering
\begingroup%
  \makeatletter%
  \providecommand\color[2][]{%
    \errmessage{(Inkscape) Color is used for the text in Inkscape, but the package 'color.sty' is not loaded}%
    \renewcommand\color[2][]{}%
  }%
  \providecommand\transparent[1]{%
    \errmessage{(Inkscape) Transparency is used (non-zero) for the text in Inkscape, but the package 'transparent.sty' is not loaded}%
    \renewcommand\transparent[1]{}%
  }%
  \providecommand\rotatebox[2]{#2}%
  \newcommand*\fsize{\dimexpr\f@size pt\relax}%
  \newcommand*\lineheight[1]{\fontsize{\fsize}{#1\fsize}\selectfont}%
  \ifx\svgwidth\undefined%
    \setlength{\unitlength}{62.90328598bp}%
    \ifx\svgscale\undefined%
      \relax%
    \else%
      \setlength{\unitlength}{\unitlength * \real{\svgscale}}%
    \fi%
  \else%
    \setlength{\unitlength}{\svgwidth}%
  \fi%
  \global\let\svgwidth\undefined%
  \global\let\svgscale\undefined%
  \makeatother%
  \begin{picture}(1,1.49826446)%
    \lineheight{1}%
    \setlength\tabcolsep{0pt}%
    \put(0,0){\includegraphics[width=\unitlength,page=1]{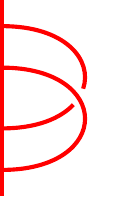}}%
    \put(0.7301011,0.50469622){\color[rgb]{0,0,0}\makebox(0,0)[lt]{\lineheight{1.25}\smash{\begin{tabular}[t]{l}$\alpha_0$\end{tabular}}}}%
    \put(0.7301011,0.9816188){\color[rgb]{0,0,0}\makebox(0,0)[lt]{\lineheight{1.25}\smash{\begin{tabular}[t]{l}$\alpha_1$\end{tabular}}}}%
  \end{picture}%
\endgroup%

    \caption{An alpha-type diagram for the (once punctured) torus algebra}
  \end{subfigure}\hspace{.5cm}
  \begin{subfigure}{.4\textwidth}
    \centering
\begingroup%
  \makeatletter%
  \providecommand\color[2][]{%
    \errmessage{(Inkscape) Color is used for the text in Inkscape, but the package 'color.sty' is not loaded}%
    \renewcommand\color[2][]{}%
  }%
  \providecommand\transparent[1]{%
    \errmessage{(Inkscape) Transparency is used (non-zero) for the text in Inkscape, but the package 'transparent.sty' is not loaded}%
    \renewcommand\transparent[1]{}%
  }%
  \providecommand\rotatebox[2]{#2}%
  \newcommand*\fsize{\dimexpr\f@size pt\relax}%
  \newcommand*\lineheight[1]{\fontsize{\fsize}{#1\fsize}\selectfont}%
  \ifx\svgwidth\undefined%
    \setlength{\unitlength}{57.8009491bp}%
    \ifx\svgscale\undefined%
      \relax%
    \else%
      \setlength{\unitlength}{\unitlength * \real{\svgscale}}%
    \fi%
  \else%
    \setlength{\unitlength}{\svgwidth}%
  \fi%
  \global\let\svgwidth\undefined%
  \global\let\svgscale\undefined%
  \makeatother%
  \begin{picture}(1,1.63052188)%
    \lineheight{1}%
    \setlength\tabcolsep{0pt}%
    \put(0,0){\includegraphics[width=\unitlength,page=1]{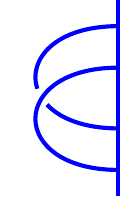}}%
    \put(-0.00392832,0.53802269){\color[rgb]{0,0,0}\makebox(0,0)[lt]{\lineheight{1.25}\smash{\begin{tabular}[t]{l}$\beta_0$\end{tabular}}}}%
    \put(-0.00392832,1.05704527){\color[rgb]{0,0,0}\makebox(0,0)[lt]{\lineheight{1.25}\smash{\begin{tabular}[t]{l}$\beta_1$\end{tabular}}}}%
  \end{picture}%
\endgroup%

    \caption{A beta-type diagram for the torus algebra}
  \end{subfigure}\\
  \vspace{.3cm}
  \begin{subfigure}{.4\textwidth}
    \centering
    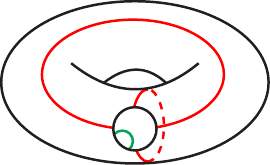
    \caption{The bordered sutured parametrization with a single disk suture}\label{fig:parametrizedTorus.c}
  \end{subfigure}
  \hspace{.5cm}
  \begin{subfigure}{.42\textwidth}
    \centering
    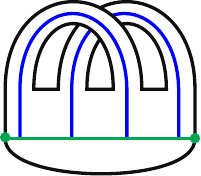
    \caption{Arc diagram for the torus embedded in the once-punctured torus}
  \end{subfigure}
  \caption{Arc diagrams and parametrization for the torus with single disk  suture}
  \label{fig:parametrizedTorus}
\end{figure}
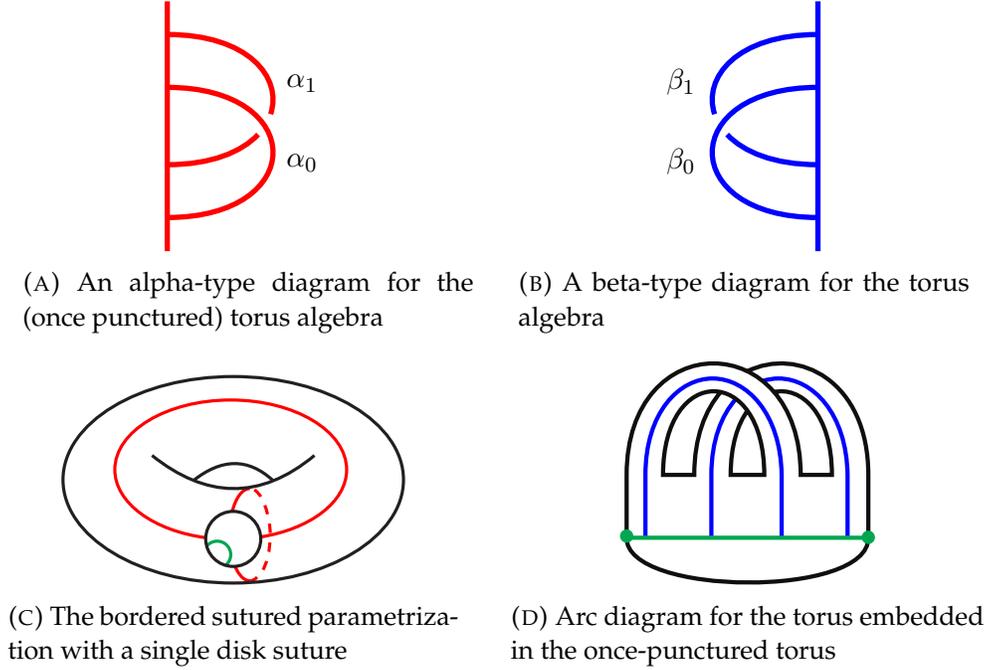

We shall often describe the type-$A$ and type-$D$ structrues over this algebra as directed graphs with edges labelled by strings of algebra elements as follows:\\
For a type-$D$ structure, an edge from $x$ to $y$ labelled by $a$ (which will always correspond to an algebra element) indicates that $\delta^1 (x) = a \otimes y.$ Hence, the higher differentials $\delta^{n}$ correspond to directed paths of length $n.$\\
For a type-$A$ structure, if there is a directed path from $x$ to $z$ with labels $a_1,\dots, a_n,$ then this will correspond to a higher multiplication of the form $m_i(x,\dots)=y.$ More precisely, the strings $a_1,\dots, a_n,$ are first concatenated, then re-bracketed in order to multiply as many adjacent terms as possible in a non-vanishing way. For example, if we had $a_1= \rho_2 \rho_1$ and $a_2=\rho_2\rho_3$ then this would be grouped as: $\rho_2, \rho_1\rho_2\rho_3$ or $\rho_2,\rho_{123}.$

Bordered modules over this torus algebra as well as the pairing theorem have been reinterpreted in terms of immersed curves on the punctured torus by Hanselman, Rasmussen, and Watson \cite{HRW:immersed}. We briefly recall the construction:
For the type-$A$ module, given a directed graph as above, we place each generator on either the vertical or horizontal ``edge'' of a punctured torus depending on its idempotent: those with $\iota_0$ on the vertical and those with $\iota_1$ on the horizontal. Then For each arrow between generators, we insert the corresponding kind of arc as in Figure \ref{fig:typeA}. Notice that the puncture is in the top-right corner.\\
For the type-$D$ structure, we proceed similarly, except this time, the generators in the $\iota_0$ idempotent are on the horizontal edge while the ones in $\iota_1$ are on the vertical edge. Once again, we follow the instructions in Figure \ref{fig:typeD} for each arrow occurring in the directed graph. Note this time the puncture is in the bottom-left corner.\\
In either case, we obtain a collection of immersed \emph{train tracks} which may be arranged as described in \cite{HRW:immersed} such that the result may be interpreted as an immersed multicurve with \emph{local systems}.
\begin{figure}
  \begin{subfigure}{.45\textwidth}
    \centering
    \footnotesize
    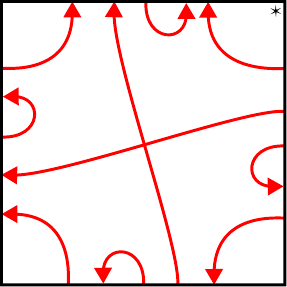
    \caption{Type-$A$ immersed curves corresponding to various edge labels in a directed graph representation}
    \label{fig:typeA}
  \end{subfigure} \hfill
  \begin{subfigure}{.45\textwidth}
    \centering
    \footnotesize
    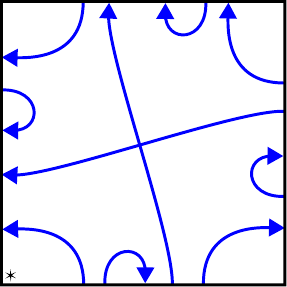
    \caption{type-$D$ immersed curves corresponding to various edge labels in a directed graph representationa}
    \label{fig:typeD}
  \end{subfigure}
  \caption{Translating a type-$A$ or type-$D$ module over the torus algebra into immersed curves}
  \label{fig:Imm_AD}
\end{figure}

Given immersed curves $\overrightarrow{\Gamma_A}$ and $\overrightarrow{\Gamma_D}$ representing $\CFA$ and $\CFD$, respectively, we can compute $\CFA \boxtimes \CFD$ geometrically as follows: we scale the immersed curves by a factor of $1/2$ and embed them into a torus in the top-right and bottom-left quadrants, respectively, extending the ``endpoints'' vertically and horizontally as appropriate. Then the Lagrangian Floer complex between $\overrightarrow{\Gamma_A}$ and $\overrightarrow{\Gamma_D}$ in $T^2_{\bullet}$ computes the Heegaard Floer complex of $Y_1\cup Y_2$: the intersections between the vertical and horizontal extended lines are the generators, and the differential amounts to counting immersed bigons between them. In this setup:
\[
  \CFhat(Y_1\cup Y_2) \cong \CFA \boxtimes \CFD \cong CF\left(\overrightarrow{\Gamma_A}, \overrightarrow{\Gamma_D}\right)
\]
Figure \ref{fig:Imm_AD_pair} shows an example (local) computation. Here we have generators $a$ and $b$ in the type-$A$ module and generators $x$ and $y$ in the type-$D$ structure. Moreover, $m_2(a,\rho_2)=b$ and $\delta^1(x)=\rho_2\otimes y.$ The bigon from $a\boxtimes x$ to $b\boxtimes y$ represents a differential in the chain complex, as expected.

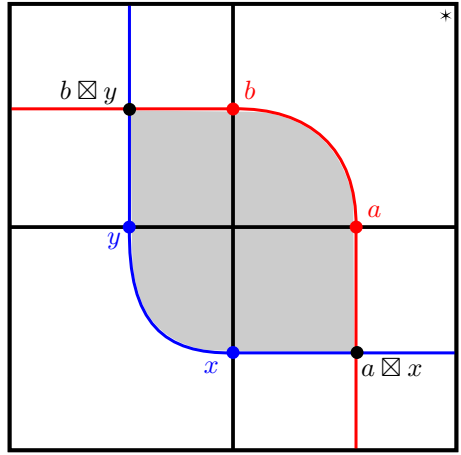
\begin{figure}
  \centering
  \small
\begingroup%
  \makeatletter%
  \providecommand\color[2][]{%
    \errmessage{(Inkscape) Color is used for the text in Inkscape, but the package 'color.sty' is not loaded}%
    \renewcommand\color[2][]{}%
  }%
  \providecommand\transparent[1]{%
    \errmessage{(Inkscape) Transparency is used (non-zero) for the text in Inkscape, but the package 'transparent.sty' is not loaded}%
    \renewcommand\transparent[1]{}%
  }%
  \providecommand\rotatebox[2]{#2}%
  \newcommand*\fsize{\dimexpr\f@size pt\relax}%
  \newcommand*\lineheight[1]{\fontsize{\fsize}{#1\fsize}\selectfont}%
  \ifx\svgwidth\undefined%
    \setlength{\unitlength}{169.79603685bp}%
    \ifx\svgscale\undefined%
      \relax%
    \else%
      \setlength{\unitlength}{\unitlength * \real{\svgscale}}%
    \fi%
  \else%
    \setlength{\unitlength}{\svgwidth}%
  \fi%
  \global\let\svgwidth\undefined%
  \global\let\svgscale\undefined%
  \makeatother%
  \begin{picture}(1,0.99999541)%
    \lineheight{1}%
    \setlength\tabcolsep{0pt}%
    \put(0,0){\includegraphics[width=\unitlength,page=1]{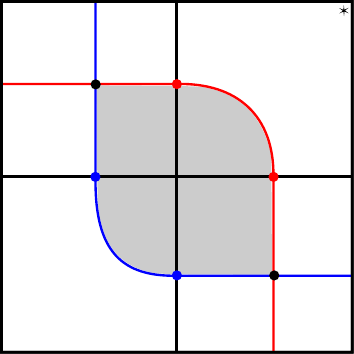}}%
    \put(0.11592896,0.78340085){\color[rgb]{0,0,0}\makebox(0,0)[lt]{\lineheight{1.25}\smash{\begin{tabular}[t]{l}$b\boxtimes y$\end{tabular}}}}%
    \put(0.78420363,0.16988071){\color[rgb]{0,0,0}\makebox(0,0)[lt]{\lineheight{1.25}\smash{\begin{tabular}[t]{l}$a\boxtimes x$\end{tabular}}}}%
    \put(0.79842578,0.52266946){\color[rgb]{1,0,0}\makebox(0,0)[lt]{\lineheight{1.25}\smash{\begin{tabular}[t]{l}$a$\end{tabular}}}}%
    \put(0.43496546,0.17507618){\color[rgb]{0,0,1}\makebox(0,0)[lt]{\lineheight{1.25}\smash{\begin{tabular}[t]{l}$x$\end{tabular}}}}%
    \put(0.22044004,0.46280648){\color[rgb]{0,0,1}\makebox(0,0)[lt]{\lineheight{1.25}\smash{\begin{tabular}[t]{l}$y$\end{tabular}}}}%
    \put(0.52510141,0.7848916){\color[rgb]{1,0,0}\makebox(0,0)[lt]{\lineheight{1.25}\smash{\begin{tabular}[t]{l}$b$\end{tabular}}}}%
  \end{picture}%
\endgroup%

  \caption{Pairing immersed curves recovers the box tensor product for pairing type-$A$ and type-$D$ modules over the torus algebra}
  \label{fig:Imm_AD_pair}
\end{figure}

\begin{remark}\label{rmk:reparam}
  The immersed curve invariant $\overrightarrow{\Gamma}$ in fact may be canonically thought of as living on the (bordered) boundary of the manifold itself, so that a reparametrization simply corresponds to applying an appropriate mapping class to $\overrightarrow{\Gamma}.$
\end{remark}

An important class of manifolds with torus boundary that we shall be interested in is complements of knots in $S^3$ (or more generally, in L-spaces).

Secion 4.2 of \cite{HRW:immersed2} describes an algorithm that translates the knot Floer chain complex to immersed curves, based on the correspendence between knot Floer homology and the bordered Floer homology of the knot exterior described in Chapter 11 of \cite{LOT:bordered}. Indeed, for a null-homologous knot, there is a canonical lift of the immersed curves to the infinitely punctured cylinder $\overline{T^2}_{\bullet}$ such that the Seifert longitude lifts to a simple closed curve, while the meridian lifts to a vertical line. It is conventional to parametrize this cylinder as $[0,1]/0\sim 1 \times \R $ The  punctures usually are chosen to be $\{\frac{1}{2}\}\times(\ZZ+\frac{1}{2}),$ although in Figure \ref{fig:Imm-HFK}, we have moved them slightly down and to the left, so that the midline $\left\{\frac{1}{2}\right\}\times \R$ and the lifts of the Seifert longitude $S^1 \times \{n+\frac{1}{n}\}$ are more clearly visible. Indeed, the various fundamental domains, which are lifts of the punctured torus, are the sets: $\left(\left[\frac{1}{2},1\right]\cup\left[0,\frac{1}{2}\right]\right)\times \left(n-\frac{1}{2},n+\frac{1}{2}\right).$ This integer $n$ is called the \emph{height} of the region, and intersection points between the lift of $\overrightarrow{\Gamma}$ and the midline occurring in such a region are also said to have height $n.$ These heights correspond to the various spin$^c$ gradings (or Alexander gradings in the knot Floer viewpoint).

The properties of the knot Floer homology for a knot $K$ in $S^3$ (or more generally in an L-space $Y$) translate as follows in this setup:
\begin{itemize}
  \item The entire collection is rotationally symmetric about $\left(\frac{1}{2},0\right)$; this is a restatement of the $i,j$-symmetry in knot Floer homology
  \item There is a distinguished curve $\gamma_{\mathfrak{s}}$ for each $\mathfrak{s}\in Spin^c(Y)$ which wraps once in the longitudinal direction; this corresponds to the ``staircase'' component of knot Floer homology. All other curves may be made to lie within an arbitrarily small neighborhood of the midline $\left\{\frac{1}{2}\right\}\times \R$ by an appropriate homotopy. Moreover, this distinguished curve has \emph{trivial} local system; this reflects the fact that $\rk(\HFhat(Y,\mathfrak{s}))=1$ for an L-space
  \item Following the distinguished curve from left to right, starting at $(0,0)$, the height at which it ``first'' intersects the midline is $\tau_{x},$ where $x$ is a generator of $\HFhat(Y,\mathfrak{s})$; indeed, this intersection point represents the homology class $x,$ which generates the vertical ($U=0$, $V=1$) quotient complex of $\CFK(Y,K)/(UV=0).$ Similarly, if one follows the distinguished curve from right to left starting at $(1,0),$ then the first intersection point with the midline will occur at height $-\tau_{x}$ and will correspond to a generator of the \emph{horizontal} ($V=0,$ $U=1$) complex.
  \item  Immediately to the right of the generator of the vertical complex (from above), if the curve proceeds to turn upward, then $\epsilon_{\mathfrak{s}} = -1;$ if it turns downward, then $\epsilon_{\mathfrak{s}} = 1.$ Otherwise, if it is horizontal, $\epsilon_{\mathfrak{s}} = 0;$ by symmetry, this actually implies $\tau_{x}=0$ as well.
\end{itemize}

\begin{figure}
  \centering
  \begin{subfigure}{.21\textwidth}
    \centering
    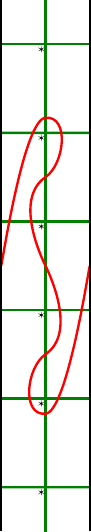
    \caption{Distinguished curve: ${\tau=2},$ ${\epsilon =1}$}
  \end{subfigure}\hspace{.3cm}
  \begin{subfigure}{.21\textwidth}
    \centering
    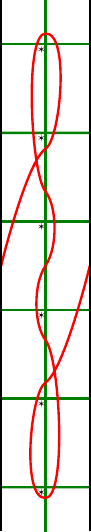
    \caption{Distinguished curve: ${\tau=1},$ ${\epsilon =-1}$}
  \end{subfigure}\hspace{.3cm}
  \begin{subfigure}{.21\textwidth}
    \centering
    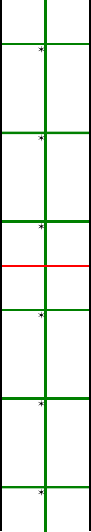
    \caption{Distinguished curve: $\tau=\epsilon =0$}
  \end{subfigure}
  \hspace{.3cm}
  \begin{subfigure}{.21\textwidth}
    \centering
    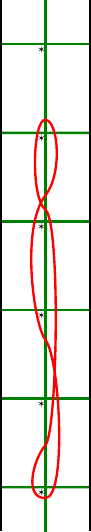
    \caption{Not a distinguished curve}
  \end{subfigure}
  \caption{Various behaviors for lifted immersed curve invariants for knot exteriors}
  \label{fig:Imm-HFK}
\end{figure}

\section{Bordered contact invariants}\label{sec:invariants}

In this section, we define contact invariants in bordered (bi)modules and prove the invariance and the pairing theorems.

\subsection{Definition and invariance}
Let $\Y=(Y,\Gamma,\F_1 \sqcup \F_2)$ be a bordered sutured manifold where $\F_1$ and $\F_2$ are surfaces parametrized by arc diagrams $\Z_1$ and $\Z_2$ of rank $n_1$ and $n_2$, respectively. Recall from Section~\ref{subsec:elementary} that for $I \subset \{1,\ldots,n_1\}$ and $J \subset \{1,\ldots,n_2\}$, there exist elementary caps $\C_I$ and $\C_J$ and we obtain a sutured manifold by attaching the caps to $\Y$:
\[
  C(\Y) := \C_I \cup \Y \cup \C_J = (Y, \Gamma_I \cup \Gamma \cup \Gamma_J).
\]    
We say a contact structure $\xi$ on $C(\Y)$ is \emph{compatible with} $\Y$, and call $(\Y,\xi)$ a \emph{bordered sutured contact manifold}. 

Let $(\Y,\xi)$ be a bordered sutured contact manifold. We first define a type-$AA$ contact invariant of $\xi$. To do so, we recall from Section~\ref{subsec:elementary} that for a bordered Heegaard diagram $\HH$ of $-\Y$, we can construct a sutured Heegaard diagram $C(\HH)$ of $C(-\Y)$ by attaching the elementary Heegaard diagrams $\HH_{I^c}$ and $\HH_{J^c}$. Also, we say two cycles in a type-$AA$ module $M$ are \emph{equivalent} if they are the same in the homology; i.e., $a\sim b$  if $[a] = [b] \in H_*(M,\partial)$ where $\partial = m_{0|1|0}$.

\begin{definition}
Let $\Y=(Y,\Gamma,\F_1\sqcup\F_2)$ be a bordered sutured manifold and $\xi$ a contact structure compatible with $\Y$. For an admissible Heegaard diagram $\HH$ for $-\Y$, we define $c_{\AAA}(\xi, \HH)$ to be the equivalence class of cycles in $\BSAA(\HH)$ such that 
\[
  \left[\iota_I \cdot c_{\AAA}(\xi,\HH) \cdot \iota_J\right]=\EH(\xi, C(\HH)) \in \SFH(C(\HH)).
\]
\end{definition}

\begin{proof}[Proof of Theorem~\ref{thm:main}] 
  We show that $c_{\AAA}(\xi)$ is well-defined and is indeed an invariant of the bordered sutured contact manifold, and then define other flavors. First, recall from Section~\ref{subsec:invariant} that $\EH(\xi)$ is independent of the choice of a Heegaard diagram for $-\Y$ and hence well-defined.  
  Now we show that $c_{\AAA}(\xi)$ also does not depend on the choice of bordered sutured Heegaard diagram of $\Y$. Let $\HH$ and $\HH'$ be two Heegaard diagrams of $\Y$, and denote $M=\BSAA(\HH)$ and $M'=\BSAA(\HH')$. Then by \cite[Proposition~4.5]{Zarev:borderedSutured}, there exists a sequence of Heegaard moves sending $\HH$ to $\HH'$, supported on the interior of the Heegaard surfaces. These moves induce an $\Aoo$ homotopy equivalence: 
  \[
    f=\bigoplus f_{i,j}\colon \A(\Z)^{\otimes (i-1)} \otimes M \otimes \A(\Z)^{\otimes (j-1)} \to M'.
  \]
  This map respects idempotents, and hence the restriction of $f_{0,0}$ to $\iota_I \cdot M \cdot \iota_J$ is a chain homotopy equivalence. We obtain an induced chain homotopy equivalence by pairing with elementary caps:
  \[
    F=id \boxtimes f \boxtimes id\colon \iota_I \cdot M \cdot \iota_J \to \iota_I \cdot M' \cdot \iota_J
  \]
  Since the type-$D$ structures for the elementary caps have a unique generator with eo differential, we must have that:
  \[
    m_I^{\vee} \boxtimes c_{\AAA}(\xi,\HH) \boxtimes m_J \mapsto m_I^{\vee} \boxtimes f_{0,0}(c_{\AAA}(\xi,\HH)) \boxtimes m_J
  \]
  More explicitly, the above Heegaard moves induce Heegaard moves between the sutured Heegaard diagrams $C(\HH)$ and $C(\HH')$, inducing the (natural) chain homotopy equivalence $F\colon \SFC(C(\HH)) \to \SFC(C(\HH'))$ from \cite{JTZ:naturality} (see Section~\ref{subsec:invariant}).
 
  Hence, under the identification $\iota_I \cdot H(\BSA(\HH))\cdot \iota_J \cong \SFH(C(\HH))$, the induced chain homotopy equivalence $f_{0,0}$ is the same as $F$. As before, naturality of $\SFC$ implies that $c(\xi,\overline{\HH}) \mapsto c(\xi,\overline{\HH'})$ under the map $F$ (up to a boundary), and so $f(c_{\AAA}(\xi,\HH)) = c_{\AAA}(\xi,\HH')$ again up to a boundary. Hence, the equivalence class of $c_{\AAA}$ is independent of the choice of a Heegaard diagram.

  From $c_{\AAA}(\xi)$, we can define other flavors as follows: 
  \begin{align*}
    &c_{\DA}(\xi) := \iota_I^{\vee} \boxtimes c_{\AAA}(\xi)\\
    &c_{\AD}(\xi) := c_{\AAA}(\xi) \boxtimes \iota_J^{\vee}\\
    &c_{\DD}(\xi) := \iota_I^{\vee} \boxtimes c_{\AAA}(\xi) \boxtimes \iota_J^{\vee}
  \end{align*}
  where $\iota_I^{\vee} \in \BSDD(\AZ(\F_1))$ and $\iota_J^{\vee} \in \BSDD(\AZ(\F_2))$

  To finish the proof of Theorem \ref{thm:main}, we remark that these other flavors defined above are well-defined  cycles, since $\iota^{\vee}_{I^c}$ is a cycle in $\BSDD(\AZ).$
\end{proof}

\begin{remark}
  In the above, the element $\iota_I^{\vee}$ belongs to the \emph{specific} model of $\BSDD(\TW_+)$ corresponding to the $\AZ$ diagram. In principle, the choice of diagram should not make a difference, but to sidestep potential issues of naturality, we define it in this way.
\end{remark}

\begin{remark}\label{rmk:friendly}
  Alishahi, F\"oldv\'ari, Hendricks, Licata, Petkova, and V\'ertesi showed \cite[Theorem 3]{AFHLPV:borderedInvariants} that their type-$A$ invariant (also denoted $c_A$) satisfies $c_A(Y,\xi,\F) \cdot \iota_+ = EH(Y,\Gamma(\F),\xi),$ where $\iota_+$ is the idempotent in $\A$ corresponding to the elementary cap. Thus, by definition, their $c_A$ agrees with the $c_A$ we defined above. However, their $c_D$ is defined from the same kind of diagram as $c_A,$ unlike in our case, where we attach the $\AZ$ diagram corresponding to a twisting slice. As a result, their invariants are only defined on a more restricted set of bordered sutured manifolds. We suspect that in this restricted context, their $c_D$ is, in fact, equivalent to ours.
\end{remark}

\subsection{Pairing bordered contact invariants}

Our goal now is to prove Theorem \ref{thm:pairing}. The Theorem will follow from the following:

\begin{lemma}\label{lem:pairing}
  Let $\Y_1=(Y_1, \Gamma_1, \F)$ and $\Y_2=(Y_2, \Gamma_2, -\F)$ be bordered sutured manifolds, and suppose $\xi_1$ and $\xi_2$ are contact structures on $C(\Y_1)=(Y_1,\Gamma_1 \cup \Gamma_I)$ and $C(\Y_2)=(Y_2,\overline{\Gamma}_I\cup \Gamma_2)$, respectively, such that $\Gamma_I$ is an elementary dividing set for $I\subset \{1,\ldots,n\}$. Then we have 
  \[
    [c_A(\xi_1) \boxtimes \iota_{I^c}^{\vee}\boxtimes c_A(\xi_2)]  = \EH(\xi \Cup \xi_2) \in \SFH(C(\Y_1)\Cup C(\Y_2)), 
  \]
  where $\iota_{I^c}^{\vee}\in\BSDD(\AZ(\F))$.
\end{lemma}

\begin{proof}
  Let us remark first that the homology class of $c_A(\xi_1) \boxtimes \iota_{I^c}^{\vee}\boxtimes c_A(\xi_2)$ is indeed independent of any diagrammatic choices. As above, this follows from the fact that $\iota_{I^c}^{\vee}$ is a cycle in $\BSDD(\AZ(\F)),$ so that any type-$A$ isomorphism mapping $c_A(\xi_1)$ to $c'_A(\xi_1)$ will induce an isomorphism mapping $c_A(\xi_1) \boxtimes \iota_{I^c}^{\vee}\boxtimes c_A(\xi_2)$ to $c'_A(\xi_1) \boxtimes \iota_{I^c}^{\vee}\boxtimes c_A(\xi_2)$.

  We shall therefore prove the lemma by computing the contact class in $C(\Y_1)\Cup C(Y_2)$ in two ways, both resulting in the same Heegaard diagram.

  First, consider a contact handle decomposition of $-C(\Y_2)$ such that contact 1- and 2-handles are attached to an $I$-invariant contact structure $\xi_I$ on $(F\times [0,1], \Gamma_{I\to I}, \xi_I)$. Here $(F\times[0,1], \Gamma_{I\to I})$ can be considered as the double $\mathcal{D}(\C_I) = -\C_I \cup \F \times [0,1] \cup \C_I$. Thus, $\Y_2$ itself may be expressed as $\F \times [0,1]\cup \C_I$ with contact handles attached away from the bordered surface $\F.$ We may choose the convenient (bordered sutured) Heegaard diagram for $\F \times [0,1] \cup \C_I$ obtained as the union of $\overline{AZ}(\F)$ with the Heegaard diagram $\HH_I$ for the cap $\C_I$.  Thus, starting with the generator representing the contact class $c_A(\xi_I) \in \BSA(\F \times [0,1] \cup \C_I),$ repeatedly applying the diagrammatic handle attachments of Juh\'asz-Zemke \cite{JZ:contact_handle} (Theorem~\ref{thm:digramGluing}), we obtain a Heegaard diagram $\HH_{2}$ for $-\Y_2$ by applying Heegaard diagramatic handles $h_1,\dots,h_n$ to $\HH_C,$ together with a representative of the contact class $c_A(\xi_2) \in \BSA(-\Y_2).$
    
  \textbf{Claim 1:} The contact class $c(\xi_I) \in \SFC(-\mathcal{D}(\C_I))$ is represented by the diagonal class $m_I^{vee}\boxtimes \iota_I \boxtimes m_I.$

  proof: By Proposition~\ref{prop:contactAZ}, we know that $a \in \A(\Z)$ is $c_{\AAA}(\xi_a)$. An idempotent $\iota_I$ represents $I$-invariant contact structure, so the isomorphism sends $\iota \to m_I^{\vee}\boxtimes \iota_I \boxtimes m_I = c(\xi_I)$. 

  Therefore, taking any Heegaard diagram $\HH_{1}$ for $-\Y_1,$ we may obtain a Heegaard diagram representing $-Y_1 \cup \TW^+ \cup -Y_2$ by the union: $\HH_{1} \cup \AZ \cup \HH_{2}$ and a generator representing $c_A(\xi_1)\boxtimes \iota_{I^c}^{\vee} \boxtimes c_A(\xi_2).$ See Figure \ref{fig:Pairing_HD}.
    
  We now wish to argue that this generator in this same Heegaard diagram represents the contact class $c(\xi_1\Cup\xi_2).$ First, we notice that $\HH_{Y_1}\cup \AZ \cup \overline{AZ} \cup \HH_C$ is a Heegaard diagram representing $-C(\Y_1).$ This is because $\AZ \cup \overline{\AZ}$ represents $\TW^+\cup\TW^-,$ which is simply $F \times I$ with vertical dividing curves without twisting.

  \textbf{Claim 2:} $c_A(\xi_1)\boxtimes \iota^{\vee}_{I^c} \boxtimes \iota_I \boxtimes m_I = c_A(\xi_1)\boxtimes m_I$

  To see this, we observe that $\TW^+\cup\TW^- \cup \C_I$ is topologically the same as $-\C_I$, and thus $\AZ \cup \overline{AZ}\cup \HH_I$ represents the same bordered sutured manifold as $\HH_I.$ Thus, there is a sequence of Heegaard moves taking one to the other, which induces a type-$D$ homotopy equivalence $f:\BSD(\HH_I)\to \A \otimes \BSD\left(\AZ \cup \overline{AZ}\cup \HH_I\right).$ On the other hand, by \cite[Lemma~5.6]{Zarev:joining}, there is a type-$D$ homotopy equivalence $g:\BSD\left(\AZ \cup \overline{AZ}\cup \HH_I\right) \to \A \otimes \BSD(\HH_I)$ which maps $\iota^{\vee}_{I^c} \boxtimes \iota_I \boxtimes m_I$ to $m_I \otimes 1.$ The composition of $f$ and $g$ is an automorphism of $\BSD(\HH_I),$ and this module has a unique generator and no differential, so that there is a unique automorphism sending $m_I$ to $m_I \otimes 1.$ (This can be seen purely algebraically: otherwise, one could find type-$A$ modules that would pair in inequivalent ways before versus after any other kind of morphism). It follows that $g$ is homotopy inverse of $f$ and thus is homotopic to the map obtained by performing the Heegaard moves in the reverse order. These Heegaard moves induce the natural map $Id \boxtimes g: \BSA(\HH_{1})\boxtimes\BSD\left(\AZ \cup \overline{AZ}\cup \HH_I\right) \to \BSA(\HH_{1})\boxtimes\BSD(\HH_I)$  (recalling that both of these may be identified with $\SFC(-C(\Y_1)).$ Therefore, by definition of $c_A,$ the claim follows.

  Now we may finish the proof of the Lemma. Since $c_A(\xi_1)\boxtimes \iota^{\vee}_{I^c} \boxtimes \iota_I \boxtimes m_I$ represents the contact class, we again appeal to the diagrammatic handle attachments of \cite{JZ:contact_handle} to see that the resulting generator in $\HH_{1} \cup \AZ \cup \HH_{2}$ represents the contact class $c(Y_1\Cup Y_2),$ but we have seen above that this element is precisely $c_A(\xi_1)\boxtimes \iota_{I^c}^{\vee} \boxtimes c_A(\xi_2),$ as required.
\end{proof}

\begin{figure}
  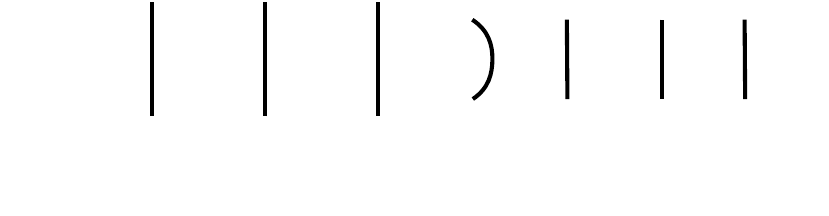
  \caption{A Heegaard diagram for $-\Y_1 \cup \TW^+ \cup -\Y_2$}
  \label{fig:Pairing_HD}
\end{figure}

\subsection{\texorpdfstring{$\Aoo$}{Aoo} operations and bypass attachments}
Here, we prove 

\bypass*

\begin{proof}
  Let $a = a(S,T,\phi)$ be a cycle representing an element in $H_*(\iota_I \cdot \A(\Z) \cdot \iota_J)$ and $\xi_a$ the corresponding contact structure on $(\Sigma \times [0,1], \Gamma_{I\to J})$. According to Proposition~\ref{prop:AZ}, there is a canonical identification of $H_*\left(\BSAA(\overline{\AZ(\Z)})\right)$ with $H_*\left(\A(\Z)\right),$ and hence we may find a corresponding cycle (which we also denote by $a$) in $\BSAA\left(\overline{\AZ(\Z)}\right).$ Indeed, $\overline{AZ}(\Z)$ is a Heegaard diagram for $\TW^-_{\Sigma}$ and the corresponding sutured contact invariant for $\xi_a$ will be $m_I \boxtimes a \boxtimes m_J \in \BSD(\HH_I)\boxtimes \BSAA\left(\overline{\AZ(\Z)}\right) \boxtimes \BSD(\HH_J)$ by Proposition~\ref{prop:contactAZ}. By the pairing theorem (Theorem \ref{thm:pairing}), attaching $(\Sigma \times [0,1], \xi_a)$ to $(\Y,\xi)$ will result in the contact invariant satisfying: 
  \[
    c_A(\xi \Cup \xi_a) \boxtimes m_J =c_A(\Y,\xi)\boxtimes \iota^{\vee}_{I^c} \boxtimes a \boxtimes m_J.
  \]
  As in the proof of Theorem \ref{thm:pairing}, we apply  \cite[Lemma~5.6]{Zarev:joining}, to obtain a  type-$D$ homotopy equivalence $g:\BSD(\AZ \cup \overline{AZ}\cup \HH_I) \to \A\otimes \BSD(\HH_I)$ which maps $\iota^{\vee}_{I^c} \boxtimes a \boxtimes m_I$ to $a \otimes m_I.$ We have seen above that this map is induced by Heegaard moves, and so, pairing with $\BSA(\Y)$, we obtain the natural map:
  \begin{align*}
    Id\boxtimes g: \BSA(\Y) \boxtimes \BSDD(\AZ(\Z))\boxtimes \BSAA\left(\overline{\AZ}(\Z)\right)\boxtimes \BSD(\HH_J) \to \BSA(\Y)\boxtimes \BSD(\HH_J) 
  \end{align*}
  which maps $c_A(\Y,\xi)\boxtimes \iota^{\vee}_{I^c} \boxtimes a \boxtimes m_J \mapsto m_2(c_A(\xi),a) \boxtimes m_J$. Hence, by definition, we must have that $c_A(\xi \Cup \xi_a) = m_2(c_A(\xi),a),$ as desired.
\end{proof}

\begin{remark}\label{rmk:reattaching}
  This theorem allows us to see geometric reattachment in the bordered algebra. More precisely, suppose we have two contact manifolds  $(\Y_1,\xi_1)$ and $(\Y_2,\xi_2)$ with the same boundary $\Sigma.$ Let us call their corresponding type-$A$ contact invariants $c_1$ and $c_2.$ Suppose further that $\Y_1 \Cup \Sigma \times '[0,1] \Cup \Y_2$ is a permissible contact gluing for some contact structure $\xi_a$ on $\Sigma \times [0,1].$ According to the theorem we may compute the contact invariant of the glued manifold in two different ways, depending on which of $\Y_1$ and $\Y_2$ receives the bypass attachment. They are: $c_1 \boxtimes \iota^{\vee} \boxtimes m_2(c_2,\rho_a)$ and $m_2(c_1,\rho_a) \boxtimes \iota^{\vee} \boxtimes c_2$. But these two generators are homologous since:
  \[
    \partial (c_1 \boxtimes \rho_a^{\vee} \boxtimes c_2) = c_1 \boxtimes \iota^{\vee} \boxtimes m_2(c_2,\rho_a) + m_2(c_1,\rho_a) \boxtimes \iota^{\vee} \boxtimes c_2
  \]
\end{remark}

\section{Torus boundary} \label{sec:torus}

Let $(Y,\Gamma_D, \F)$ be a bordered sutured manifold with torus boundary where $\F = (T^2_{\bullet}, \Lambda)$ is a punctured torus which is parametrized by $\Z$ shown in Figure~\ref{fig:parametrizedTorus} and $\Gamma_D$ is a dividing set consisting a single dividing curve on a disk $D$. For $I \subset \{0,1\}$, we take an elementary cap $\C_I$ and obtain a sutured manifold by attaching it to  $\Y$:
\[
  \Y \cup \C_I = (Y, \Gamma_D \cup \Gamma_I).
\]

If $I = \varnothing$ or $\{0,1\}$, then $\Gamma_D \cup \Gamma_I$ contains a contractible dividing curve, so any contact structure on $Y \cup \C_I$ is overtwisted. If $I = \{0\}$ or $\{1\}$, then $\Gamma$ consists of two homologically essential closed curves. We will denote 
\[
  \Gamma_0 := \Gamma_D \cup \Gamma_{\{0\}} \quad \Gamma_1 := \Gamma_D \cup \Gamma_{\{1\}}.
\]
According to Theorem~\ref{thm:CFA=SFC}, there exist natural homotopy equivalences:
\begin{align*}
  CFA(-\Y)\cdot \iota_0 \simeq \SFC(-Y,-\Gamma_0),\\
  CFA(-\Y)\cdot \iota_1 \simeq \SFC(-Y,-\Gamma_1)
\end{align*}
and they map $c_A(\xi)$ to $c(\xi)$, where $c(\xi)$ is a sutured contact class such that $[c(\xi)] = \EH(\xi)$.  
Recall that the torus algebra $\A(\Z,0)$ is given by
\[
  \xymatrix{
    \iota_0\bullet\ar@/^1pc/[r]^{\rho_1}\ar@/_1pc/[r]_{\rho_3} & \bullet\iota_1\ar[l]_{\rho_2}
  }/(\rho_2\rho_1=\rho_3\rho_2=0)
\]
with $\rho_{12}=\rho_1\rho_2$, $\rho_{23}=\rho_2\rho_3$ and $\rho_{123}=\rho_1\rho_2\rho_3$. Thus $\{\iota_0,\iota_1,\rho_1,\rho_2,\rho_3,\rho_{12},\rho_{e3},\rho_{123}\}$ is a basis for $\A(\Z,0)$.

In this section, we compute the modules $\BSDD(\AZ)$ and for $\TW_{\F}^{+}$. Then using this, we calculate the bordered contact invariants of tight contact structures on integer framed solid tori. We also explicitly describe the contact structures corresponding to the generators of the torus algebra and prove Theorem~\ref{thm:contactTorusAlg}.

\subsection{A model for \texorpdfstring{$\BSDD(\AZ)$}{BSDD(AZ)} for the torus}\label{sec:DD_torus}
Here, we compute $\BSDD(\TW^{+}_{\F})$ using the Heegaard diagram $\AZ$ shown in Figure \ref{fig:DD_id}. We shall only be interested in the generators of $\BSDD$ which arise as intersection points between exactly one $\alpha$-arc and one $\beta$-arc, which are the generators corresponding to the idempotents $\iota_0$ and $\iota_1$. Labelling the generators as in the figure, we see that they satisfy:

\begin{align*}
  \delta^1\left(\rho_{123}^{\vee}\right) &= \rho_3 \otimes \rho_{12}^{\vee} + \rho_{23}^{\vee} \otimes \rho_1, & \delta^1\left(\rho_{23}^{\vee}\right) &= \rho_3 \otimes \rho_{2}^{\vee} + \rho_{3}^{\vee} \otimes \rho_2, \\
  \delta^1\left(\rho_{12}^{\vee}\right) &= \rho_2 \otimes \rho_{1}^{\vee} + \rho_{2}^{\vee} \otimes \rho_1, & \delta^1\left(\rho_{3}^{\vee}\right) &= \rho_3 \otimes \iota_0^{\vee} + \iota_1^{\vee} \otimes \rho_3, \\
  \delta^1\left(\rho_{1}^{\vee}\right) &= \rho_1 \otimes \iota_0^{\vee} + \iota_1^{\vee} \otimes \rho_1, & \delta^1\left(\rho_{2}^{\vee}\right) &= \rho_2 \otimes \iota_1^{\vee} + \iota_0^{\vee} \otimes \rho_2.
\end{align*}

\begin{figure}
  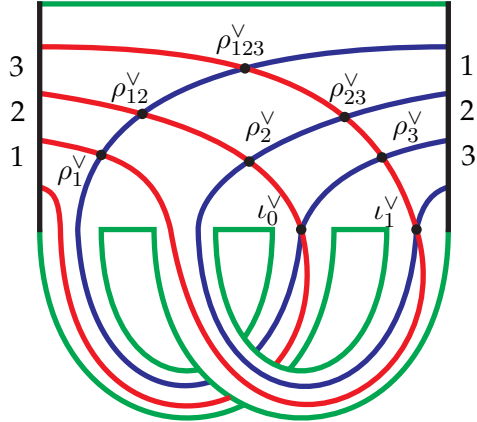
  \caption{The $\AZ$ diagram for $\TW^+_{\F}$ in the case of $\F = (T^2_{\bullet}, \Lambda).$ The intersections between $\alpha$- and $\beta$-arcs are labelled.}
  \label{fig:DD_id}
\end{figure}

\subsection{Basic slices and \texorpdfstring{$\Aoo$}{Aoo} operations} \label{subsec:Aoo}
We first prove Theorem~\ref{thm:contactTorusAlg}. According to Theorem~\ref{thm:bypass}, a generator $a$ of $\A(\Z)$ corresponds to a contact structure $\xi_a$ on $(\Sigma\times[0,1],\Gamma_{I\to J})$. We show that each $\xi_a$ is a (union of) basic slice.

\begin{proof}[Proof of Theorem~\ref{thm:contactTorusAlg}] 
  It is sufficient to show the statement for $\rho_1$, $\rho_2$ and $\rho_3$. Since none of $\rho_{ij} = \rho_i \rho_j$ for $i<j \in \{1,2,3\}$ is vanishing in the homology, contact structures $\xi_{\rho_1}$, $\xi_{\rho_2}$ and $\xi_{\rho_3}$ are minimally twisting and hence basic slices. Thus it remains to determine the sign of each basic slice. 

  To do so, we will make use of the Legendrian right-handed trefoil $K$ with $\tb=1$ in the standard contact 3-sphere. Let $(Y(K),\xi)$ be the complement of a standard neighborhood of $K$. Also, let $\xi_{\pm}$ denotes the basic slices $B_{\pm}(\infty,1)$. According to \cite{SV:LOSS} we know 
  \begin{align*}
    \Phi(\EH(\xi), EH(\xi_-)) &= \Lhat(K),\\ 
    \Phi(\EH(\xi), EH(\xi_+)) &= \Lhat^*(K).
  \end{align*}
  under the identification of $SFH(-Y,-\Gamma_{\mu})$ and $\HFKhat(-Y,K)$. Here, $\Lhat^*(K)$ is the sutured contact invariant of $\xi \cup \xi_+$, which corresponds to the conjugate contact structure of $\xi\cup \xi_-.$ 
  According to \cite{LOSS:LOSS}, we know $\Lhat(K)$ and $\Lhat^*(K)$ are nonvanishing.

  Figure \ref{fig:RHT_A1} shows the type-$A$ structure for a $-1$-framed left-handed trefoil, where the generators in the  meridional idempotents are shown as filled circles. By computing the spin$^c$ (or Alexander) gradings, we can locate $\Lhat(K)$ and  $\Lhat^*(K)$, as shown in the figure. We see that of these two generators, only $\Lhat$ is of the form $m_2(x,\rho_2)$ for some $x\in\BSA(-Y(K)).$ Hence, $\rho_2$ must correspond to a negative basic slice. 

  The argument for $\rho_1$ and $\rho_3$ is similar. Figure \ref{fig:RHT_A2} shows the type-$A$ module for the ``rotated'' parametrization of $-Y(K)$ (interchanging the roles of meridian and longitude). In this case, we see that $\Lhat$ can be expressed as $m_2(x,\rho_3)$ but not as $m_2(x,\rho_1).$ The exact opposite is true for $\Lhat^*,$ and so we conclude that $\rho_3$ is a negative basic slice, whereas $\rho_1$ is a positive basic slice.
\end{proof}

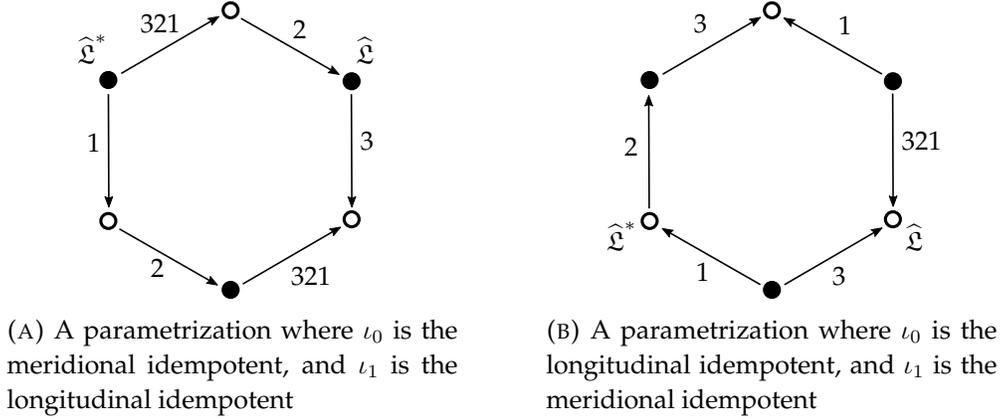
\begin{figure}
  \centering
  \small
  \begin{subfigure}{.4\textwidth}
    \centering
\begingroup%
  \makeatletter%
  \providecommand\color[2][]{%
    \errmessage{(Inkscape) Color is used for the text in Inkscape, but the package 'color.sty' is not loaded}%
    \renewcommand\color[2][]{}%
  }%
  \providecommand\transparent[1]{%
    \errmessage{(Inkscape) Transparency is used (non-zero) for the text in Inkscape, but the package 'transparent.sty' is not loaded}%
    \renewcommand\transparent[1]{}%
  }%
  \providecommand\rotatebox[2]{#2}%
  \newcommand*\fsize{\dimexpr\f@size pt\relax}%
  \newcommand*\lineheight[1]{\fontsize{\fsize}{#1\fsize}\selectfont}%
  \ifx\svgwidth\undefined%
    \setlength{\unitlength}{116.29399217bp}%
    \ifx\svgscale\undefined%
      \relax%
    \else%
      \setlength{\unitlength}{\unitlength * \real{\svgscale}}%
    \fi%
  \else%
    \setlength{\unitlength}{\svgwidth}%
  \fi%
  \global\let\svgwidth\undefined%
  \global\let\svgscale\undefined%
  \makeatother%
  \begin{picture}(1,0.96509518)%
    \lineheight{1}%
    \setlength\tabcolsep{0pt}%
    \put(0,0){\includegraphics[width=\unitlength,page=1]{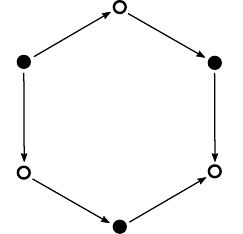}}%
    \put(0.23638354,0.07000909){\color[rgb]{0,0,0}\makebox(0,0)[lt]{\lineheight{1.25}\smash{\begin{tabular}[t]{l}2\end{tabular}}}}%
    \put(0.69105003,0.04661958){\color[rgb]{0,0,0}\makebox(0,0)[lt]{\lineheight{1.25}\smash{\begin{tabular}[t]{l}321\end{tabular}}}}%
    \put(0.91553497,0.48419097){\color[rgb]{0,0,0}\makebox(0,0)[lt]{\lineheight{1.25}\smash{\begin{tabular}[t]{l}3\end{tabular}}}}%
    \put(0.02986671,0.47871408){\color[rgb]{0,0,0}\makebox(0,0)[lt]{\lineheight{1.25}\smash{\begin{tabular}[t]{l}1\end{tabular}}}}%
    \put(0.20275806,0.86469204){\color[rgb]{0,0,0}\makebox(0,0)[lt]{\lineheight{1.25}\smash{\begin{tabular}[t]{l}321\end{tabular}}}}%
    \put(0.69696551,0.84604613){\color[rgb]{0,0,0}\makebox(0,0)[lt]{\lineheight{1.25}\smash{\begin{tabular}[t]{l}2\end{tabular}}}}%
    \put(0.90093179,0.76593711){\color[rgb]{0,0,0}\makebox(0,0)[lt]{\lineheight{1.25}\smash{\begin{tabular}[t]{l}$\Lhat$\end{tabular}}}}%
    \put(-0.00195238,0.76593711){\color[rgb]{0,0,0}\makebox(0,0)[lt]{\lineheight{1.25}\smash{\begin{tabular}[t]{l}$\Lhat^*$\end{tabular}}}}%
  \end{picture}%
\endgroup%

    \caption{A parametrization where $\iota_0$ is the meridional idempotent, and $\iota_1$ is the longitudinal idempotent}
    \label{fig:RHT_A1}
  \end{subfigure}
  \hspace{1cm} 
  \begin{subfigure}{.4\textwidth}
    \centering
\begingroup%
  \makeatletter%
  \providecommand\color[2][]{%
    \errmessage{(Inkscape) Color is used for the text in Inkscape, but the package 'color.sty' is not loaded}%
    \renewcommand\color[2][]{}%
  }%
  \providecommand\transparent[1]{%
    \errmessage{(Inkscape) Transparency is used (non-zero) for the text in Inkscape, but the package 'transparent.sty' is not loaded}%
    \renewcommand\transparent[1]{}%
  }%
  \providecommand\rotatebox[2]{#2}%
  \newcommand*\fsize{\dimexpr\f@size pt\relax}%
  \newcommand*\lineheight[1]{\fontsize{\fsize}{#1\fsize}\selectfont}%
  \ifx\svgwidth\undefined%
    \setlength{\unitlength}{123.79394531bp}%
    \ifx\svgscale\undefined%
      \relax%
    \else%
      \setlength{\unitlength}{\unitlength * \real{\svgscale}}%
    \fi%
  \else%
    \setlength{\unitlength}{\svgwidth}%
  \fi%
  \global\let\svgwidth\undefined%
  \global\let\svgscale\undefined%
  \makeatother%
  \begin{picture}(1,0.90662569)%
    \lineheight{1}%
    \setlength\tabcolsep{0pt}%
    \put(0,0){\includegraphics[width=\unitlength,page=1]{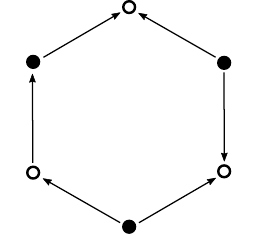}}%
    \put(0.27053033,0.05365064){\color[rgb]{0,0,0}\makebox(0,0)[lt]{\lineheight{1.25}\smash{\begin{tabular}[t]{l}1\end{tabular}}}}%
    \put(0.68553428,0.04379513){\color[rgb]{0,0,0}\makebox(0,0)[lt]{\lineheight{1.25}\smash{\begin{tabular}[t]{l}3\end{tabular}}}}%
    \put(0.89641899,0.45485662){\color[rgb]{0,0,0}\makebox(0,0)[lt]{\lineheight{1.25}\smash{\begin{tabular}[t]{l}321\end{tabular}}}}%
    \put(0.05229123,0.43759458){\color[rgb]{0,0,0}\makebox(0,0)[lt]{\lineheight{1.25}\smash{\begin{tabular}[t]{l}2\end{tabular}}}}%
    \put(0.2631758,0.81230535){\color[rgb]{0,0,0}\makebox(0,0)[lt]{\lineheight{1.25}\smash{\begin{tabular}[t]{l}3\end{tabular}}}}%
    \put(0.70320819,0.80690605){\color[rgb]{0,0,0}\makebox(0,0)[lt]{\lineheight{1.25}\smash{\begin{tabular}[t]{l}1\end{tabular}}}}%
    \put(0.90693416,0.15003857){\color[rgb]{0,0,0}\makebox(0,0)[lt]{\lineheight{1.25}\smash{\begin{tabular}[t]{l}$\Lhat$\end{tabular}}}}%
    \put(-0.00183403,0.15003857){\color[rgb]{0,0,0}\makebox(0,0)[lt]{\lineheight{1.25}\smash{\begin{tabular}[t]{l}$\Lhat^*$\end{tabular}}}}%
  \end{picture}%
\endgroup%

    \caption{A parametrization where $\iota_0$ is the longitudinal idempotent, and $\iota_1$ is the meridional idempotent}
    \label{fig:RHT_A2}
  \end{subfigure}
  \caption{Type-$A$ modules for the $-1$-framed left-handed trefoil complement}
  \label{fig:RHT_typeA}
\end{figure}

\subsection{Solid tori}\label{subsec:solid}
Here, we compute the type-$A$ and type-$D$ contact invariants of some solid tori parametrized by $\beta$-type diagrams, which will be used for integer-framed surgery calculations in Section~\ref{sec:applications}. Let $S = S^1 \times D^2$ a solid torus, $\mu$ a meridian of $S$ and $\lambda$ a preferred longitude of $S$. Now we consider a bordered manifold $\SSS = (S, \F)$ where $\F = (\partial S, \Lambda)$ is parametrized by the arc diagram $\Z$ in Figure~\ref{fig:parametrizedTorus}.(b) such that $(\beta_0 = \lambda,\beta_1 = n\lambda + \mu)$ or $(\beta_0 = n\lambda + \mu,\beta_1 = \lambda)$ for some $n \in \mathbb{Z}\setminus\{0\}$. The type-$A$ module for $\SSS$ is generally well-known, but since we are interested in invariants coming from $\beta$-type diagrams, we compute them from Heegaard diagrams in Figure~\ref{fig:HD-solid-tori}. There are four cases, depending on orientation and choice of parametrization.

\begin{figure}
  \centering
  \begin{subfigure}{.45\textwidth}
    \centering
    \footnotesize
    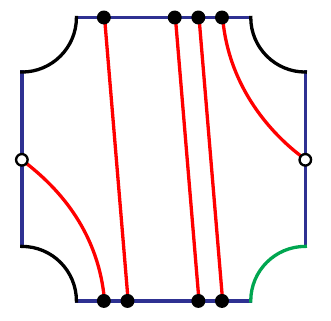
    \caption{$n>0$}
    \label{fig:HD-solid-tori.a}
  \end{subfigure}
  \hfill
  \begin{subfigure}{.45\textwidth}
    \centering
    \footnotesize
    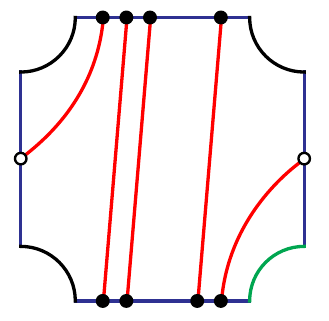
    \caption{$n<0$}
    \label{fig:HD-solid-tori.b}
  \end{subfigure}\\
  \vspace{0.5cm}
  \begin{subfigure}{.45\textwidth}
    \centering
    \footnotesize
    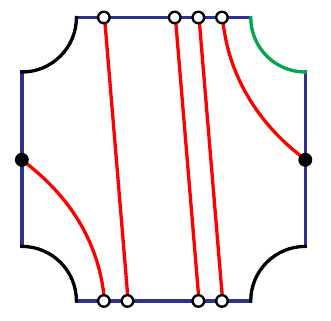
    \caption{Another parametrization for $n>0$}
    \label{fig:HD-solid-tori.c}
  \end{subfigure}
  \hfill
  \begin{subfigure}{.45\textwidth}
    \centering
    \footnotesize
    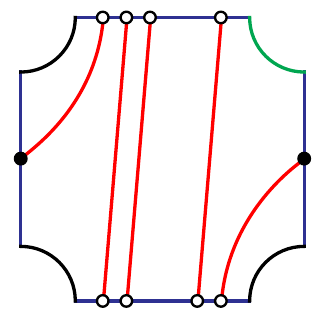
    \caption{Another parametrization for $n<0$}
    \label{fig:HD-solid-tori.d}
  \end{subfigure}
  \caption{Heegaard diagrams for various $n$-framed solid tori}
  \label{fig:HD-solid-tori}
\end{figure}

From these Heegaard diagrams, we compute the type-$A$ modules as shown in Figure~\ref{fig:A-solid-tori}. We see that the modules have $n+1$ generators: in one of the idempotents, there is a single generator, which we call $y,$ whereas the other idempotent has the remaining generators, which we label $x_1,\dots,x_n$. Notice also that all generators are in distinct spin$^c$-structures. 

\begin{figure}
  \centering
  \begin{subfigure}{.45\textwidth}
    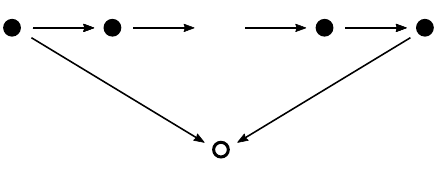
    \caption{$n>0$, every basis element represents a contact invariant for a  negatively framed solid torus.}
    \label{fig:A-solid-tori.a}
  \end{subfigure}
  \hfill
  \begin{subfigure}{.45\textwidth}
    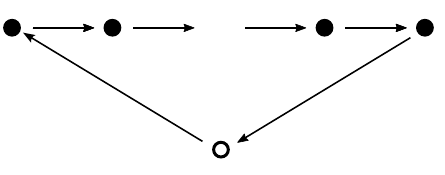
    \caption{$n<0$, only $x_1$, $x_2$ and $y$ represent contact invariants for a positively framed solid torus.}
    \label{fig:A-solid-tori.b}
  \end{subfigure}\\
  \vspace{0.5cm}
  \begin{subfigure}{.45\textwidth}
    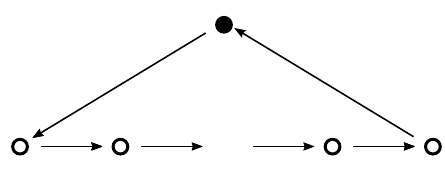
    \caption{Another parametrization for $n>0$}
    \label{fig:A-solid-tori.c}
  \end{subfigure}
  \hfill
  \begin{subfigure}{.45\textwidth}
    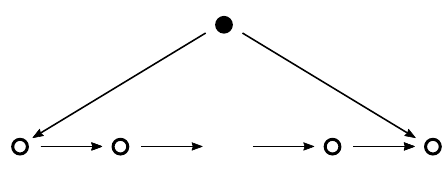
    \caption{Another parametrization for $n<0$}
    \label{fig:A-solid-tori.d}
  \end{subfigure}
  \caption{Directed graphs representing the type-$A$ modules for the solid tori in Figure \ref{fig:HD-solid-tori}}
  \label{fig:A-solid-tori}
\end{figure}

As discussed at the beginning of the section, we have 
\begin{align*}
  \BSA(-\SSS)\cdot \iota_0 \simeq \SFC(-S,-\Gamma_0),\\
  \BSA(-\SSS)\cdot \iota_1 \simeq \SFC(-S,-\Gamma_1).
\end{align*}
where $\Gamma_0$ consists of two dividing curves parallel to $\beta_0$ and $\Gamma_1$ consists of two dividing curves parallel to $\beta_1$.

We fix the parametrization by assuming that $\beta_0 = \lambda$ and $\beta_1 = n\lambda + \mu$. This corresponds to Figure~\ref{fig:HD-solid-tori.a}, \ref{fig:HD-solid-tori.b} and Figure~\ref{fig:A-solid-tori.a}, \ref{fig:A-solid-tori.b}. Then there exists a unique tight contact structure on $(S,\Gamma_0)$ by Theorem~\ref{thm:solid-torus}. We denote this contact structure by $\xi_0$. Also, according to Proposition~\ref{prop:tight-integer-framed}, there exist two tight contact structures on $(S,\Gamma_1)$ if $n > 1$, obtained by attaching a basic slice to $(S,\Gamma_0)$. We denote these contact structures $\xi_+$ and $\xi_-$ depending on the sign of the basic slice. There exists a unique tight contact structure if $n=1$. There are $|n|$ tight contact structures if $n < 0$ as it contains a continued fraction block with length $n-1$. We denote these contact structures by $\xi_1, \ldots, \xi_n$. Also suppose that the sign of basic slices in the continued fraction block is all negative for $\xi_1$ and all positive for $\xi_n$. 

Since $\CFA(-\SSS) \cdot \iota_0$ contains the unique generator $y$, it is the contact invariant $c_A(\xi_0)$. $\CFA(-\SSS) \cdot \iota_1$ contains $|n|$ distinct elements. When $n<0$, $c_A(\xi_i)$ corresponds to $x_i$. When $n>0$, we know $c_A(\xi_-) = x_1$ and $c_A(\xi_+) =x_n$. See Figure~\ref{fig:A-solid-tori.a} and \ref{fig:A-solid-tori.b}. The same analysis works for the other parametrization $\beta_0 = n\lambda + \mu$ and $\beta_1 = \mu$. This corresponds to Figure~\ref{fig:A-solid-tori.c} and \ref{fig:A-solid-tori.d}.

Now we can compute type-$D$ invariants by pairing the type-$A$ invariants with the $\DD$-bimodule $\BSDD(\AZ)$ for twisting slice $\TW^+$ from Section~\ref{sec:DD_torus}

\begin{figure}
  \centering
  \begin{subfigure}{\textwidth}
    \centering
    \scriptsize
    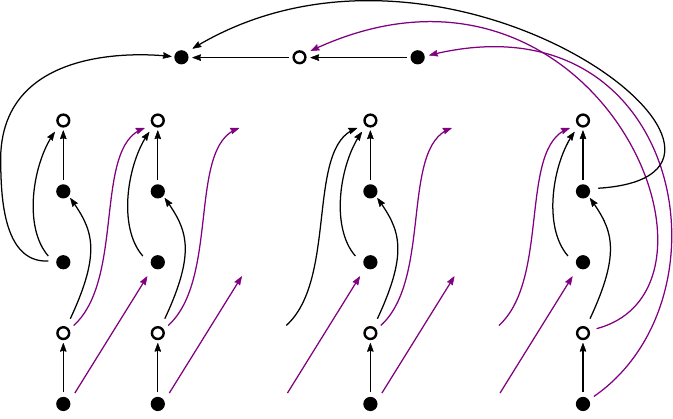
    \caption{$n>0$}\label{fig:test}
  \end{subfigure}\\
  \begin{subfigure}{\textwidth}
    \centering
    \scriptsize
    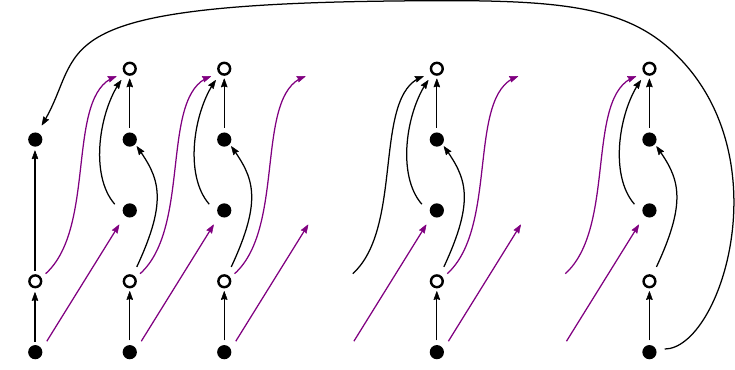
    \caption{$n<0$}
  \end{subfigure}
  \caption{Directed graphs representing the type-$D$ structures for solid tori obtained by pairing with the twisting slice}
  \label{fig:D-solid-tori}
\end{figure}

\begin{figure}\ContinuedFloat
  \centering
  \begin{subfigure}{\textwidth}
    \centering
    \scriptsize
    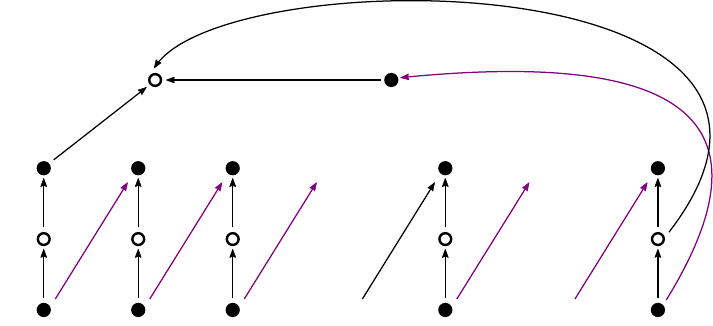
    \caption{Another parametrization for $n>0$}
  \end{subfigure}\\
  \begin{subfigure}{\textwidth}
    \centering
    \scriptsize
    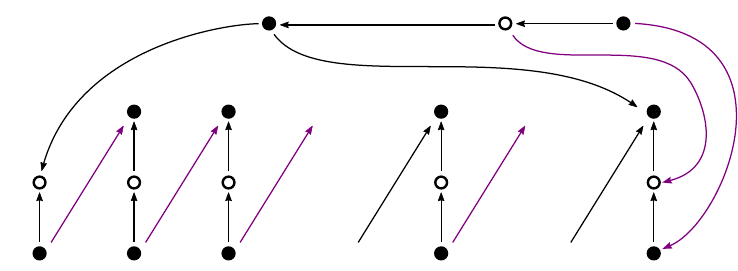
    \caption{Another parametrization for $n<0$}
  \end{subfigure}
  \caption{(continued)}
\end{figure}

Suppose we are interested in the contact invariant corresponding to $x_j$ on the type-$A$ side. Then the corresponding type-$D$ contact invariant is either $\iota_0^{\vee} \boxtimes x_j$ or $\iota_1^{\vee} \boxtimes x_j$ depending on the parametrization. Figure \ref{fig:D-solid-tori} shows the type-$D$ structures which arise from pairing with the twisting slice bimodule from Section \ref{sec:DD_torus}. We omit the box tensor from the notation of the generators to save space. Contracting the indicated arrows, we simplify the diagram so that it is nearly as reduced as possible, while still preserving our preferred generator. We also keep $\iota_1^{\vee}\boxtimes y$ (or $\iota_0^{\vee}\boxtimes y,$ respectively) in case that generator is of interest to us (that generator is the invariant of the unique tight contact structure with longitudinal suture).
Figure \ref{fig:D-solid-tori-simp} shows the resulting simplified type-$D$ structures.

\begin{figure}
  \centering
  \begin{subfigure}{.45\textwidth}
    \centering
    \footnotesize
    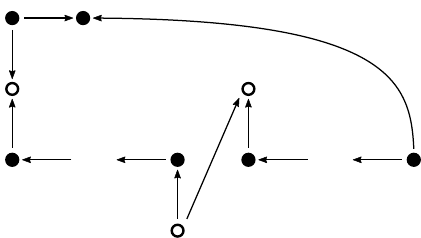
    \caption{$n>0$}
  \end{subfigure}\hfill
  \begin{subfigure}{.45\textwidth}
    \centering
    \footnotesize
    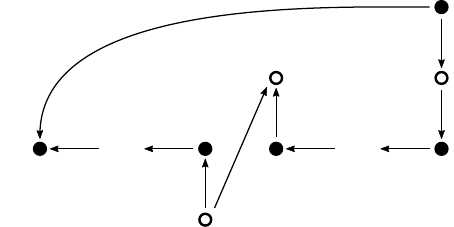
    \caption{$n<0$}
  \end{subfigure}\\
  \begin{subfigure}{.45\textwidth}
    \centering
    \footnotesize
    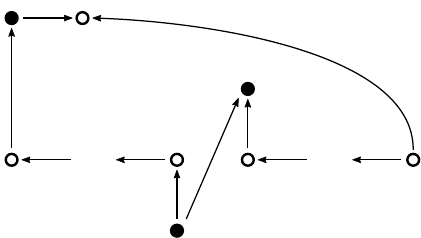
    \caption{Another parametrization for $n>0$}
  \end{subfigure}
  \hfill
  \begin{subfigure}{.45\textwidth}
    \centering
    \footnotesize
    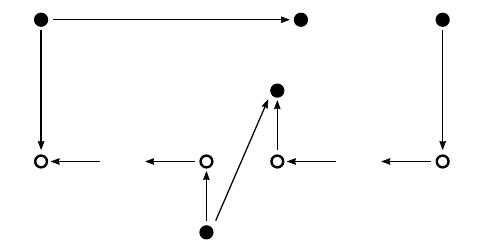
    \caption{Another parametrization for $n<0$}
  \end{subfigure}
  \caption{Simplified type-$D$ structures for the various parametrizations of the $\pm n$-framed solid torus}
  \label{fig:D-solid-tori-simp}
\end{figure}

We can also see these contact invariants in the immersed curve representation of these type-$D$ structures; see Figure \ref{fig:D-solid-tori-imm}.

\begin{figure}
  \centering
  \begin{subfigure}{.35\textwidth}
    \centering
    \small
    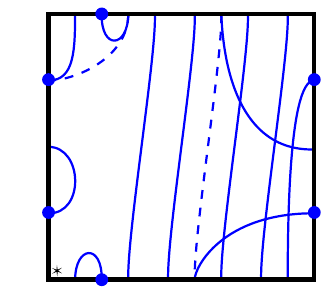
    \caption{$n>0$}
    \label{fig:solid-torus-imm-P1}
  \end{subfigure}\hspace{1cm}
  \begin{subfigure}{.35\textwidth}
    \centering
    \small
    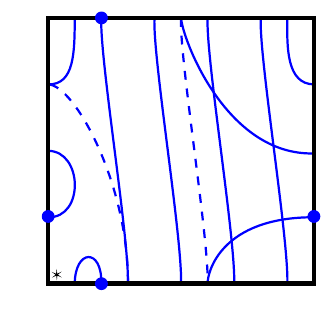
    \caption{$n<0$}
    \label{fig:solid-torus-imm-N1}
  \end{subfigure}
  \begin{subfigure}{.35\textwidth}
    \centering
    \small
    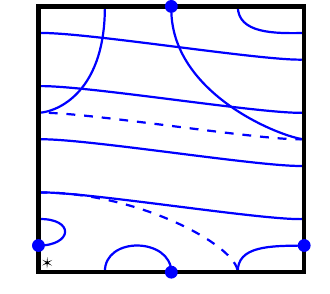
    \caption{Another parametrization for $n > 0$}
    \label{fig:solid-torus-imm-P2}
  \end{subfigure}
  \hspace{1cm}
  \begin{subfigure}{.35\textwidth}
    \centering
    \small
    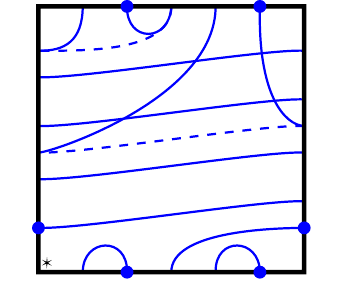
    \caption{Another parametrization for $n < 0$}
    \label{fig:solid-torus-imm-N2}
  \end{subfigure}
  \caption{type-$D$ immersed curves for the various parametrizations of the $n$-framed solid torus}
  \label{fig:D-solid-tori-imm}
\end{figure}

The general pattern may be summarized as follows: starting with the (appropriately parametrized) usual immersed curve diagram for an $n$-framed solid torus, which is a single closed curve on $T^2_{\bullet}$ of slope $n,$ (or$-\frac{1}{n}$ depending on parametrization), in order to see the generator corresponding to $x_j,$ one simply performs a ``finger move'' homotopy of the $j$th strand, as indicated in Figure \ref{fig:D-solid-tori-imm}. Similarly for the generator corresponding to $y$ (in some cases, such a finger move is not always necessary).

\section{Applications on contact surgery} \label{sec:applications}

In this section, we provide some examples and applications using the type-$D$ computations from Section~\ref{sec:torus} and the immersed curve technique. 

\subsection{Surgery formula for L-spaces}\label{subsec:surgery}
Here we generalize the result of Mark--Tosun \cite{MT:surgery} and Golla \cite{Golla:surgery} and prove Theorem~\ref{thm:surgery}.

To do so, we first pin down the behavior of the LOSS invariant for knots in L-spaces (cf. Theorem 1.2 of \cite{LOSS:LOSS}). Let $(Y,\xi)$ be a contact L-space with $\chat(\xi) = [c(\xi)] \neq 0$, and $K\subset Y$ be a null-homologous Legendrian knot. By construction, if we choose a representative for $\LOSS(K)$, which we denote $c^-\in \CFK^-(-Y,K,\mathfrak{s}),$ setting $U=1$ will give a chain complex $\CFhat(-Y),$ and $c^- \mapsto c(\xi)$ which we have assumed to be non-null-homologous. For convenience, let us choose a suitable basis $x, x_1,\dots,x_N$ for $\CFK^-(-Y,K,\mathfrak{s})$ wherein $x$ represents a generator of the $U=1$ quotient complex. Thus we must have that $c^- = U^{n_0} x + U^{n_1}x_{i_1}+\dots U^{n_k}x_{i_k}.$ 

Now, we shall see how to translate this information to the bordered contact invariant. Let us consider $Y(K)=Y\setminus \mathring{\nu}(K),$ the exterior of the knot $K.$ The natural contact structure $\xi_{\ell}$ induces the $\tb$-framed dividing set on the boundary, and one can attach a basic slice yielding the contact structure $\xi_{\mu}$ that induces the meridional dividing set. Let us parametrize the torus boundary with the meridian and $\tb$-framed longitude. Stipsicz--V\'ertesi \cite{SV:LOSS} showed that $EH(\xi_{\mu})$ may be identified with $\Lhat(K)$ under the natural identification $\SFH(-Y(K),\Gamma_\mu)=\HFKhat(-Y,K).$  
Lipshitz, Ozsv\'ath, and Thurston \cite{LOT:bordered} showed that there is a natural correspondence between the bordered invariants of a knot complement and the knot Floer invariants. In our case, by \cite[Theorem 11.9]{LOT:bordered}, we can recover the complex $\CFK^-(-Y,K)$ by pairing $\CFA(-Y(K))$ with a ``minus'' version of $\BSD(\HH_I),$ where $\HH_I$ is the cap diagram in Figure \ref{fig:caps.d}. This means that the interior suture is treated as an extra basepoint, in the usual minus manner.
We may perform Heegaard moves to obtain a doubly-pointed Heegaard diagram where we can see $\mathfrak{L}(K),$ and $\Lhat(K)$ is obtained from the same diagram by treating the extra basepoint as suture (algebraically, this corresponds to setting $U=0$). By naturality, this is the equivalent to attaching an ordinary cap, and perfoming the same Heegaard moves, which by definition maps $c_A(\xi_{\mu})$ to $\Lhat(K).$ Observe that $x,x_1,\dots,x_N$ is also a basis for $\CFKhat(-Y,K),$ which we may identify with $\CFA(-Y(K))\cdot \iota_0$ as usual, and as we have seen above, this identification naturally ``factors through'' $\CFK^-(-Y,K),$ so that $c_A(\xi_{\mu} = \left(U^{n_0} x + U^{n_1}x_{i_1}+\dots U^{n_k}x_{i_k}\right)/(U=0).$

\begin{lemma}\label{lem:loss_gen}
  With notation as above, $n_0=0$ if and only if $\tb(K) - \rot(K) = 2\tau_{\xi}(K) - 1.$
\end{lemma}

\begin{proof}
  By Theorem 1.6 of \cite{OS:grading}, $2A(\mathcal{L}(K))=\tb(K)-\rot(K)+1,$ so that $\tb(K) -\rot(K) = 2\tau_{\xi}(K) - 1$ if and only if $A(\LOSS(K))=\tau_{\xi}(K).$ Since $A(x)=\tau_{\xi}(K),$ this is equivalent to $n_0=0.$
\end{proof}

Let $(\Y,\xi)$ be a bordered contact 3-manifold with torus boundary, and suppose $c$ is a generator in a reduced basis for $\BSA(\Y)$; that is, there is a (simplified) immersed multicurve $\overrightarrow{\Gamma} \subset T^2_{\bullet}$ representing $\BSA(\Y)$, and $a$ is a single intersection point between one ``edge'' of $T_{\bullet}^2$ and one of the immersed curves $\gamma \subset \overrightarrow{\Gamma}.$ Suppose we have also parametrized the torus so that the idempotent of $a$ corresponds to the horizontal edge of the torus.  We fill with an $n$-framed solid torus. Recall from Section~\ref{subsec:solid}, that there exist two tight contact structures $\xi_+$ and $\xi_-$. Since we choose only negative stabilization for contact surgery, we pick the contact structure $\xi_-$ that corresponds to $x_1$; that is, $\iota_0^{\vee} x_1$ in $\BSD(S^1\times D^2).$ The relevant immersed curve is shown in Figure \ref{fig:solid-torus-imm-N1}. In what follows, we denote the resulting element $a\boxtimes \iota_1^{\vee} x_1$ in $\CFhat$ by $c$ as a shorthand.

We show several conditions under which positive contact filling results in a null-homologous contact class.
Let us first examine the cases arising from some kinds of ways $\gamma$ behaves when we follow it \emph{upwards} from $a$.

\begin{lemma}\label{lem:surg-V-U}
  If the curve $\gamma$ immediately above $a$ curves:
  \begin{enumerate}
    \item rightward or continues upward, OR
    \item leftward followed by a turn upward, and then (after possibly some vertical passes) turns rightward
  \end{enumerate}
  then $c$ is null-homologous.
\end{lemma}

\begin{proof}
  The situation in the first case is illustrated in Figure \ref{fig:surg-U-V1a}. We see that there is a generator $x$ and a bigon between $x$ and $c.$ Furthermore, there are no other bigons with $x$ as a starting point due to the assumption on $\gamma.$ Hence, we have that $\partial x =c$ in $\CFhat,$ so that $c$ is null-homologous.\\
  For the second case, there are several subcases. Let $m \geq 0$ be the number of vertical passes before the rightward turn. If $m<n$ then the situation is shown in Figure \ref{fig:surg-U-V2a}. As before, there is a generator $x$ and a bigon from $x$ to $c.$ No other bigons emanate from $x,$ and so we have that $\partial x =c,$ once again implying that $c$ is null-homologous. Figure \ref{fig:surg-U-V2b} depicts the case where $m=n.$ Here we see that there are generators $x,y,z$ and bigons from $x$ to $c,$ from $x$ to $y,$ and from $z$ to $y.$ No other bigons emanate from $z$ (or $x$) so that we have $\partial x = c+y$ and $\partial z = y.$ Thus $\partial (x+z)=c,$ so that $c$ is null-homologous. The case of $m>n$ is illustrated in Figure \ref{fig:surg-U-V2c} and is analogous to the $m=n$ case: we once again have that $\partial (x+z) = c.$
\end{proof}

\begin{remark}
  Note that the above Lemma (as well as the following two Lemmas) can be expressed in purely algebraic terms instead of immersed curves. For instance, a ``rightward or upward'' turn above $a$ corresponds to requiring that $m_i(a,\rho_2,\dots)=0$ for all $i\geq 2$.
\end{remark}
    
\begin{figure}
  \begin{subfigure}{.45\textwidth}
\begingroup%
  \makeatletter%
  \providecommand\color[2][]{%
    \errmessage{(Inkscape) Color is used for the text in Inkscape, but the package 'color.sty' is not loaded}%
    \renewcommand\color[2][]{}%
  }%
  \providecommand\transparent[1]{%
    \errmessage{(Inkscape) Transparency is used (non-zero) for the text in Inkscape, but the package 'transparent.sty' is not loaded}%
    \renewcommand\transparent[1]{}%
  }%
  \providecommand\rotatebox[2]{#2}%
  \newcommand*\fsize{\dimexpr\f@size pt\relax}%
  \newcommand*\lineheight[1]{\fontsize{\fsize}{#1\fsize}\selectfont}%
  \ifx\svgwidth\undefined%
    \setlength{\unitlength}{194.44844884bp}%
    \ifx\svgscale\undefined%
      \relax%
    \else%
      \setlength{\unitlength}{\unitlength * \real{\svgscale}}%
    \fi%
  \else%
    \setlength{\unitlength}{\svgwidth}%
  \fi%
  \global\let\svgwidth\undefined%
  \global\let\svgscale\undefined%
  \makeatother%
  \begin{picture}(1,0.99220679)%
    \lineheight{1}%
    \setlength\tabcolsep{0pt}%
    \put(0,0){\includegraphics[width=\unitlength,page=1]{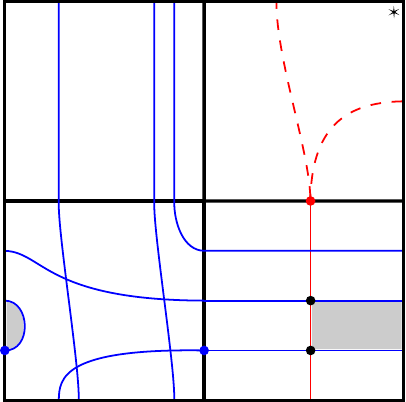}}%
    \put(0.24857876,0.19334885){\color[rgb]{0,0,1}\makebox(0,0)[lt]{\lineheight{1.25}\smash{\begin{tabular}[t]{l}$\dots$\end{tabular}}}}%
    \put(0.72312023,0.0798272){\color[rgb]{0,0,0}\makebox(0,0)[lt]{\lineheight{1.25}\smash{\begin{tabular}[t]{l}$c$\end{tabular}}}}%
    \put(0.72480195,0.26924079){\color[rgb]{0,0,0}\makebox(0,0)[lt]{\lineheight{1.25}\smash{\begin{tabular}[t]{l}$x$\end{tabular}}}}%
    \put(0.79248037,0.51860028){\color[rgb]{1,0,0}\makebox(0,0)[lt]{\lineheight{1.25}\smash{\begin{tabular}[t]{l}$a$\end{tabular}}}}%
    \put(0.02056242,0.08253223){\color[rgb]{0,0,1}\makebox(0,0)[lt]{\lineheight{1.25}\smash{\begin{tabular}[t]{l}$\iota_0^{\vee}x_1$\end{tabular}}}}%
  \end{picture}%
\endgroup%

    \caption{The case when $\gamma$ does not turn left}
    \label{fig:surg-U-V1a}
  \end{subfigure}\hfill
  \begin{subfigure}{.45\textwidth}
\begingroup%
  \makeatletter%
  \providecommand\color[2][]{%
    \errmessage{(Inkscape) Color is used for the text in Inkscape, but the package 'color.sty' is not loaded}%
    \renewcommand\color[2][]{}%
  }%
  \providecommand\transparent[1]{%
    \errmessage{(Inkscape) Transparency is used (non-zero) for the text in Inkscape, but the package 'transparent.sty' is not loaded}%
    \renewcommand\transparent[1]{}%
  }%
  \providecommand\rotatebox[2]{#2}%
  \newcommand*\fsize{\dimexpr\f@size pt\relax}%
  \newcommand*\lineheight[1]{\fontsize{\fsize}{#1\fsize}\selectfont}%
  \ifx\svgwidth\undefined%
    \setlength{\unitlength}{194.44842722bp}%
    \ifx\svgscale\undefined%
      \relax%
    \else%
      \setlength{\unitlength}{\unitlength * \real{\svgscale}}%
    \fi%
  \else%
    \setlength{\unitlength}{\svgwidth}%
  \fi%
  \global\let\svgwidth\undefined%
  \global\let\svgscale\undefined%
  \makeatother%
  \begin{picture}(1,0.99220668)%
    \lineheight{1}%
    \setlength\tabcolsep{0pt}%
    \put(0,0){\includegraphics[width=\unitlength,page=1]{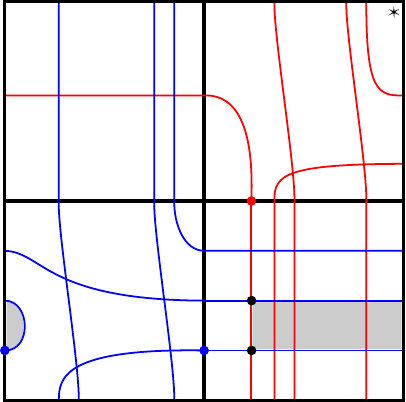}}%
    \put(0.57734186,0.08754134){\color[rgb]{0,0,0}\makebox(0,0)[lt]{\lineheight{1.25}\smash{\begin{tabular}[t]{l}$c$\end{tabular}}}}%
    \put(0.57551932,0.21858304){\color[rgb]{0,0,0}\makebox(0,0)[lt]{\lineheight{1.25}\smash{\begin{tabular}[t]{l}$x$\end{tabular}}}}%
    \put(0.63127404,0.51860039){\color[rgb]{1,0,0}\makebox(0,0)[lt]{\lineheight{1.25}\smash{\begin{tabular}[t]{l}$a$\end{tabular}}}}%
    \put(0.02056233,0.0825323){\color[rgb]{0,0,1}\makebox(0,0)[lt]{\lineheight{1.25}\smash{\begin{tabular}[t]{l}$\iota_0^{\vee}x_1$\end{tabular}}}}%
    \put(0.75544742,0.73308805){\color[rgb]{1,0,0}\makebox(0,0)[lt]{\lineheight{1.25}\smash{\begin{tabular}[t]{l}$\dots$\end{tabular}}}}%
    \put(0,0){\includegraphics[width=\unitlength,page=2]{Imm_surg_U_V_2a.pdf}}%
    \put(0.20229394,0.19334899){\color[rgb]{0,0,1}\makebox(0,0)[lt]{\lineheight{1.25}\smash{\begin{tabular}[t]{l}$\dots$\end{tabular}}}}%
    \put(0.31800589,0.19334899){\color[rgb]{0,0,1}\makebox(0,0)[lt]{\lineheight{1.25}\smash{\begin{tabular}[t]{l}$\dots$\end{tabular}}}}%
  \end{picture}%
\endgroup%

    \caption{One case when $\gamma$ turns left then up then (eventually) right: $n>m$}
    \label{fig:surg-U-V2a}
  \end{subfigure}\\
  \vspace{0.2cm}
  \begin{subfigure}{.45\textwidth}
\begingroup%
  \makeatletter%
  \providecommand\color[2][]{%
    \errmessage{(Inkscape) Color is used for the text in Inkscape, but the package 'color.sty' is not loaded}%
    \renewcommand\color[2][]{}%
  }%
  \providecommand\transparent[1]{%
    \errmessage{(Inkscape) Transparency is used (non-zero) for the text in Inkscape, but the package 'transparent.sty' is not loaded}%
    \renewcommand\transparent[1]{}%
  }%
  \providecommand\rotatebox[2]{#2}%
  \newcommand*\fsize{\dimexpr\f@size pt\relax}%
  \newcommand*\lineheight[1]{\fontsize{\fsize}{#1\fsize}\selectfont}%
  \ifx\svgwidth\undefined%
    \setlength{\unitlength}{194.44838396bp}%
    \ifx\svgscale\undefined%
      \relax%
    \else%
      \setlength{\unitlength}{\unitlength * \real{\svgscale}}%
    \fi%
  \else%
    \setlength{\unitlength}{\svgwidth}%
  \fi%
  \global\let\svgwidth\undefined%
  \global\let\svgscale\undefined%
  \makeatother%
  \begin{picture}(1,0.9922069)%
    \lineheight{1}%
    \setlength\tabcolsep{0pt}%
    \put(0,0){\includegraphics[width=\unitlength,page=1]{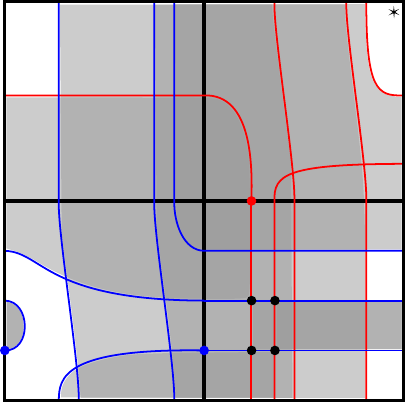}}%
    \put(0.24857861,0.19334897){\color[rgb]{0,0,1}\makebox(0,0)[lt]{\lineheight{1.25}\smash{\begin{tabular}[t]{l}$\dots$\end{tabular}}}}%
    \put(0.57734184,0.14925451){\color[rgb]{0,0,0}\makebox(0,0)[lt]{\lineheight{1.25}\smash{\begin{tabular}[t]{l}$c$\end{tabular}}}}%
    \put(0.62951818,0.28029617){\color[rgb]{0,0,0}\makebox(0,0)[lt]{\lineheight{1.25}\smash{\begin{tabular}[t]{l}$x$\end{tabular}}}}%
    \put(0.63127403,0.5186005){\color[rgb]{1,0,0}\makebox(0,0)[lt]{\lineheight{1.25}\smash{\begin{tabular}[t]{l}$a$\end{tabular}}}}%
    \put(0.02056219,0.08253231){\color[rgb]{0,0,1}\makebox(0,0)[lt]{\lineheight{1.25}\smash{\begin{tabular}[t]{l}$\iota_0^{\vee}x_1$\end{tabular}}}}%
    \put(0.75544745,0.73308821){\color[rgb]{1,0,0}\makebox(0,0)[lt]{\lineheight{1.25}\smash{\begin{tabular}[t]{l}$\dots$\end{tabular}}}}%
    \put(0.68321133,0.08840796){\color[rgb]{0,0,0}\makebox(0,0)[lt]{\lineheight{1.25}\smash{\begin{tabular}[t]{l}$y$\end{tabular}}}}%
    \put(0.68714006,0.21858321){\color[rgb]{0,0,0}\makebox(0,0)[lt]{\lineheight{1.25}\smash{\begin{tabular}[t]{l}$z$\end{tabular}}}}%
  \end{picture}%
\endgroup%

    \caption{One case when $\gamma$ turns left then up then (eventually) right: $n=m$}
    \label{fig:surg-U-V2b}
  \end{subfigure}\hfill
  \begin{subfigure}{.45\textwidth}
    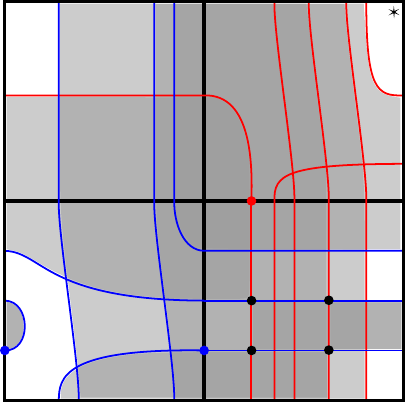
    \caption{One case when $\gamma$ turns left then up then (eventually) right: $n<m$}
    \label{fig:surg-U-V2c}
  \end{subfigure}
  \caption{Various behaviors of $\gamma$ directly above $a$ for which pairing with $\iota_0^{\vee}x_1$ results in a null-homologous class}
  \label{fig:surg-U-V}
\end{figure}

Now let us turn to the behavior of $\gamma$ immediately \emph{below} $a$.

\begin{lemma}\label{lem:surg-V-D}
  If the curve $\gamma$ immediately below $a$  after potentially some number $m \geq 0$ of vertical passes turns:
  \begin{enumerate}
    \item to the left, OR
    \item to the right AND $m \geq n,$
  \end{enumerate}
  then $c$ is null-homologous.
\end{lemma}

\begin{proof}
   Let us consider the left-turning cases first. The case where $m<n-1$ is depicted in Figure \ref{fig:surg-D-Va}. Figure \ref{fig:surg-D-Vb} shows the case where $m=n-1,$ and Figure \ref{fig:surg-D-Vc} illustrates the case where $m \geq n.$ In any case, we find  a generator $x$ and a bigon from $x$ to $c.$ Moreover, there cannot be any other bigon emanating from $x,$ and so we conclude that $\partial x=c$ in $\CFhat$ so that $c$ is null-homologous.\\
   Now we assume $\gamma$ turns to the right and $m \geq n.$ Figure \ref{fig:surg-D-V2} shows that there is once again some generator $x$ and a bigon from $x$ to $c.$ As before, this is also the only bigon emanating from $x,$ and so $\partial x= c$ as in the previous cases.
\end{proof}

\begin{figure}
  \centering
  \begin{subfigure}{.45\textwidth}
    \centering
\begingroup%
  \makeatletter%
  \providecommand\color[2][]{%
    \errmessage{(Inkscape) Color is used for the text in Inkscape, but the package 'color.sty' is not loaded}%
    \renewcommand\color[2][]{}%
  }%
  \providecommand\transparent[1]{%
    \errmessage{(Inkscape) Transparency is used (non-zero) for the text in Inkscape, but the package 'transparent.sty' is not loaded}%
    \renewcommand\transparent[1]{}%
  }%
  \providecommand\rotatebox[2]{#2}%
  \newcommand*\fsize{\dimexpr\f@size pt\relax}%
  \newcommand*\lineheight[1]{\fontsize{\fsize}{#1\fsize}\selectfont}%
  \ifx\svgwidth\undefined%
    \setlength{\unitlength}{194.44847047bp}%
    \ifx\svgscale\undefined%
      \relax%
    \else%
      \setlength{\unitlength}{\unitlength * \real{\svgscale}}%
    \fi%
  \else%
    \setlength{\unitlength}{\svgwidth}%
  \fi%
  \global\let\svgwidth\undefined%
  \global\let\svgscale\undefined%
  \makeatother%
  \begin{picture}(1,0.99999255)%
    \lineheight{1}%
    \setlength\tabcolsep{0pt}%
    \put(0,0){\includegraphics[width=\unitlength,page=1]{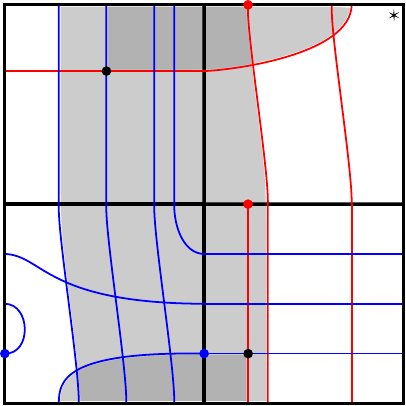}}%
    \put(0.20229397,0.19334886){\color[rgb]{0,0,1}\makebox(0,0)[lt]{\lineheight{1.25}\smash{\begin{tabular}[t]{l}$\dots$\end{tabular}}}}%
    \put(0.56883756,0.14925443){\color[rgb]{0,0,0}\makebox(0,0)[lt]{\lineheight{1.25}\smash{\begin{tabular}[t]{l}$c$\end{tabular}}}}%
    \put(0.2735441,0.77965748){\color[rgb]{0,0,0}\makebox(0,0)[lt]{\lineheight{1.25}\smash{\begin{tabular}[t]{l}$x$\end{tabular}}}}%
    \put(0.56105647,0.51860036){\color[rgb]{1,0,0}\makebox(0,0)[lt]{\lineheight{1.25}\smash{\begin{tabular}[t]{l}$a$\end{tabular}}}}%
    \put(0.02056236,0.08253222){\color[rgb]{0,0,1}\makebox(0,0)[lt]{\lineheight{1.25}\smash{\begin{tabular}[t]{l}$\iota_0^{\vee}x_1$\end{tabular}}}}%
    \put(0.7087629,0.73308417){\color[rgb]{1,0,0}\makebox(0,0)[lt]{\lineheight{1.25}\smash{\begin{tabular}[t]{l}$\dots$\end{tabular}}}}%
    \put(0.31800588,0.19334886){\color[rgb]{0,0,1}\makebox(0,0)[lt]{\lineheight{1.25}\smash{\begin{tabular}[t]{l}$\dots$\end{tabular}}}}%
  \end{picture}%
\endgroup%

    \caption{One case when $\gamma$ turns  (eventually) left: $n>m$}
    \label{fig:surg-D-Va}
  \end{subfigure}\hfill
  \begin{subfigure}{.45\textwidth}
    \centering
\begingroup%
  \makeatletter%
  \providecommand\color[2][]{%
    \errmessage{(Inkscape) Color is used for the text in Inkscape, but the package 'color.sty' is not loaded}%
    \renewcommand\color[2][]{}%
  }%
  \providecommand\transparent[1]{%
    \errmessage{(Inkscape) Transparency is used (non-zero) for the text in Inkscape, but the package 'transparent.sty' is not loaded}%
    \renewcommand\transparent[1]{}%
  }%
  \providecommand\rotatebox[2]{#2}%
  \newcommand*\fsize{\dimexpr\f@size pt\relax}%
  \newcommand*\lineheight[1]{\fontsize{\fsize}{#1\fsize}\selectfont}%
  \ifx\svgwidth\undefined%
    \setlength{\unitlength}{194.44844884bp}%
    \ifx\svgscale\undefined%
      \relax%
    \else%
      \setlength{\unitlength}{\unitlength * \real{\svgscale}}%
    \fi%
  \else%
    \setlength{\unitlength}{\svgwidth}%
  \fi%
  \global\let\svgwidth\undefined%
  \global\let\svgscale\undefined%
  \makeatother%
  \begin{picture}(1,0.99999266)%
    \lineheight{1}%
    \setlength\tabcolsep{0pt}%
    \put(0,0){\includegraphics[width=\unitlength,page=1]{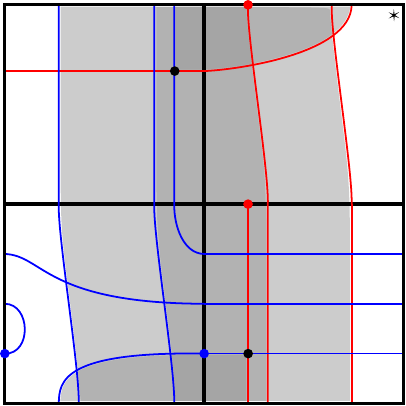}}%
    \put(0.24857879,0.19334888){\color[rgb]{0,0,1}\makebox(0,0)[lt]{\lineheight{1.25}\smash{\begin{tabular}[t]{l}$\dots$\end{tabular}}}}%
    \put(0.56883765,0.14925445){\color[rgb]{0,0,0}\makebox(0,0)[lt]{\lineheight{1.25}\smash{\begin{tabular}[t]{l}$c$\end{tabular}}}}%
    \put(0.56105645,0.51860042){\color[rgb]{1,0,0}\makebox(0,0)[lt]{\lineheight{1.25}\smash{\begin{tabular}[t]{l}$a$\end{tabular}}}}%
    \put(0.02056233,0.08253223){\color[rgb]{0,0,1}\makebox(0,0)[lt]{\lineheight{1.25}\smash{\begin{tabular}[t]{l}$\iota_0^{\vee}x_1$\end{tabular}}}}%
    \put(0.7087629,0.73308425){\color[rgb]{1,0,0}\makebox(0,0)[lt]{\lineheight{1.25}\smash{\begin{tabular}[t]{l}$\dots$\end{tabular}}}}%
    \put(0.44007626,0.77965267){\color[rgb]{0,0,0}\makebox(0,0)[lt]{\lineheight{1.25}\smash{\begin{tabular}[t]{l}$x$\end{tabular}}}}%
  \end{picture}%
\endgroup%

    \caption{One case when $\gamma$ turns  (eventually) left: $n=m$}
    \label{fig:surg-D-Vb}
  \end{subfigure}\\
  \begin{subfigure}{.45\textwidth}
\begingroup%
  \makeatletter%
  \providecommand\color[2][]{%
    \errmessage{(Inkscape) Color is used for the text in Inkscape, but the package 'color.sty' is not loaded}%
    \renewcommand\color[2][]{}%
  }%
  \providecommand\transparent[1]{%
    \errmessage{(Inkscape) Transparency is used (non-zero) for the text in Inkscape, but the package 'transparent.sty' is not loaded}%
    \renewcommand\transparent[1]{}%
  }%
  \providecommand\rotatebox[2]{#2}%
  \newcommand*\fsize{\dimexpr\f@size pt\relax}%
  \newcommand*\lineheight[1]{\fontsize{\fsize}{#1\fsize}\selectfont}%
  \ifx\svgwidth\undefined%
    \setlength{\unitlength}{194.44842722bp}%
    \ifx\svgscale\undefined%
      \relax%
    \else%
      \setlength{\unitlength}{\unitlength * \real{\svgscale}}%
    \fi%
  \else%
    \setlength{\unitlength}{\svgwidth}%
  \fi%
  \global\let\svgwidth\undefined%
  \global\let\svgscale\undefined%
  \makeatother%
  \begin{picture}(1,0.99999277)%
    \lineheight{1}%
    \setlength\tabcolsep{0pt}%
    \put(0,0){\includegraphics[width=\unitlength,page=1]{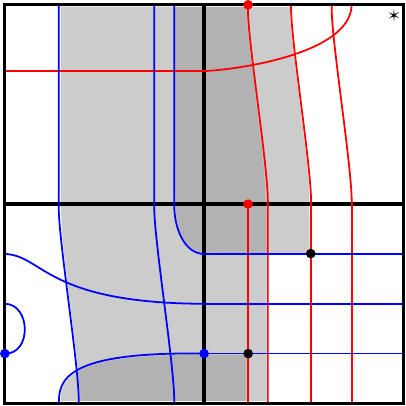}}%
    \put(0.24857864,0.19334886){\color[rgb]{0,0,1}\makebox(0,0)[lt]{\lineheight{1.25}\smash{\begin{tabular}[t]{l}$\dots$\end{tabular}}}}%
    \put(0.56883754,0.14925443){\color[rgb]{0,0,0}\makebox(0,0)[lt]{\lineheight{1.25}\smash{\begin{tabular}[t]{l}$c$\end{tabular}}}}%
    \put(0.77910747,0.33133978){\color[rgb]{0,0,0}\makebox(0,0)[lt]{\lineheight{1.25}\smash{\begin{tabular}[t]{l}$x$\end{tabular}}}}%
    \put(0.56105656,0.51860044){\color[rgb]{1,0,0}\makebox(0,0)[lt]{\lineheight{1.25}\smash{\begin{tabular}[t]{l}$a$\end{tabular}}}}%
    \put(0.02056238,0.0825322){\color[rgb]{0,0,1}\makebox(0,0)[lt]{\lineheight{1.25}\smash{\begin{tabular}[t]{l}$\iota_0^{\vee}x_1$\end{tabular}}}}%
    \put(0.66247866,0.73308429){\color[rgb]{1,0,0}\makebox(0,0)[lt]{\lineheight{1.25}\smash{\begin{tabular}[t]{l}$\dots$\end{tabular}}}}%
    \put(0.7656547,0.73308429){\color[rgb]{1,0,0}\makebox(0,0)[lt]{\lineheight{1.25}\smash{\begin{tabular}[t]{l}$\dots$\end{tabular}}}}%
  \end{picture}%
\endgroup%

    \caption{One case when $\gamma$ turns  (eventually) left: $n<m$}
    \label{fig:surg-D-Vc}
  \end{subfigure}\hfill
  \begin{subfigure}{.45\textwidth}
\begingroup%
  \makeatletter%
  \providecommand\color[2][]{%
    \errmessage{(Inkscape) Color is used for the text in Inkscape, but the package 'color.sty' is not loaded}%
    \renewcommand\color[2][]{}%
  }%
  \providecommand\transparent[1]{%
    \errmessage{(Inkscape) Transparency is used (non-zero) for the text in Inkscape, but the package 'transparent.sty' is not loaded}%
    \renewcommand\transparent[1]{}%
  }%
  \providecommand\rotatebox[2]{#2}%
  \newcommand*\fsize{\dimexpr\f@size pt\relax}%
  \newcommand*\lineheight[1]{\fontsize{\fsize}{#1\fsize}\selectfont}%
  \ifx\svgwidth\undefined%
    \setlength{\unitlength}{194.44847047bp}%
    \ifx\svgscale\undefined%
      \relax%
    \else%
      \setlength{\unitlength}{\unitlength * \real{\svgscale}}%
    \fi%
  \else%
    \setlength{\unitlength}{\svgwidth}%
  \fi%
  \global\let\svgwidth\undefined%
  \global\let\svgscale\undefined%
  \makeatother%
  \begin{picture}(1,0.99220657)%
    \lineheight{1}%
    \setlength\tabcolsep{0pt}%
    \put(0,0){\includegraphics[width=\unitlength,page=1]{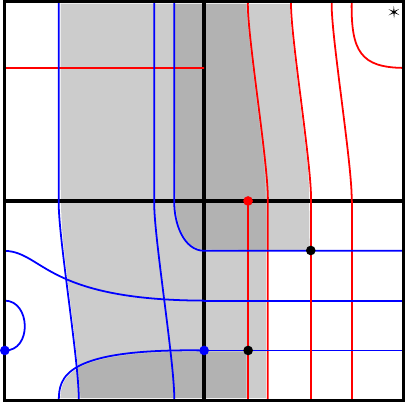}}%
    \put(0.24857849,0.19334875){\color[rgb]{0,0,1}\makebox(0,0)[lt]{\lineheight{1.25}\smash{\begin{tabular}[t]{l}$\dots$\end{tabular}}}}%
    \put(0.56883731,0.14925432){\color[rgb]{0,0,0}\makebox(0,0)[lt]{\lineheight{1.25}\smash{\begin{tabular}[t]{l}$c$\end{tabular}}}}%
    \put(0.7791072,0.33905366){\color[rgb]{0,0,0}\makebox(0,0)[lt]{\lineheight{1.25}\smash{\begin{tabular}[t]{l}$x$\end{tabular}}}}%
    \put(0.56105634,0.51860026){\color[rgb]{1,0,0}\makebox(0,0)[lt]{\lineheight{1.25}\smash{\begin{tabular}[t]{l}$a$\end{tabular}}}}%
    \put(0.02056228,0.08253211){\color[rgb]{0,0,1}\makebox(0,0)[lt]{\lineheight{1.25}\smash{\begin{tabular}[t]{l}$\iota_0^{\vee}x_1$\end{tabular}}}}%
    \put(0.66247841,0.73308406){\color[rgb]{1,0,0}\makebox(0,0)[lt]{\lineheight{1.25}\smash{\begin{tabular}[t]{l}$\dots$\end{tabular}}}}%
    \put(0.76565443,0.73308406){\color[rgb]{1,0,0}\makebox(0,0)[lt]{\lineheight{1.25}\smash{\begin{tabular}[t]{l}$\dots$\end{tabular}}}}%
  \end{picture}%
\endgroup%

    \caption{One case when $\gamma$ turns  (eventually) right and $n\leq m$}
    \label{fig:surg-D-V2}
  \end{subfigure}
  \caption{Various behaviors of $\gamma$ directly below $a$ for which pairing with $\iota_0^{\vee}x_1$ results in a null-homologous class}
  \label{fig:surg-D-V}
\end{figure}

\begin{remark}
  In the above calculations, we have been assuming that the curve $\gamma$ has a trivial local system. However, the result is the same even for nontrivial local systems. This is because, in that case, $\gamma$ can be thought of as $N$ parallel copies of itself, and so each intersection point in the above argument corresponds to $N$ points. The local system alters the differential by a vector space automorphism on $\FF^N$. The above lemmas show that a certain element is in the image of the differential; in the case of nontrivial local system, each of the $N$ corresponding elements must still be in the image, as it only differs by an isomorphism.
\end{remark}

Now we record some nonvanishing conditions:

\begin{lemma}\label{lem:surg-NV-1}
  Let $a$ and $c$ as above. Suppose $\gamma$
  \begin{enumerate}
    \item  immediately above $a$ curves to the left, followed by either a downward turn OR a continuation to the left, AND
    \item immediately below $a$ curves to the right after $m$ vertical passes, and furthermore if
    \begin{enumerate}
      \item $m<n-1,$ OR
      \item $m=n-1$ AND $\gamma$ proceeds to the right OR curves up.
    \end{enumerate}
  \end{enumerate}
  Then $c$ is not null-homologous.
\end{lemma}

\begin{proof}
  We must rule out the possibility of $c$ being in the image of the differential. Hence, we must analyze potential bigons terminating at $c$ both from above and below. Figure \ref{fig:surg-U-NV} shows that the first condition implies the existence of a generator $x$ as well as a generator $y$ or $y',$ depending on whether we have a downward turn or horizontal continuation, respectively, such that $\partial x=c+y$ or $\partial x = c+y',$ respectively. We also see that neither $y$ nor $y'$ is involved in any futher bigons.\\
  Now we turn to the second condition, which has two possibilities, illustrated in Figures \ref{fig:surg-D-NVa} and \ref{fig:surg-D-NVb}. When either of these holds, we see that there are no bigons terminating at $c$ from below. Therefore, we conclude that $\partial x=c+y$ or $\partial x = c+y'$ is the only appearance $c$ has in any differential of a generator (and the same is true for $y$ or $y'$), so that $c$ is indeed not null-homologous.
\end{proof}

\begin{figure}
  \centering
  \begin{subfigure}{.45\textwidth}
\begingroup%
  \makeatletter%
  \providecommand\color[2][]{%
    \errmessage{(Inkscape) Color is used for the text in Inkscape, but the package 'color.sty' is not loaded}%
    \renewcommand\color[2][]{}%
  }%
  \providecommand\transparent[1]{%
    \errmessage{(Inkscape) Transparency is used (non-zero) for the text in Inkscape, but the package 'transparent.sty' is not loaded}%
    \renewcommand\transparent[1]{}%
  }%
  \providecommand\rotatebox[2]{#2}%
  \newcommand*\fsize{\dimexpr\f@size pt\relax}%
  \newcommand*\lineheight[1]{\fontsize{\fsize}{#1\fsize}\selectfont}%
  \ifx\svgwidth\undefined%
    \setlength{\unitlength}{194.44844884bp}%
    \ifx\svgscale\undefined%
      \relax%
    \else%
      \setlength{\unitlength}{\unitlength * \real{\svgscale}}%
    \fi%
  \else%
    \setlength{\unitlength}{\svgwidth}%
  \fi%
  \global\let\svgwidth\undefined%
  \global\let\svgscale\undefined%
  \makeatother%
  \begin{picture}(1,0.99220679)%
    \lineheight{1}%
    \setlength\tabcolsep{0pt}%
    \put(0,0){\includegraphics[width=\unitlength,page=1]{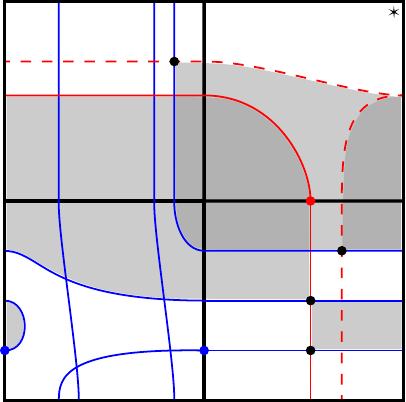}}%
    \put(0.24857882,0.19334899){\color[rgb]{0,0,1}\makebox(0,0)[lt]{\lineheight{1.25}\smash{\begin{tabular}[t]{l}$\dots$\end{tabular}}}}%
    \put(0.72312029,0.08754141){\color[rgb]{0,0,0}\makebox(0,0)[lt]{\lineheight{1.25}\smash{\begin{tabular}[t]{l}$c$\end{tabular}}}}%
    \put(0.72129775,0.21858309){\color[rgb]{0,0,0}\makebox(0,0)[lt]{\lineheight{1.25}\smash{\begin{tabular}[t]{l}$x$\end{tabular}}}}%
    \put(0.77705246,0.51860041){\color[rgb]{1,0,0}\makebox(0,0)[lt]{\lineheight{1.25}\smash{\begin{tabular}[t]{l}$a$\end{tabular}}}}%
    \put(0.02056247,0.08253237){\color[rgb]{0,0,1}\makebox(0,0)[lt]{\lineheight{1.25}\smash{\begin{tabular}[t]{l}$\iota_0^{\vee}x_1$\end{tabular}}}}%
    \put(0.8761918,0.33727846){\color[rgb]{0,0,0}\makebox(0,0)[lt]{\lineheight{1.25}\smash{\begin{tabular}[t]{l}$y$\end{tabular}}}}%
    \put(0.44893247,0.87537624){\color[rgb]{0,0,0}\makebox(0,0)[lt]{\lineheight{1.25}\smash{\begin{tabular}[t]{l}$y'$\end{tabular}}}}%
  \end{picture}%
\endgroup%

    \caption{The case where above $a,$ $\gamma$ turns left and then down or continues left}
    \label{fig:surg-U-NV}
  \end{subfigure}
  \hfill
  \begin{subfigure}{.45\textwidth}
\begingroup%
  \makeatletter%
  \providecommand\color[2][]{%
    \errmessage{(Inkscape) Color is used for the text in Inkscape, but the package 'color.sty' is not loaded}%
    \renewcommand\color[2][]{}%
  }%
  \providecommand\transparent[1]{%
    \errmessage{(Inkscape) Transparency is used (non-zero) for the text in Inkscape, but the package 'transparent.sty' is not loaded}%
    \renewcommand\transparent[1]{}%
  }%
  \providecommand\rotatebox[2]{#2}%
  \newcommand*\fsize{\dimexpr\f@size pt\relax}%
  \newcommand*\lineheight[1]{\fontsize{\fsize}{#1\fsize}\selectfont}%
  \ifx\svgwidth\undefined%
    \setlength{\unitlength}{194.44844884bp}%
    \ifx\svgscale\undefined%
      \relax%
    \else%
      \setlength{\unitlength}{\unitlength * \real{\svgscale}}%
    \fi%
  \else%
    \setlength{\unitlength}{\svgwidth}%
  \fi%
  \global\let\svgwidth\undefined%
  \global\let\svgscale\undefined%
  \makeatother%
  \begin{picture}(1,1.00000389)%
    \lineheight{1}%
    \setlength\tabcolsep{0pt}%
    \put(0,0){\includegraphics[width=\unitlength,page=1]{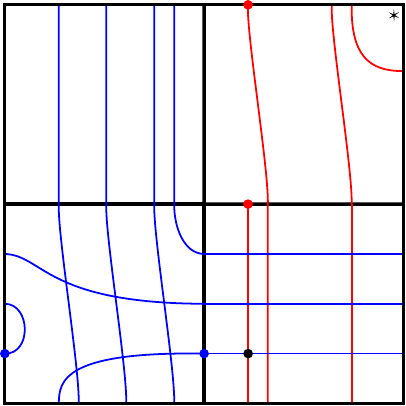}}%
    \put(0.20229402,0.19334876){\color[rgb]{0,0,1}\makebox(0,0)[lt]{\lineheight{1.25}\smash{\begin{tabular}[t]{l}$\dots$\end{tabular}}}}%
    \put(0.56883765,0.14925433){\color[rgb]{0,0,0}\makebox(0,0)[lt]{\lineheight{1.25}\smash{\begin{tabular}[t]{l}$c$\end{tabular}}}}%
    \put(0.56105656,0.5186003){\color[rgb]{1,0,0}\makebox(0,0)[lt]{\lineheight{1.25}\smash{\begin{tabular}[t]{l}$a$\end{tabular}}}}%
    \put(0.02056239,0.08253211){\color[rgb]{0,0,1}\makebox(0,0)[lt]{\lineheight{1.25}\smash{\begin{tabular}[t]{l}$\iota_0^{\vee}x_1$\end{tabular}}}}%
    \put(0.70876301,0.73308413){\color[rgb]{1,0,0}\makebox(0,0)[lt]{\lineheight{1.25}\smash{\begin{tabular}[t]{l}$\dots$\end{tabular}}}}%
    \put(0.31800595,0.19334876){\color[rgb]{0,0,1}\makebox(0,0)[lt]{\lineheight{1.25}\smash{\begin{tabular}[t]{l}$\dots$\end{tabular}}}}%
  \end{picture}%
\endgroup%

    \caption{One case where below $a,$ $\gamma$ turns (eventually) right: $n>m$}
    \label{fig:surg-D-NVa}
  \end{subfigure}\\
  \begin{subfigure}{.45\textwidth}
\begingroup%
  \makeatletter%
  \providecommand\color[2][]{%
    \errmessage{(Inkscape) Color is used for the text in Inkscape, but the package 'color.sty' is not loaded}%
    \renewcommand\color[2][]{}%
  }%
  \providecommand\transparent[1]{%
    \errmessage{(Inkscape) Transparency is used (non-zero) for the text in Inkscape, but the package 'transparent.sty' is not loaded}%
    \renewcommand\transparent[1]{}%
  }%
  \providecommand\rotatebox[2]{#2}%
  \newcommand*\fsize{\dimexpr\f@size pt\relax}%
  \newcommand*\lineheight[1]{\fontsize{\fsize}{#1\fsize}\selectfont}%
  \ifx\svgwidth\undefined%
    \setlength{\unitlength}{194.44844884bp}%
    \ifx\svgscale\undefined%
      \relax%
    \else%
      \setlength{\unitlength}{\unitlength * \real{\svgscale}}%
    \fi%
  \else%
    \setlength{\unitlength}{\svgwidth}%
  \fi%
  \global\let\svgwidth\undefined%
  \global\let\svgscale\undefined%
  \makeatother%
  \begin{picture}(1,1.00000389)%
    \lineheight{1}%
    \setlength\tabcolsep{0pt}%
    \put(0,0){\includegraphics[width=\unitlength,page=1]{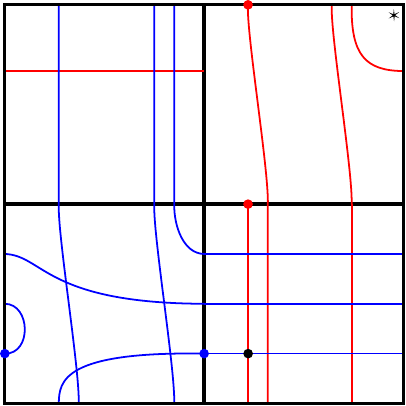}}%
    \put(0.24857879,0.19334876){\color[rgb]{0,0,1}\makebox(0,0)[lt]{\lineheight{1.25}\smash{\begin{tabular}[t]{l}$\dots$\end{tabular}}}}%
    \put(0.56883765,0.14925433){\color[rgb]{0,0,0}\makebox(0,0)[lt]{\lineheight{1.25}\smash{\begin{tabular}[t]{l}$c$\end{tabular}}}}%
    \put(0.56105645,0.5186003){\color[rgb]{1,0,0}\makebox(0,0)[lt]{\lineheight{1.25}\smash{\begin{tabular}[t]{l}$a$\end{tabular}}}}%
    \put(0.02056233,0.08253211){\color[rgb]{0,0,1}\makebox(0,0)[lt]{\lineheight{1.25}\smash{\begin{tabular}[t]{l}$\iota_0^{\vee}x_1$\end{tabular}}}}%
    \put(0.7087629,0.73308413){\color[rgb]{1,0,0}\makebox(0,0)[lt]{\lineheight{1.25}\smash{\begin{tabular}[t]{l}$\dots$\end{tabular}}}}%
    \put(0,0){\includegraphics[width=\unitlength,page=2]{Imm_surg_D_NV_2.pdf}}%
  \end{picture}%
\endgroup%

    \caption{One case where below $a,$ $\gamma$ turns (eventually) right: $n=m$ and then no immediate downward turn}
    \label{fig:surg-D-NVb}
  \end{subfigure}
  \caption{Behaviors of $\gamma$ such that if (a) and one of (b) or (c) is satisfied, result in a nonvanishing generator}
  \label{fig:surg-NV}
\end{figure}

We also record another important special nonvanishing case:

\begin{lemma}\label{lem:surg-NV-2}
  Let $a$ and $c$ as above. Suppose $\gamma$ is a simple closed curve of slope $-n.$ Then $c$ is not null-homologous.
\end{lemma}

\begin{proof}
  According to Figure \ref{fig:surg-C}, there is a generator $x,$ and two bigons from $x$ to $c.$ There are no other bigons involving $x$ or $c,$ and so we see that $c$ is not null-homologous.
\end{proof}

\begin{figure}
  \centering
\begingroup%
  \makeatletter%
  \providecommand\color[2][]{%
    \errmessage{(Inkscape) Color is used for the text in Inkscape, but the package 'color.sty' is not loaded}%
    \renewcommand\color[2][]{}%
  }%
  \providecommand\transparent[1]{%
    \errmessage{(Inkscape) Transparency is used (non-zero) for the text in Inkscape, but the package 'transparent.sty' is not loaded}%
    \renewcommand\transparent[1]{}%
  }%
  \providecommand\rotatebox[2]{#2}%
  \newcommand*\fsize{\dimexpr\f@size pt\relax}%
  \newcommand*\lineheight[1]{\fontsize{\fsize}{#1\fsize}\selectfont}%
  \ifx\svgwidth\undefined%
    \setlength{\unitlength}{194.44844884bp}%
    \ifx\svgscale\undefined%
      \relax%
    \else%
      \setlength{\unitlength}{\unitlength * \real{\svgscale}}%
    \fi%
  \else%
    \setlength{\unitlength}{\svgwidth}%
  \fi%
  \global\let\svgwidth\undefined%
  \global\let\svgscale\undefined%
  \makeatother%
  \begin{picture}(1,0.99999733)%
    \lineheight{1}%
    \setlength\tabcolsep{0pt}%
    \put(0,0){\includegraphics[width=\unitlength,page=1]{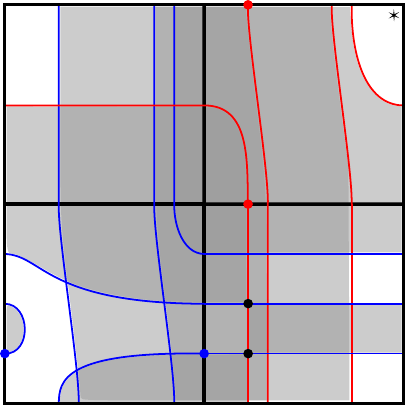}}%
    \put(0.24857896,0.19334874){\color[rgb]{0,0,1}\makebox(0,0)[lt]{\lineheight{1.25}\smash{\begin{tabular}[t]{l}$\dots$\end{tabular}}}}%
    \put(0.56883782,0.14925431){\color[rgb]{0,0,0}\makebox(0,0)[lt]{\lineheight{1.25}\smash{\begin{tabular}[t]{l}$c$\end{tabular}}}}%
    \put(0.56105662,0.51860028){\color[rgb]{1,0,0}\makebox(0,0)[lt]{\lineheight{1.25}\smash{\begin{tabular}[t]{l}$a$\end{tabular}}}}%
    \put(0.0205625,0.08253209){\color[rgb]{0,0,1}\makebox(0,0)[lt]{\lineheight{1.25}\smash{\begin{tabular}[t]{l}$\iota_0^{\vee}x_1$\end{tabular}}}}%
    \put(0.70876307,0.73308411){\color[rgb]{1,0,0}\makebox(0,0)[lt]{\lineheight{1.25}\smash{\begin{tabular}[t]{l}$\dots$\end{tabular}}}}%
    \put(0.62283709,0.2726809){\color[rgb]{0,0,0}\makebox(0,0)[lt]{\lineheight{1.25}\smash{\begin{tabular}[t]{l}$x$\end{tabular}}}}%
  \end{picture}%
\endgroup%

  \caption{The case where $\gamma$ is a simple closed curve of slope $-n:$ the  resulting generator is not null-homologous}
  \label{fig:surg-C}
\end{figure}

\begin{remark}
  It is possible to prove analogous results for positive surgery using the alternative parametrization and the immersed curve in Figure \ref{fig:solid-torus-imm-N2}.
\end{remark}

Now we are ready to prove Theorem~\ref{thm:surgery}. We prove an equivalent statement.

\begin{theorem}
  Let $(Y,\xi)$ be a contact L-space with $\chat(\xi) \neq 0$ and $K$ a null-homologous Legendrian knot in $Y$. Denote by $\xi_{(n)}$ the contact structure obtained by a positive contact $(n)$-surgery on $K$ for $n \in \mathbb{N}$, in which all stabilizations are chosen to be negative. Let $s = n + \tb(K)$ be the corresponding smooth surgery coefficient.

  \begin{itemize}
    \item If $\tb(K) - \rot(K) < 2\tau_{\xi}(K) - 1$, then $\chat(\xi_{(n)}) = 0$,
    \item if $\tb(K) - \rot(K) = 2\tau_{\xi}(K) - 1$ and
    \begin{itemize}
      \item if $\epsilon_{\xi}(K) = 0,1$, then $\chat(\xi_{(n)})\neq 0$ if and only if $s \geq 2\tau_{\xi}(K)$,
      \item If $\epsilon_{\xi}(K) = -1$, then $\chat(\xi_{(n)})=0$.
    \end{itemize}
  \end{itemize}
\end{theorem}

\begin{proof}
  We choose a reduced basis corresponding to some ``pulled tight'' immersed curve. In this basis, we have $c_A=a_1+\dots+a_N.$ We first observe that the result of pairing with a positive framed unknot $a_j \boxtimes \iota_0^{\vee} \boxtimes x_1$ will be null-homologous for all but at most one generator $a_j$. In particular, by Lemma \ref{lem:surg-V-U}, this will be the case unless the curve immediately above $a_j$ turns leftward. Thus, we may as well assume $a_j$ is one such generator. We claim that $a_j$ must belong to the distingushed component $\gamma_{\xi},$ as otherwise, following $\Gamma$ downward from $a_j$ will result in a turn to the left, as all non-distinguished components in this Spin$^c$-grading are constrained to the midline. Hence, by Lemma \ref{lem:surg-V-D} if $a_j$ were on one of those non-distinguished components, it would again result in a null-homologous contribution. Moreover, by Lemma \ref{lem:surg-V-D} again, we see that not only must $a_j$ belong to $\gamma_{\xi},$ it must belong to the ``non-vertical'' part.\\
  Therefore, we now turn to the cases where $a_j$ belongs to the ``non-vertical'' segment of the immersed curve. Recall from Section \ref{sec:immersed} that with the standard meridian and Seifert longitude parametrization, the ``non-vertical'' segment has slope $\epsilon_{\xi}-2\tau_{\xi}$ (recall our invariant lives in the homology of the mirror knot). Since we are using a parametrization with the $\tb$ framing instead, this results in the segment having slope $\tb-2\tau_{\xi} + \epsilon_{\xi}.$ Notice that $\tb-2\tau_{\xi} + \epsilon_{\xi} \leq \tb-2\tau_{\xi} -1 \leq 0,$ with equality only possible when $\epsilon_{\xi}=1.$ Figure \ref{fig:nonVert} illustrates the two kinds of behaviors of the non-vertical segment when $\epsilon_{\xi} = 1$ or $-1.$ Let us put $m=\tb-2\tau_{\xi} + \epsilon_{\xi}.$ We consider the various cases:\\
  \\
  \textbf{Case I:} $\epsilon_{\xi} = -1$\\
  The first part of Lemma \ref{lem:surg-V-U} implies that only two generators might have a nonvanishing contribution to $c_A \boxtimes \iota^{\vee}_0 \boxtimes x_1;$ namely the two points labelled $a$ and $a'$ in Figure \ref{fig:eN}. From the Figure, we also see that the curve above $a'$, after turning left proceeds to turn upwards. Since it is at the ``leftmost'' end of the nonvertical segment, the curve must immediately enter the part that is constrained to the vertical parametrizing segment, and so must eventually (after some vertical passes) turn right again. Thus, Lemma \ref{lem:surg-V-U} implies that this generator would result in a null-homologous contribution. The other generator $a$ is on the ``rightmost'' of the nonvertical segment, and so following the curve downwards, it must become vertical and turn towards the left. Hence by Lemma \ref{lem:surg-V-D}, it must give a null-homologous contribution. Thus, we see that when $\epsilon_{\xi} = -1,$ all possible contributions vanish in homology, so that $c(\xi)$ itself is null-homologous.\\
  \\
  \textbf{Case II:} $\epsilon_{\xi} = 1$\\
  By Lemma \ref{lem:surg-V-U} and Figure \ref{fig:eP}, we see that there is only one potential generator on the non-vertical segment which may contribute; namely the one labelled as $a.$ So the vanishing of $c_A \boxtimes \iota_0^{\vee}\boxtimes x_1$ is equivalent to the vanishing of $a \boxtimes \iota_0^{\vee}\boxtimes x_1.$ From the figure, we see that there is a generator $L=m_2(a,\rho_2)$ in the idempotent corresponding to the meridian. This generator is at the ``top'' of the non-vertical segment and represents the generator of $\HFhat(Y,\mathfrak{s})\cong \HFK^-(-Y,K)/{U=1}.$ If $a$ is one of the $a_j$ generators included in $c_A,$ then $L$ is included in $m_2(c_A)=\widehat{\mathfrak{L}}.$ By Lemma \ref{lem:loss_gen}, this occurs if and only if $\tb-\rot = 2\tau_{\xi} -1.$ Thus, if $\tb-\rot < 2\tau_{\xi} -1,$ $c(\xi)$ is null-homologous. On the other hand, if equality is satisfied, we have from Lemma \ref{lem:surg-V-D} that if $-m \geq n,$ then $c(\xi)$ must be null-homologous. Note that $m=\tb-2\tau_{\xi} +1$ in this case, so that $c(\xi)$ vanishes when $\tb+n \leq  2\tau_{\xi}-1.$ If, however, $-m < n,$ then we see that Lemma \ref{lem:surg-NV-1} implies that $c(\xi)$ is not null-homologous.\\
  \textbf{Case III:} $\epsilon_{\xi} = 0$\\
  In this case, the non-vertical segment is a simple closed curve, which, in the Seifert framing is horizontal. Thus in our framing, it is an embedded curve of slope $\tb.$ The situation is similar to the $\epsilon =1$ case: there is once again a single generator $a$ which can contribute, and once again, $m_2(a,\rho_2)$ corresponds to the generator of $\HFhat(Y,\mathfrak{s})\cong \HFK^-(-Y,K)/{U=1}$ so that $a$ is included in $c_A$ exactly when $\tb-\rot=2\tau_{\xi}-1 = -1$ ($\epsilon_{\xi}=0$ forces $\tau_{\xi}$ to be zero as well). So once again, $c(\xi)$ will be null-homologous unless that equality is satisfied. In that case, notice that the line of slope $\tb$ has $-\tb-1$ vertical passes. Therefore, by Lemma \ref{lem:surg-V-D}, if $-\tb-1 \geq n$ (equivalently, $\tb+n \leq -1$), then $c(\xi)$ will again be null-homologous. If $-\tb-1 < n-1,$ then Lemma \ref{lem:surg-NV-1} implies that $c(\xi)$ is nonvanishing, whereas if $-\tb=n,$ then Lemma \ref{lem:surg-NV-2} also guarantees that $c(\xi)$ is nonvanishing. In summary, when $\epsilon = 0,$ we have that $c(\xi)$ is null-homologous if and only if $\tb+n <0.$
\end{proof}

\begin{figure}
\centering
  \begin{subfigure}{.45\textwidth}
  \centering
\begingroup%
  \makeatletter%
  \providecommand\color[2][]{%
    \errmessage{(Inkscape) Color is used for the text in Inkscape, but the package 'color.sty' is not loaded}%
    \renewcommand\color[2][]{}%
  }%
  \providecommand\transparent[1]{%
    \errmessage{(Inkscape) Transparency is used (non-zero) for the text in Inkscape, but the package 'transparent.sty' is not loaded}%
    \renewcommand\transparent[1]{}%
  }%
  \providecommand\rotatebox[2]{#2}%
  \newcommand*\fsize{\dimexpr\f@size pt\relax}%
  \newcommand*\lineheight[1]{\fontsize{\fsize}{#1\fsize}\selectfont}%
  \ifx\svgwidth\undefined%
    \setlength{\unitlength}{136.24685092bp}%
    \ifx\svgscale\undefined%
      \relax%
    \else%
      \setlength{\unitlength}{\unitlength * \real{\svgscale}}%
    \fi%
  \else%
    \setlength{\unitlength}{\svgwidth}%
  \fi%
  \global\let\svgwidth\undefined%
  \global\let\svgscale\undefined%
  \makeatother%
  \begin{picture}(1,0.96794677)%
    \lineheight{1}%
    \setlength\tabcolsep{0pt}%
    \put(0,0){\includegraphics[width=\unitlength,page=1]{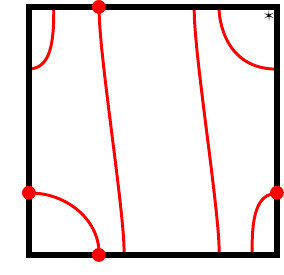}}%
    \put(0.51302349,0.50604254){\color[rgb]{1,0,0}\makebox(0,0)[lt]{\lineheight{1.25}\smash{\begin{tabular}[t]{l}$\dots$\end{tabular}}}}%
    \put(0.36559605,0.0045282){\color[rgb]{1,0,0}\makebox(0,0)[lt]{\lineheight{1.25}\smash{\begin{tabular}[t]{l}$a$\end{tabular}}}}%
    \put(-0.0019442,0.27261091){\color[rgb]{1,0,0}\makebox(0,0)[lt]{\lineheight{1.25}\smash{\begin{tabular}[t]{l}$L$\end{tabular}}}}%
  \end{picture}%
\endgroup%

    \caption{$\epsilon =1$}
    \label{fig:eP}
  \end{subfigure}\hfill
  \begin{subfigure}{.45\textwidth}
  \centering
\begingroup%
  \makeatletter%
  \providecommand\color[2][]{%
    \errmessage{(Inkscape) Color is used for the text in Inkscape, but the package 'color.sty' is not loaded}%
    \renewcommand\color[2][]{}%
  }%
  \providecommand\transparent[1]{%
    \errmessage{(Inkscape) Transparency is used (non-zero) for the text in Inkscape, but the package 'transparent.sty' is not loaded}%
    \renewcommand\transparent[1]{}%
  }%
  \providecommand\rotatebox[2]{#2}%
  \newcommand*\fsize{\dimexpr\f@size pt\relax}%
  \newcommand*\lineheight[1]{\fontsize{\fsize}{#1\fsize}\selectfont}%
  \ifx\svgwidth\undefined%
    \setlength{\unitlength}{121.02987815bp}%
    \ifx\svgscale\undefined%
      \relax%
    \else%
      \setlength{\unitlength}{\unitlength * \real{\svgscale}}%
    \fi%
  \else%
    \setlength{\unitlength}{\svgwidth}%
  \fi%
  \global\let\svgwidth\undefined%
  \global\let\svgscale\undefined%
  \makeatother%
  \begin{picture}(1,1.12683997)%
    \lineheight{1}%
    \setlength\tabcolsep{0pt}%
    \put(0,0){\includegraphics[width=\unitlength,page=1]{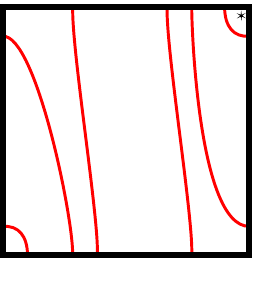}}%
    \put(0.47129232,0.60618894){\color[rgb]{1,0,0}\makebox(0,0)[lt]{\lineheight{1.25}\smash{\begin{tabular}[t]{l}$\dots$\end{tabular}}}}%
    \put(0,0){\includegraphics[width=\unitlength,page=2]{Imm_nonVert_eN.pdf}}%
    \put(0.31322008,0.01146428){\color[rgb]{1,0,0}\makebox(0,0)[lt]{\lineheight{1.25}\smash{\begin{tabular}[t]{l}$a'$\end{tabular}}}}%
    \put(0.12675017,0.00509758){\color[rgb]{1,0,0}\makebox(0,0)[lt]{\lineheight{1.25}\smash{\begin{tabular}[t]{l}$a$\end{tabular}}}}%
  \end{picture}%
\endgroup%

    \caption{$\epsilon =-1$}
    \label{fig:eN}
  \end{subfigure}
  \caption{The general behavior of the distinguished curve $\gamma_{\mathfrak{s}}$ in the meridian-$\tb$-framed longitude parametrization ($\epsilon \neq 0$)}
  \label{fig:nonVert}
\end{figure}

\subsection{Legendrian surgery on non-loose torus knots}\label{subsec:torusknots}
Here, we prove Theorem~\ref{thm:T34} and \ref{thm:lht}. Since the classification of strongly non-loose Legendrian representative of the left-handed trefoil is already known, we first consider Theorem~\ref{thm:lht}.

According to \cite[Theorem~1.12]{EMM:nonloose}, there are strongly non-loose Legendrian representatives $L^n_{\pm}$ with $\tb=n \in \mathbb{Z}$ in an overtwisted $S^3$. Let $K = L^n_-$ and consider the complement $(Y(K), \Gamma_n, \xi^-_n)$. By \cite[Lemma 6.2]{EMM:nonloose}, it can be decomposed into a negative basic slice and the knot complement with meridional sutures: 
\[
  (Y(K), \Gamma_n, \xi^-_n) = (Y(K),\Gamma_{\mu}, \xi_{in}) \cup B_-(n,\infty).
\]
Similarly, when $K = L^n_+$ we have 
\[
  (Y(K), \Gamma_n, \xi^+_n) = (Y(K),\Gamma_{\mu}, \xi_{in}) \cup B_+(n,\infty).
\]

\begin{figure}
  \centering
  \begin{subfigure}{.35\textwidth}
    \centering
    \footnotesize
\begingroup%
  \makeatletter%
  \providecommand\color[2][]{%
    \errmessage{(Inkscape) Color is used for the text in Inkscape, but the package 'color.sty' is not loaded}%
    \renewcommand\color[2][]{}%
  }%
  \providecommand\transparent[1]{%
    \errmessage{(Inkscape) Transparency is used (non-zero) for the text in Inkscape, but the package 'transparent.sty' is not loaded}%
    \renewcommand\transparent[1]{}%
  }%
  \providecommand\rotatebox[2]{#2}%
  \newcommand*\fsize{\dimexpr\f@size pt\relax}%
  \newcommand*\lineheight[1]{\fontsize{\fsize}{#1\fsize}\selectfont}%
  \ifx\svgwidth\undefined%
    \setlength{\unitlength}{104.39360371bp}%
    \ifx\svgscale\undefined%
      \relax%
    \else%
      \setlength{\unitlength}{\unitlength * \real{\svgscale}}%
    \fi%
  \else%
    \setlength{\unitlength}{\svgwidth}%
  \fi%
  \global\let\svgwidth\undefined%
  \global\let\svgscale\undefined%
  \makeatother%
  \begin{picture}(1,1.97049338)%
    \lineheight{1}%
    \setlength\tabcolsep{0pt}%
    \put(0,0){\includegraphics[width=\unitlength,page=1]{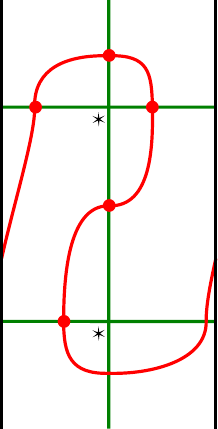}}%
    \put(0.5356353,0.94485009){\color[rgb]{1,0,0}\makebox(0,0)[lt]{\lineheight{1.25}\smash{\begin{tabular}[t]{l}$c_A(\xi_{in})$\end{tabular}}}}%
    \put(0.72656212,1.53405444){\color[rgb]{1,0,0}\makebox(0,0)[lt]{\lineheight{1.25}\smash{\begin{tabular}[t]{l}$c_A(\xi_{1})$\end{tabular}}}}%
    \put(0.52795126,1.77181913){\color[rgb]{1,0,0}\makebox(0,0)[lt]{\lineheight{1.25}\smash{\begin{tabular}[t]{l}$c_A(\xi_{12})$\end{tabular}}}}%
    \put(0.03581341,0.38630024){\color[rgb]{1,0,0}\makebox(0,0)[lt]{\lineheight{1.25}\smash{\begin{tabular}[t]{l}$c_A(\xi_{3})$\end{tabular}}}}%
    \put(0.17510072,1.53393362){\color[rgb]{1,0,0}\makebox(0,0)[lt]{\lineheight{1.25}\smash{\begin{tabular}[t]{l}$c_A(\xi_{123})$\end{tabular}}}}%
  \end{picture}%
\endgroup%

    \caption{The (lift) of the immersed curve invariant of the right-handed trefoil}
    \label{fig:LHT_immersed}
  \end{subfigure}
  \hfill
  \begin{subfigure}{.5\textwidth}
    \centering
\begingroup%
  \makeatletter%
  \providecommand\color[2][]{%
    \errmessage{(Inkscape) Color is used for the text in Inkscape, but the package 'color.sty' is not loaded}%
    \renewcommand\color[2][]{}%
  }%
  \providecommand\transparent[1]{%
    \errmessage{(Inkscape) Transparency is used (non-zero) for the text in Inkscape, but the package 'transparent.sty' is not loaded}%
    \renewcommand\transparent[1]{}%
  }%
  \providecommand\rotatebox[2]{#2}%
  \newcommand*\fsize{\dimexpr\f@size pt\relax}%
  \newcommand*\lineheight[1]{\fontsize{\fsize}{#1\fsize}\selectfont}%
  \ifx\svgwidth\undefined%
    \setlength{\unitlength}{194.54507663bp}%
    \ifx\svgscale\undefined%
      \relax%
    \else%
      \setlength{\unitlength}{\unitlength * \real{\svgscale}}%
    \fi%
  \else%
    \setlength{\unitlength}{\svgwidth}%
  \fi%
  \global\let\svgwidth\undefined%
  \global\let\svgscale\undefined%
  \makeatother%
  \begin{picture}(1,1.000456)%
    \lineheight{1}%
    \setlength\tabcolsep{0pt}%
    \put(0,0){\includegraphics[width=\unitlength,page=1]{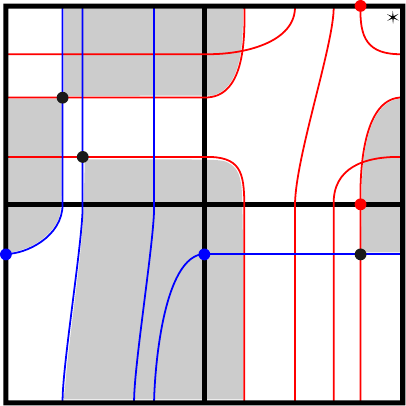}}%
    \put(0.90786664,0.52693709){\color[rgb]{1,0,0}\makebox(0,0)[lt]{\lineheight{1.25}\smash{\begin{tabular}[t]{l}$c_A$\end{tabular}}}}%
    \put(0.89481451,0.32565637){\color[rgb]{0.10196078,0.10196078,0.10196078}\makebox(0,0)[lt]{\lineheight{1.25}\smash{\begin{tabular}[t]{l}$c(\xi)$\end{tabular}}}}%
    \put(0.24573517,0.22354997){\color[rgb]{0,0,1}\makebox(0,0)[lt]{\lineheight{1.25}\smash{\begin{tabular}[t]{l}$\dots$\end{tabular}}}}%
    \put(0.0340999,0.31668913){\color[rgb]{0,0,1}\makebox(0,0)[lt]{\lineheight{1.25}\smash{\begin{tabular}[t]{l}$\iota_0^{\vee}x_1$\end{tabular}}}}%
    \put(0.08570451,0.77814851){\color[rgb]{0.10196078,0.10196078,0.10196078}\makebox(0,0)[lt]{\lineheight{1.25}\smash{\begin{tabular}[t]{l}$x$\end{tabular}}}}%
    \put(0.21718969,0.63913289){\color[rgb]{0.10196078,0.10196078,0.10196078}\makebox(0,0)[lt]{\lineheight{1.25}\smash{\begin{tabular}[t]{l}$y$\end{tabular}}}}%
  \end{picture}%
\endgroup%

    \caption{Computing the $c(\xi)$ for $n$-framed surgery on the left-handed trefoil, $n<0$}
    \label{fig:LHT_neg}
  \end{subfigure}
  \caption{The immersed curve invariant for the trefoil and a computation of the contact invariant of negative surgery}
  \label{fig:LHT}
\end{figure}

We see from Figure \ref{fig:LHT_immersed} that in the idempotent corresponding to the meridian, there are three generators, with distinct spin$^c$-gradings (corresponding to Alexander gradings $-1,0,$ and $1$). Since $\xi_{in}$ has a self-conjugate spin$^c$ structure, $c_A(\xi_{in})$ must be the middle generator (Alexander grading 0). From  Figure \ref{fig:LHT_immersed}, we also see that this generator admits both a nontrival $\rho_1$ and $\rho_3$ multiplication, corresponding to attaching positive and negative basic slices to $\xi_{in}.$ We denote these by $\xi_1$ and $\xi_3,$ respectively. Similarly, $\xi_1$ admits two more consecutive negative basic slice attachments (corresponding to $\rho_2$ and $\rho_3$ actions), and we denote the corresponding contact structures by $\xi_{12}$ and $\xi_{123}.$

\begin{remark}\label{rmk:LHT_rho13}
  It is known that attaching a basic slice $B_{\pm}(0,\infty)$ to $\xi_{in}$ results in a tight contact structure with nonvanishing $\EH$ invariant. This property also allows us to locate the contact class $c_A,$ since it is the unique generator with nontrivial $\rho_1$ and $\rho_3$ multiplications.
\end{remark}

\lht*

\begin{proof}
  It is sufficient to prove the statement for $K = L^n_{\pm}$ for sufficiently negative integer $n$ since Legendrian surgery on $L^{n+k}_{\pm}$ for $k\geq 0$ can be obtained by a sequence of negative contact surgeries on $L^{n}_{\pm}$. Also, Legendrian surgery on $L^n_{\pm}$ is equivalent to a contact Dehn filling on $(Y(K),\Gamma_{0},\xi^{\pm}_{0})$ using an $n-1$ framed solid torus consisting of only negative (\emph{resp.} positive) basic slices. According to Section~\ref{subsec:solid}, the contact structures on the solid torus correspond to the elements $x_1$ and $x_{n-1}$ in $\CFA$.   
  
  We first consider $L^{n}_-$ and $(Y(K),\Gamma_{0},\xi^{-}_{0})$. Notice that $\xi^-_0$ corresponds to $\xi_3$ and the contact structure on the solid torus corresponds to $x_1$. The computation is shown in Figure \ref{fig:LHT_neg}. We see that there are generators $x$ and $y$ such that $\partial x = c(\xi)+ y$ where $\xi$ is the result of Legendrian surgery. There are no other bigons involving $x,$ $y,$ or $c(\xi),$ and so $c(\xi)$ survives in homology. Therefore, $\xi$ is tight. 

  A similar argument also works for $L^+_n$.
\end{proof}

\begin{remark} 
  We may obtain equivalent contact structures by ``peeling off'' a basic slice from the knot complement and attaching it to the solid torus. Similarly, we can replace the negative basic slice with a positive one, and re-attach it to the knot exterior. The following computations show that indeed we obtain homologous contact invariants (cf. Remark \ref{rmk:reattaching}):
  \begin{align*}
    c(\xi) &= c_A(\xi_1') \boxtimes \iota_0^{\vee}\boxtimes x_1 \\
    &= m_2(c_A(\xi_{in},\rho_3)\boxtimes \iota_0^{\vee}\boxtimes x_1 \\
    &\sim c_A(\xi_{in}) \boxtimes \iota_1^{\vee} m_2(x_1,\rho_3)\\
    &= c_A(\xi_{in}) \boxtimes \iota_1^{\vee} y\\
    &=c_A(\xi_{in}) \boxtimes \iota_1^{\vee} m_2(x_n,\rho_1)\\
    &\sim m_2(c_A(\xi_{in},\rho_1)\boxtimes \iota_0^{\vee}\boxtimes x_n \\
    &= c_A(\xi_1)\boxtimes \iota_0^{\vee}\boxtimes x_n
  \end{align*}
\end{remark}

Next, we consider the torus knot $T_{3,4}$. We first classify non-loose Legendrian representatives of non-loose torus knot $T_{3,4}$ by applying the algorithm in \cite[Section 3]{EMM:nonloose} to the decorated Farey graph representing $T_{3,4}$ shown in Figure~\ref{fig:T34_path}.

\begin{figure}[htbp]{\scriptsize
  \begin{overpic}[tics=20]{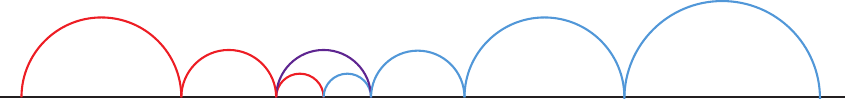}
    \put(6, -10){$\infty$}
    \put(85, -10){$0$}
    \put(132, -10){$1$}
    
    \put(150, -10){$4/3$}
    
    \put(175, -10){$3/2$}
    \put(220, -10){$2$}
    \put(293, -10){$\infty$}
    \put(392, -10){$0$}
    
    \put(45, 10){\LARGE$\circ$}
    \put(107, 5){$\pm$}
    \put(141, 3){\scriptsize$\pm$}
    \put(164, 3){\scriptsize$\pm$}
    \put(198, 5){\scriptsize$\pm$}
    \put(258, 8){\large$\pm$}
    \put(344, 10){\LARGE$\circ$}
  \end{overpic}}
  \vspace{0.1cm}
  \caption{The decorated path for non-loose torus knots $T_{3,4}$.}
  \label{fig:T34_path}
\end{figure}

\begin{proposition}\label{prop:34classification}
  The $(3,4)$-torus knot has non-loose Legendrian representative in $\xi^1$, $\xi^{-1}$, $\xi^{0}$ $\xi^{-5}$. Here, $\xi^i$ denotes the overtwisted contact structure in $S^3$ with $d_3(\xi^i) = i$. The classification in each of these contact structure is as follows. 
  \begin{enumerate}
    \item In $(S^3,\xi^1)$, there exist non-loose $L^i_{\pm}$ for $i > 7$ and $L^7$ such that 
    \begin{alignat*}{3}
      &\tb(L^i_{\pm}) = i, \quad &&\tb(L^7) = 7\\ 
      &\rot(L^i_{\pm}) = \pm(7-i), \quad &&\rot(L^7)=0. 
    \end{alignat*}
    They satisfy $S_{\pm}(L^i_{\pm})=L^{i-1}_{\pm}$ for $i>8$ and $S_{\pm}(L^8_{\pm}) = L^7$, and $S_{\mp}(L^i_{\pm})$ and $S_{\pm}(L^7)$ are loose. In addtion, there are Legendrian knots $K^{12}_{\pm}$ with 
    \[
      \tb(K^{12}_{\pm}) = 12, \quad \rot(K^{12}_{\pm})=\pm1.
    \]
    When they stabilize to have the same invariants, they become equivalent and they are non-loose until stabilized outside the $\mathsf{V}$ defined by $L^{i}_{\pm}$ and $L^{12}$. None of these Legendrian knots have half Giroux torsion in their complement. See Figure~\ref{fig:T34_V}.
       
    \begin{figure}[htbp]{\scriptsize
      \vspace{0.2cm}
      \begin{overpic}[tics=20]{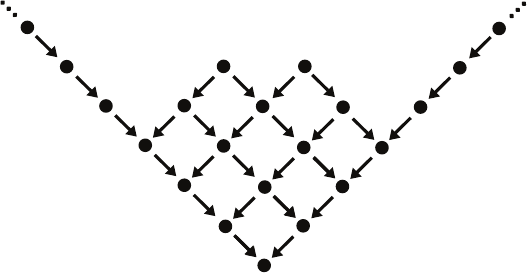}
        \put(-10, 96){$12$}
        \put(-10, 77){$11$}
        \put(-10, 58){$10$}
        \put(-10, 39){$9$}
        \put(-10, 19){$8$}
        \put(-10, 1){$7$}
        
        \put(125, 130){$0$}
        \put(145, 130){$1$}
        \put(99, 130){$-1$}
      \end{overpic}}
      \caption{The mountain range of non-loose $T_{3,4}$ in $\xi^1$}
      \label{fig:T34_V}
    \end{figure}
            
    \item In $(S^3,\xi^{-1})$, there are non-loose $L^i_{\pm}$ for $i \in \mathbb{Z}$ such that
    \begin{alignat*}{3}
      \tb(L^i_{\pm})=i,\quad \rot(L^i_{\pm}) = \pm(1-i) 
    \end{alignat*}
    They satisfy $S_{\pm}(L^i_{\pm})=L^{i-1}_{\pm}$ and $S_{\mp}(L^i_{\pm})$ are loose. None of these Legendrian knots have half Giroux torsion in their complement. See Figure~\ref{fig:T34_X} for the mountain range.
       
    \item In $\xi^{-5}$, there are $L^{i,k}_{\pm}$ and in $\xi_0$ there are  $L^{i,k+1/2}_{\pm}$ for $i\in \mathbb{Z}$ and  $k\geq 0$ such that $\tb(L^{i,k}_{\pm}) = \tb(L^{i,k+1/2}_{\pm}) = i$ and 
    \begin{align*}
        \rot(L^{i,k}_{\pm}) = \pm(-5-i),\quad \rot(L^{i,k+1/2}_{\pm}) = \pm(5-i)        
    \end{align*}
    They satisfy $S_{\pm}(L^{i,k}_{\pm})=L^{i-1,k}_{\pm}$ and $S_{\mp}(L^{i,k}_{\pm})$ are loose. A similar statement holds for $L^{i,k+1/2}_{\pm}$. Also, $L^{i,k}_{\pm}$ has $k$-torsion in its complement if $i > 5$ and $k+1/2$-torsion if $i\leq 5$. $L^{i,k+1/2}_{\pm}$ has $k+1/2$-torsion in its complement if $i > 5$ and $k+1$-torsion if $i \leq 5$. See Figure~\ref{fig:T34_X} for the mountain range.
  \end{enumerate}
\end{proposition}

\begin{figure}
  \centering
  \begin{subfigure}{.47\textwidth}
    \centering
    \includegraphics{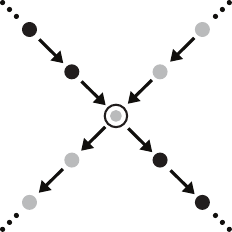}
    \caption{Each dot represents a unique Legendrian representative.}
    \label{fig:T34_X1}
  \end{subfigure}\hfill
  \begin{subfigure}{.47\textwidth}
    \centering    
    \includegraphics{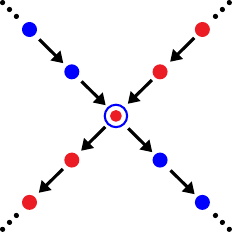}
    \caption{Each dot represents an infinite family of Legendrian representatives.}
    \label{fig:T34_X2}
  \end{subfigure}
  \caption{The mountain ranges of non-loose $T_{3,4}$ in $\xi^{-1}$ (left), $\xi^{0}$ and $\xi^5$ (right).}
  \label{fig:T34_X}
\end{figure}

Now we are ready to prove Theorem~\ref{thm:T34}.

\ptorus*

\begin{proof}
  Since we only consider strongly non-loose $T_{3,4}$, we can exclude Legendrian knots in $\xi^0$ and knots with $\tb < 7$ in $\xi^1$ and $\xi^{-5}$ by Proposition~\ref{prop:34classification}. Lemma~7.1 and 7.4 in \cite{EMM:nonloose} implies that Legendrian surgery on the strongly non-loose knots in $\xi^1$ and $\xi^{-5}$ with $\tb \geq 7$ is equivalent to a positive contact surgery on Legendrian $T_{3,4}$ with $\tb=5$ in $(S^3,\xi_{std})$, and we know the resulting contact manifolds are tight by Theorem~\ref{thm:surgery}. Thus it remains to consider strongly non-loose knots in $\xi^{-1}$. Like the left-handed trefoil case, \cite[Lemma~6.1]{EMM:nonloose} implies that the knot complement is decomposed into $\xi_{in}$ and $B_{\pm}(n,\infty)$ and it is sufficient to conisder a Dehn filling on $\xi^{\pm}_{0} = \xi_{in} \cup B_{\pm}(0,\infty)$ with sufficiently negative integer surgery coefficients. Since $\xi_{in}$ has a self-conjugate spin$^c$ structure, $c_A(\xi_{in})$ must be the middle generator (Alexander grading 0). From  Figure \ref{fig:T34_neg.a}, we also see that this generator admits both a nontrival $\rho_1$ and $\rho_3$ multiplication, corresponding to attaching positive and negative basic slices to $\xi_{in}.$ We denote these by $\xi_1 (=\xi^+_0)$ and $\xi_3 (= \xi^-_0),$ respectively.
    (We could also deduce the generator in a similar way to Remark \ref{rmk:LHT_rho13}, although we would need a slightly different parametrization).
  \begin{figure}
    \centering
    \begin{subfigure}{.3\textwidth}
      \centering
      \scriptsize
\begingroup%
  \makeatletter%
  \providecommand\color[2][]{%
    \errmessage{(Inkscape) Color is used for the text in Inkscape, but the package 'color.sty' is not loaded}%
    \renewcommand\color[2][]{}%
  }%
  \providecommand\transparent[1]{%
    \errmessage{(Inkscape) Transparency is used (non-zero) for the text in Inkscape, but the package 'transparent.sty' is not loaded}%
    \renewcommand\transparent[1]{}%
  }%
  \providecommand\rotatebox[2]{#2}%
  \newcommand*\fsize{\dimexpr\f@size pt\relax}%
  \newcommand*\lineheight[1]{\fontsize{\fsize}{#1\fsize}\selectfont}%
  \ifx\svgwidth\undefined%
    \setlength{\unitlength}{97.43399192bp}%
    \ifx\svgscale\undefined%
      \relax%
    \else%
      \setlength{\unitlength}{\unitlength * \real{\svgscale}}%
    \fi%
  \else%
    \setlength{\unitlength}{\svgwidth}%
  \fi%
  \global\let\svgwidth\undefined%
  \global\let\svgscale\undefined%
  \makeatother%
  \begin{picture}(1,2.98634102)%
    \lineheight{1}%
    \setlength\tabcolsep{0pt}%
    \put(0,0){\includegraphics[width=\unitlength,page=1]{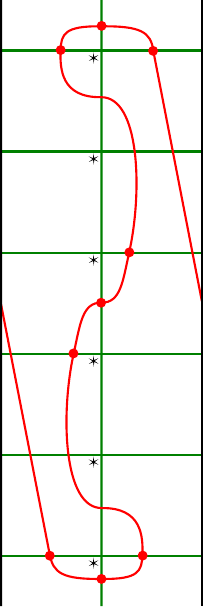}}%
    \put(0.52803225,1.41977259){\color[rgb]{1,0,0}\makebox(0,0)[lt]{\lineheight{1.25}\smash{\begin{tabular}[t]{l}$c_A(\xi_{in})$\end{tabular}}}}%
    \put(0.67311073,1.78599973){\color[rgb]{1,0,0}\makebox(0,0)[lt]{\lineheight{1.25}\smash{\begin{tabular}[t]{l}$c_A(\xi_{1})$\end{tabular}}}}%
    \put(0.11033871,1.29474508){\color[rgb]{1,0,0}\makebox(0,0)[lt]{\lineheight{1.25}\smash{\begin{tabular}[t]{l}$c_A(\xi_{3})$\end{tabular}}}}%
  \end{picture}%
\endgroup%

      \caption{The (lift of) the immersed curve invariant for $T_{-3,4}$}
      \label{fig:T34_neg.a}
    \end{subfigure}
    \hfill
    \begin{subfigure}{.6\textwidth}
      \centering
\begingroup%
  \makeatletter%
  \providecommand\color[2][]{%
    \errmessage{(Inkscape) Color is used for the text in Inkscape, but the package 'color.sty' is not loaded}%
    \renewcommand\color[2][]{}%
  }%
  \providecommand\transparent[1]{%
    \errmessage{(Inkscape) Transparency is used (non-zero) for the text in Inkscape, but the package 'transparent.sty' is not loaded}%
    \renewcommand\transparent[1]{}%
  }%
  \providecommand\rotatebox[2]{#2}%
  \newcommand*\fsize{\dimexpr\f@size pt\relax}%
  \newcommand*\lineheight[1]{\fontsize{\fsize}{#1\fsize}\selectfont}%
  \ifx\svgwidth\undefined%
    \setlength{\unitlength}{262.25390337bp}%
    \ifx\svgscale\undefined%
      \relax%
    \else%
      \setlength{\unitlength}{\unitlength * \real{\svgscale}}%
    \fi%
  \else%
    \setlength{\unitlength}{\svgwidth}%
  \fi%
  \global\let\svgwidth\undefined%
  \global\let\svgscale\undefined%
  \makeatother%
  \begin{picture}(1,1.00045611)%
    \lineheight{1}%
    \setlength\tabcolsep{0pt}%
    \put(0,0){\includegraphics[width=\unitlength,page=1]{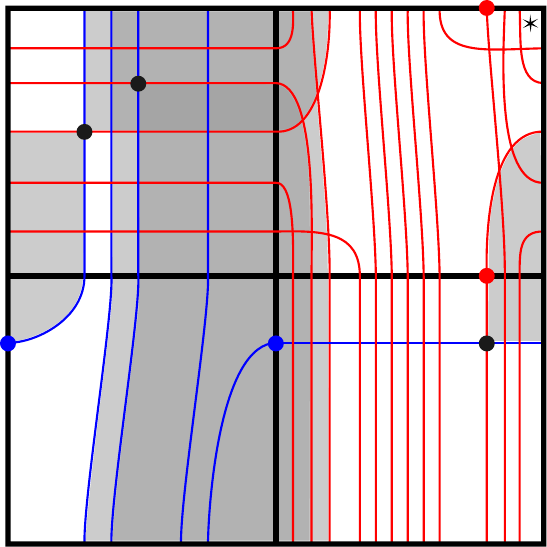}}%
    \put(0.83343104,0.515387){\color[rgb]{1,0,0}\makebox(0,0)[lt]{\lineheight{1.25}\smash{\begin{tabular}[t]{l}$c_A$\end{tabular}}}}%
    \put(0.81862209,0.32789101){\color[rgb]{0.10196078,0.10196078,0.10196078}\makebox(0,0)[lt]{\lineheight{1.25}\smash{\begin{tabular}[t]{l}$c(\xi)$\end{tabular}}}}%
    \put(0.26153613,0.22241866){\color[rgb]{0,0,1}\makebox(0,0)[lt]{\lineheight{1.25}\smash{\begin{tabular}[t]{l}$\dots$\end{tabular}}}}%
    \put(0.03414957,0.31569346){\color[rgb]{0,0,1}\makebox(0,0)[lt]{\lineheight{1.25}\smash{\begin{tabular}[t]{l}$\iota_0^{\vee}x_1$\end{tabular}}}}%
    \put(0.1167154,0.77267721){\color[rgb]{0.10196078,0.10196078,0.10196078}\makebox(0,0)[lt]{\lineheight{1.25}\smash{\begin{tabular}[t]{l}$x$\end{tabular}}}}%
    \put(0.21952041,0.87123281){\color[rgb]{0.10196078,0.10196078,0.10196078}\makebox(0,0)[lt]{\lineheight{1.25}\smash{\begin{tabular}[t]{l}$y$\end{tabular}}}}%
  \end{picture}%
\endgroup%

      \caption{Computing $c(\xi_n)$ for $n$-framed surgery on $T_{3,4}$ when $n<-2$}
      \label{fig:imm_T34}
    \end{subfigure}
    \caption{The immersed curve invariant for the torus knot $T_{3,4}$ and the computation of sufficiently negative contact surgery on it}
    \label{fig:T34_neg}
  \end{figure}
    
  The situation is quite similar to the trefoil case discussed above. Figure \ref{fig:imm_T34} shows the computation when $n<-2$ (the cases $n=-1,-2$ are similar, but they also follow from the general case by Legendrian surgery). We see that there are generators $x$ and $y$ such that $\partial x = y+c(\xi).$ There are no other differentials involving these three generators, and so we conclude that $c(\xi)$ is not null-homologous, so that $\xi$ is tight.
\end{proof}

\bibliography{references}
\bibliographystyle{plain}
\end{document}